\documentclass[11pt,letterpaper]{amsart}

\usepackage{amsmath,amsthm,amsfonts,amssymb,bm,graphicx, mathrsfs}

\usepackage{comment}
\usepackage[normalem]{ulem}
\usepackage%[backref=page]
{hyperref}

\hypersetup{colorlinks=true,citecolor=blue,linkcolor=blue,urlcolor=blue,
pdfstartview=FitH }

\theoremstyle{plain}
\newtheorem{lemma}{Lemma}
\newtheorem{theorem}{Theorem}

\newtheorem{prop}[lemma]{Proposition}
\numberwithin{equation}{section}

\theoremstyle{definition}

\newtheorem*{Acknowledgements}{Acknowledgement}

\theoremstyle{remark}

\renewcommand{\geq}{\geqslant}
\renewcommand{\leq}{\leqslant}

\DeclareMathOperator{\GL}{GL}
\DeclareMathOperator{\SL}{SL}
\DeclareMathOperator{\sgn}{sgn}
\DeclareMathOperator{\res}{res}
\DeclareMathOperator{\arctanh}{arctanh}
\DeclareMathOperator{\arcsinh}{arcsinh}
\DeclareMathOperator{\arccosh}{arccosh}
\DeclareMathOperator{\Ad}{Ad}
\DeclareMathOperator{\cond}{cond}

\newcommand{\eps}{\varepsilon}
\def\PGL{\operatorname{PGL}}

\newcommand{\changed}[1]{{\color{black} #1}}

\usepackage{calc}
\newsavebox\CBox
\newcommand\hcancel[2][0.5pt]{%
  \changed{\ifmmode\sbox\CBox{$#2$}\else\sbox\CBox{#2}\fi%
  \makebox[0pt][l]{\usebox\CBox}%  
  \rule[0.5\ht\CBox-#1/2]{\wd\CBox}{#1}}}

\makeatletter
\DeclareRobustCommand\widecheck[1]{{\mathpalette\@widecheck{#1}}}
\def\@widecheck#1#2{%
    \setbox\z@\hbox{\m@th$#1#2$}%
    \setbox\tw@\hbox{\m@th$#1%
       \widehat{%
          \vrule\@width\z@\@height\ht\z@
          \vrule\@height\z@\@width\wd\z@}$}%
    \dp\tw@-\ht\z@
    \@tempdima\ht\z@ \advance\@tempdima2\ht\tw@ \divide\@tempdima\thr@@
    \setbox\tw@\hbox{%
       \raise\@tempdima\hbox{\scalebox{1}[-1]{\lower\@tempdima\box
\tw@}}}%
    {\ooalign{\box\tw@ \cr \box\z@}}}
\makeatother

\newcommand{\Z}{\mathbb{Z}}

\newcommand{\Q}{\mathbb{Q}}
\newcommand{\R}{\mathbb{R}}
\newcommand{\C}{\mathbb{C}}

\newcommand{\Ic}{\mathcal{I}}

\def\cond{\operatorname{cond}}

\renewcommand{\Bbb}{\mathbb}

\usepackage[usenames,dvipsnames]{color}

\usepackage{enumerate}

\begin{document}

\author{Valentin Blomer}
\author{Subhajit Jana}
\author{Paul D. Nelson}
 
\address{Mathematisches Institut, Endenicher Allee 60, 53115 Bonn, Germany}
\email{blomer@math.uni-bonn.de}

 \address{Max-Planck-Institut f\"ur Mathematik, Vivatsgasse 7, 53111 Bonn, Germany}
\email{subhajit@mpim-bonn.mpg.de}

\address{Institute for Advanced Study, 1 Einstein Dr, Princeton, NJ 08540}
\email{nelson.paul.david@gmail.com}
  
\title{The Weyl bound for triple product $L$-functions}

\thanks{The first author was supported in part by the DFG-SNF lead agency program grant BL 915/2-2 as well as  Germany's Excellence Strategy - EXC-2047/1 - 390685813. }

\begin{abstract} Let $\pi_1, \pi_2, \pi_3$ be three  cuspidal  automorphic representations for the group ${\rm SL}(2, \Bbb{Z})$, where $\pi_1$ and $\pi_2$ are fixed and $\pi_3$ has large analytic conductor. We prove a subconvex bound for $L(1/2, \pi_1 \otimes \pi_2 \otimes \pi_3)$ of Weyl-type quality. Allowing $\pi_3$  to be an Eisenstein series we also obtain a Weyl-type subconvex bound for $L(1/2 + it, \pi_1 \otimes \pi_2)$. %\ll \text{cond}(h)^{2/3 + \eps} $.   
 
\end{abstract}


\setcounter{tocdepth}{2}  \maketitle 

\maketitle

\section{Introduction} 

\subsection{Weyl-type subconvexity}  Subconvexity estimates belong to the core topics in the theory of $L$-functions and are one of the most challenging testing grounds for the strength of existing technology. If $C$ denotes the analytic conductor of the relevant $L$-function (restricted to the parameters of interest), then the Phragm\'en--Lindel\"of principle gives the bound $C^{1/4+\eps}$ for the $L$-function on the central line $\Re s = 1/2$. In the most favourable cases, one can obtain an upper bound $C^{1/6+\eps}$, which we refer to as a Weyl-type subconvex bound.  For instance, a classical result states that the Riemann zeta function satisfies the bound
\begin{equation*} 
  \zeta(1/2 + it) \ll_{\eps} (1+|t|)^{1/6+\eps}
\end{equation*}
on the critical line. Based on work of Weyl \cite{We},  it was proved first by Hardy and Littlewood (cf.\ \cite{Li}), and first written down by Landau \cite{La} in a slightly refined form and generalized to all Dirichlet $L$-functions. Results of similar strength  exist in the  archimedean aspect for automorphic $L$-functions of degree 2, starting with the work of Good \cite{Go} and culminating in the hybrid bound of Jutila and Motohashi \cite{JM}. 
We also have a Weyl-type bound in degree 4 in some limited cases pertaining to Rankin-Selberg $L$-functions, such as \cite{JM2, LLY}
$$L(1/2, f \otimes g) \ll_{g, \eps} C(f)^{1/3 +\eps}$$
for two cusp forms $f, g$ for ${\rm SL}_2(\Bbb{Z})$, where
$C(f)$ denotes the conductor of $f$ as defined in \S\ref{223}
below.
%\pn{edited}
Although
slightly better bounds are available
for ${\rm GL}(1)$ 
\cite{Bo, Mi},
% in the special situation of the Riemann zeta function slightly
% better bounds are available,
the Weyl exponent marks a natural
barrier that has never been improved, and rarely been reached,
beyond ${\rm GL}(1)$. 
We note that for some applications (see
\cite{GS}, \cite{Ma}), the essential input is a Weyl-type
subconvex bound (or something approaching it), rather than
merely any nontrivial subconvex bound.\footnote{
  Indeed, to show that the number of zeros on $i[1, \infty)$ of
  a holomorphic Hecke eigenform $f$ of weight $k$ tends to
  infinity as $k\rightarrow\infty$, the proof in \cite{GS} needs
  $L(1/2 + it, f) \ll k^{0.335}$ with polynomial dependence in
  $t$.
}
We note also that Petrow
and Young recently established Weyl-type subconvex bounds
for ${\rm GL}(1)$ in the level aspect, improving spectacularly upon the
decades-old Burgess-type bounds (see \cite{PY1, PY2, Nel2}).

A celebrated result of  Bernstein and Reznikov  established for
the first time
subconvex bounds for certain $L$-functions of degree 8. For two fixed spherical cuspidal automorphic representations $\pi_1, \pi_2$ (i.e.\ generated by   Maa{\ss} forms   for ${\rm SL}_2(\Bbb{Z})$)  and another spherical cuspidal automorphic representation $\pi_3$   of large analytic conductor  $C(\pi_3)$ %(as defined in \S\ref{223} below) 
they proved \cite{BR1, BR2}
\begin{equation}\label{BR}
\sum_{T \leq C(\pi_3)^{1/2} \leq T+T^{1/3}}L(1/2, \pi_1 \otimes \pi_2 \otimes \pi_3) \ll _{\pi_1, \pi_2, \eps} T^{5/3+\eps},
\end{equation}
which implies (by non-negativity of the central value) in
particular $L(1/2,  \pi_1 \otimes \pi_2 \otimes \pi_3) \ll
_{\pi_1, \pi_2, \eps} C(\pi_3)^{5/6+\eps}$. The proof employs a
beautiful combination of
representation theory,
invariant norms and asymptotic analysis of oscillatory Airy-type
integrals.
In fact, their result is really an estimate for triple product
periods,
as $L$-functions enter only through the period formula of Watson--Ichino \cite{Ic}. 

We observe, however, that \eqref{BR} is not optimal. The  Lindel\"of hypothesis   suggests the Weyl-type bound with an exponent $4/3$ instead of $5/3$.  That such a result might be within reach   was indicated by Suvitie \cite{Su1}. For a fixed holomorphic cusp form $F$ of weight $k$ and a Maa{\ss} form $h$  
she showed 
\begin{equation}\label{showed}
  \sum_{T \leq C(h)^{1/2} \leq T+T^{1/3}} e^{\pi r_h} |\langle y^k |F|^2, h \rangle|^2 \ll_{F, \eps} T^{2k - \frac{2}{3} + \eps}
  \end{equation}
which via the formula of Watson--Ichino translates into
$$\sum_{T \leq C(h)^{1/2} \leq T+T^{1/3}}  L(1/2, F \otimes F \otimes h) \ll_{F, \eps}  T^{4/3 + \eps}.$$
This $L$-function is not primitive,
as it factorizes into a degree 6 and a degree 2 $L$-function, but the same argument would work for $F G$ instead of $|F|^2$ in \eqref{showed}. More seriously, however, the proof starts by replacing $F$ with a holomorphic Poincar\'e series and unfolding the inner product, a route that is  not available in general. In fact, an attempt to generalize this to Maa{\ss} forms remained incomplete \cite{Su2} and seems not to work. In particular, the work of Bernstein--Reznikov remained unimproved. 

\medskip

In this paper we establish the Weyl-type bound for triple product  $L$-functions in a uniform fashion for all combinations of local types at infinity, i.e., any of the three factors can be holomorphic or Maa{\ss}. As mentioned before, the Weyl-type bound marks the natural limit  of all present day approaches to subconvexity.
The key novelty in our paper is the method. We combine in a substantial way 
representation theory, local harmonic analysis,  and analytic number theory to establish a robust method for the subconvexity problem for triple product $L$-functions.  
\begin{theorem}\label{thm1} Let $\pi_1, \pi_2$ be  two  fixed cuspidal automorphic representations for the group ${\rm SL}_2(\Bbb{Z})$. Let $\pi_3$ run over cuspidal automorphic representations for ${\rm SL}_2(\Bbb{Z})$ with conductor satisfying $T \leq C(\pi_3)^{1/2} \leq T+T^{1/3}$. 
 Then
$$\sum_{T \leq C(\pi_3)^{1/2} \leq T+T^{1/3}} L(1/2, \pi_1 \otimes \pi_2 \otimes \pi_3) \ll_{\pi_1, \pi_2, \eps} T^{4/3 + \eps},$$
in particular
$$L(1/2, \pi_1 \otimes \pi_2 \otimes \pi_3) \ll_{\pi_1, \pi_2, \eps} C(\pi_3)^{2/3 + \eps}$$
for every $\eps > 0$.  
\end{theorem}

An inspection of the proof shows that the dependence on the analytic conductors of $\pi_1, \pi_2$ is polynomial. Under Watson's formula \cite{Wat},
the latter estimate translates
to bounds
for triple product
integrals of Maa{\ss} forms
$\varphi_j$
of eigenvalue $1/4+t_j^2$
($j=1,2,3$):
\[
  \Big\lvert
    \int_{\SL_2(\mathbb{Z}) \backslash \mathbb{H}}
    \varphi_1 \varphi_2 \varphi_3
  \Big\rvert^2
  \ll_{\varphi_1, \varphi_2,\eps}
  |t_3|^{-2/3+\eps} \exp(- \pi |t_3|),
\]
giving a further improvement beyond that in \cite{BR2} on the
general exponential decay bounds of \cite{Sar94}.

We can allow  $\pi_3$ to be an Eisenstein series, and our proof yields as a by-product a Weyl-type bound for Rankin-Selberg $L$-functions.
 \begin{theorem}\label{thm2} Let $\pi_1, \pi_2$ be  two  fixed  cuspidal automorphic representations for the group ${\rm SL}_2(\Bbb{Z})$. Then
$$\int_{T}^{T+T^{1/3}} |L(1/2 + it, \pi_1 \otimes \pi_2 )|^2 dt  \ll_{\pi_1, \pi_2, \eps} T^{4/3 + \eps},$$
in particular
$$L(1/2+it , \pi_1 \otimes \pi_2 ) \ll_{\pi_1, \pi_2, \eps}(1 + |t|)^{2/3 + \eps}$$
for every $\eps > 0$ and $t \in \Bbb{R}$. 
\end{theorem}

For the rest of the paper, all implied constants may depend on $\eps$, and we suppress it in subsequent formul\ae. The weaker bound 
$L(1/2+it , \pi_1 \otimes \pi_2 ) \ll_{\pi_1, \pi_2}(1 + |t|)^{5/6 + \eps}$ is implicit in \cite[Remarks 7.2.2.2]{BR2} and was the best known result until now. By a method purely based on analytic number theory,  the bound $L(1/2+it , \pi_1 \otimes \pi_2 ) \ll_{\pi_1, \pi_2}(1 + |t|)^{15/16 + \eps}$ was recently shown  in \cite{ASS}. For bounds of triple product $L$-functions in the level aspect see \cite{Ve, Hu}. 

\subsection{Remarks}   ${}$

1) Our results feature   ``pure'' exponents of Weyl-type quality that are independent of bounds towards the Ramanujan conjecture or the Selberg eigenvalue conjecture. The proof uses at one place that the smallest non-zero Laplace eigenvalue is larger than $3/16$. 
 \medskip

2)  
In principle,
the proof produces an asymptotic  formula. If $\psi$ is a
sufficiently regular test function with ``essential support''
in $[T, T+H]$,
e.g.\ $\psi(t) = \exp(-(t - T)^2 H^{-2})$, and
$t_{\pi_3} \asymp \sqrt{C(\pi_3)}$
(cf.\ \eqref{maass}) denotes the spectral parameter of $\pi_3$, then with the same notation and under the same assumptions as in Theorem \ref{thm1}, one can relate
\begin{equation}\label{spec}
\begin{split}
& \sum_{\pi_3} \psi(t_{\pi_3}) \frac{L(1/2, \pi_1 \otimes \pi_2 \otimes \pi_3)}{L(1, \Ad^2 \pi_3)} + \int_{t\in\Bbb{R}} \psi(t) \frac{|L(1/2 + it, \pi_1 \otimes \pi_2 )|}{|\zeta(1 + 2it)|^2} \frac{dt}{2\pi} \\
\end{split}
\end{equation}
%(where $t_{\pi_3} \asymp \text{cond}(\pi_3)^{1/2}$ is the spectral parameter of $\pi_3$)
 to 
\begin{displaymath}
\begin{split}
c   L(1, \Ad^2 \pi_1) 
L(1, \Ad^2 \pi_2) TH+ O(T^{3/2+\eps}H^{-1/2})
\end{split}
\end{displaymath}
for a suitable constant $c$ (depending on $\psi$).\medskip

3) The   proof of Theorem \ref{thm1} has the shape of a
\emph{reciprocity formula} as for instance in \cite{BLM}.  A
spectral sum in a window $[T, T+H]$ as in \eqref{spec} is
ultimately transformed into a spectral sum of similar shape with
spectral parameter up to  $\ll T/H$ (see  \eqref{dualspectral},
\eqref{z}, \eqref{C}). This is analogous to the discussion after
(1.11) in \cite{BLM}, and a new instance of a reciprocity
phenomenon.  The optimal choice is $H = T^{1/3}$,
in which case both spectral sums have length $T^{4/3}$. This yields   the Weyl bound, and we see that the Weyl bound is indeed the natural limit from the point of view of spectral analysis. An abstract version of the underlying reciprocity formula is displayed in \eqref{eqn:spectral-identity-L-functions} below which features central $L$-values as well as their ``canonical square roots''. \medskip

4) Implicit in the proof of Theorem \ref{thm1} is an alternative description of the central triple product $L$-value $L(1/2, \pi_1 \otimes  \pi_2\otimes \pi_3)$ in terms of a certain shifted convolution problem very roughly of the shape 
\begin{multline}\label{alternative}
L(1/2, \pi_1\otimes \pi_2 \otimes \pi_3) \\ \approx \left| \frac{1}{t^{1/2}_{\pi_3}} \sum_{\nu \asymp 1} \sum_{m \ll t_{\pi_3}^2} \frac{\lambda_{\pi_3}(m) \lambda_{\pi_2}(m+\nu)\lambda_{\pi_1}(\nu)}{m^{1/4}} \exp\left( \pm 2i t_{\pi_3}\sqrt{\frac{ \nu}{m}}\right)  \right|^2.
\end{multline}
The $\approx$ sign has to be interpreted in a broad sense
(see \S\ref{final} for details). In the generic range $m \asymp t_{\pi_3}^2$ the oscillatory factor is flat (see \S\ref{weight} for definition of flatness).  
% a more precise treatment can be found in the proof. 

\subsection{Comparison with Bernstein--Reznikov and Michel--Venkatesh}
% \label{sec:}
It is instructive
to compare
our approach for studying the moment
\eqref{BR}
with those of
Bernstein--Reznikov
\cite{BR1, BR2}
and Michel--Venkatesh \cite{MV}.

To begin, we briefly sketch the approach to the subconvexity problem introduced by
Bernstein--Reznikov.
We borrow some presentation features from
Michel--Venkatesh, see especially
\cite[\S1.1.1, \S1.1.3]{MV}.
The starting point of this approach is the triple product
formula \cite{Ic}:
for unit vectors $v_j \in \pi_j$ we have 
\begin{multline}\label{period}
\frac{L(1/2, \pi_1\otimes \pi_2 \otimes \pi_3)}{L(1, \Ad^2 \pi_1)L(1, \Ad^2 \pi_2)
L(1, \Ad^2 \pi_3)} \mathcal{L}_{\infty}(v_1, v_2, v_3)\\ = \Big|\int_{g\in {\rm SL}_2(\Bbb{Z}) \backslash {\rm SL}_2(\Bbb{R})} v_1 v_2v_3(g)\,  dg\Big|^2
\end{multline}
for a suitable local factor
$\mathcal{L}_{\infty}(v_1, v_2,v_3)$ (a constant
multiple of a matrix coefficient integral).

For unit vectors  $v_1 \in \pi_1$ and $v_2 \in \pi_2$ we consider the inner product identity
\begin{equation}\label{eqn:obvious-inner-product-identity}
  \langle v_1 v_2, v_1 v_2 \rangle
  =
  \int_{g\in \SL_2(\mathbb{Z})
    \backslash \SL_2(\mathbb{R})
  }
  |v_1 v_2|^2(g) \, d g
  =
  \langle |v_1|^2, |v_2|^2 \rangle.
\end{equation}
%\sj{Is it $|v_2|^2$ above?}
By expanding each of these inner products
over the spectrum of $L^2(\SL_2(\mathbb{Z}) \backslash
\SL_2(\mathbb{R}))$
and applying \eqref{period},
we obtain a spectral identity
of families of $L$-functions, roughly of the shape 
\begin{multline}\label{eqn:spectral-identity-L-functions}
  \sum_{\pi_3}
  h(\pi_3)
  L(1/2,\pi_1 \otimes \pi_2 \otimes \pi_3)\\
  \approx
  1 + 
  \sum_{\sigma}
  \tilde{h}(\sigma)
  \sqrt{L(1/2,\pi_1 \otimes \pi_1 \otimes \sigma)
    L(1/2,\pi_2 \otimes \pi_2 \otimes \sigma)}.
\end{multline}
Here $\pi_3$ and $\sigma$ run over cuspidal automorphic
representations of $\SL_2(\mathbb{Z})$, the square roots of $L$-values are ``canonical square roots'' (in the sense \cite[\S 1.1.3]{MV}) and 
the meaning of ``$\approx$''
is that
\begin{itemize}
\item we have suppressed adjoint $L$-factors
and other proportionality constants, and
\item we have elided the contribution of the continuous spectrum
  and all degenerate terms except for the ``expected main term'' $1$,
  which arises from the inner product
  $\langle |v_1|^2, 1 \rangle \langle 1, |v_2|^2 \rangle = 1$
  (up to volume factors).
\end{itemize}
The weight functions $h$ and $\tilde{h}$ depend upon the choice
of vectors $v_1$ and $v_2$.

The weights $h(\pi_3)$ and the $L$-values
$L(1/2, \pi_1 \otimes \pi_2 \otimes \pi_3)$
are known to be nonnegative,
so if we can bound the right hand side of
\eqref{eqn:spectral-identity-L-functions}
by $O(1)$ (the natural limit, in view of the expected main
term),
then we deduce by dropping all but
one term the estimate
\begin{equation}\label{eqn:positivity-estimate-h-pi-1}
  L(1/2,\pi_1 \otimes \pi_2 \otimes \pi_3)
  \ll
  1/h(\pi_3).
\end{equation}

Given some $\pi_3$ with $\cond(\pi_3)^{1/2} \asymp T$, we
now
face the optimization problem of choosing unit vectors $v_1$
and $v_2$ for which $h(\pi_3)$ is as large as possible, so that
the bound \eqref{eqn:positivity-estimate-h-pi-1} is as strong as
possible.
Bernstein--Reznikov
\cite[(2.6.3), Prop 9.1]{BR2}
showed (for 
$\pi_j$ spherical) that one may choose $v_1$ and $v_2$ so
that $h(\pi_3)$ is roughly $T^{-5/3}$.
This choice
and
a suitable bound for the global period
eventually yields their estimate \eqref{BR}.
Michel--Venkatesh
\cite[\S3.6-3.7]{MV}
(for $\pi_1$ tempered and spherical) employed a simpler
choice of vectors for which $h(\pi_3)$ is
of size $T^{-2}$; for this
choice, the estimate \eqref{eqn:positivity-estimate-h-pi-1} only
recovers the convexity bound, but Michel--Venkatesh managed to
apply the amplification method
to save a further small power of
$T$ 
%\pn{added}
(in a more general ``all aspects'' setting).

To approach the Weyl bound
$L(1/2, \pi_1 \otimes \pi_2 \otimes \pi_3) \ll T^{4/3+\eps}$ using
\eqref{eqn:spectral-identity-L-functions} would seem to require
producing $v_1$ and $v_2$ for which $h(\pi_3)$ is at least
$T^{-4/3}$, but the analysis of Bernstein--Reznikov strongly
suggests that their
lower bound $T^{-5/3}$ is best possible.
To obtain a stronger lower bound
thus requires a more flexible class
of weight functions $h(\pi_3)$.
Such a class
may be obtained from the generalization of
\eqref{eqn:obvious-inner-product-identity}
to higher-rank tensors
$\sum_{j} v_{1,j} \otimes v_{2,j} \in \pi_1 \otimes \pi_2$,
namely
\begin{equation*}%\label{eqn:obvious-inner-product-identity}
  \sum_{j,k}
  \langle v_{1,j} v_{2,j}, v_{1,k} v_{2,k} \rangle
  =
  \sum_{j,k}
  \langle v_{1,j} \overline{v_{1,k}},
  v_{2,k} \overline{v_{2,j}}
  \rangle.
\end{equation*}
Such tensors yield
more flexible forms of \eqref{eqn:spectral-identity-L-functions}.
One could hope to find a tensor
$\sum_{j} v_{1,j} \otimes v_{2,j}$ for which the corresponding
weight $h(\pi_3)$ localizes on $\pi_3$ satisfying
$T \leq \cond(\pi_3)^{1/2} \leq T+T^{1/3}$ and for which
the right hand side of the corresponding spectral identity as in
\eqref{eqn:spectral-identity-L-functions} may be effectively
bounded.  To implement this idea in practice would require a
careful study of the spaces of test functions $\{ h \}$ and
$\{\tilde{h}\}$ as well as the transform relating them.
Unfortunately, 
such a study has not yet been carried out,
and does not seem
straightforward
in the generality of
Theorem \ref{thm1}
(which, we should emphasize,
imposes no local conditions
on the representations $\pi_j$).

The method of this paper consists of two stages.
We first use
a somewhat crude choice of $v_1$ and $v_2$
(like in the
work of Michel--Venkatesh)
and
an unfolding technique (see \S\ref{sec:unfolding}, in particular \eqref{zag})
to express the $L$-values of interest
as bilinear forms in the Hecke eigenvalues of the varying form
$\pi_3$.
We then average over the spectral window
$T \leq C(\pi_3)^{1/2} \leq T+T^{1/3}$ of interest by means of analytic number theory, and in particular 
 the Kuznetsov formula.

Our approach has in common with
the works
of Bernstein--Reznikov and Michel--Venkatesh
that we make use of the well-developed
theory of integral representations of $L$-functions
to produce and analyze our
test vectors.
In this respect, our basic framework
owes much to those works.
The essential difference
is that we analyze sums
over much narrower spectral windows,
and the more technical difference is that
we implement this analysis
using the Kuznetsov formula
rather than the spectral theory of triple product periods.

The main advantage of our
approach is that we
can make full use of available technology
related to the  Kuznetsov formula
(abundance of test
functions, explicit integral transforms, Bessel function
asymptotics, large sieve estimates $\dotsc$),
whose avatars are not available at the level
of the triple product periods.
To the best of our knowledge,
this paper is the first to employ
such a combination
of
the theory of
integral representations and analytic number theory. 
We hope that this methodology will be useful
more broadly.

\subsection{Analytic number theory}\label{analytic} Having discussed the representation theoretic ideas in the previous subsection, we now give a brief sketch
of
how analytic number theory can handle expressions like \eqref{alternative} that can be extracted from the triple product formula.  This is a precursor to the analysis in \S \ref{shifted}. 
%We now put analytic number theory.  As remarked in \S\ref{alternative}, our argument will eventually recover this representation of the central value as a shifted convolution problem in full generality. 
 The trivial bound in the $m$-sum recovers convexity, and if we
 had square-root cancellation in the $m$-sum we would obtain
 Lindel\"of.  We now consider the 
 sum
\begin{equation}\label{av}
\sum_{2\pi T \leq r_h \leq 2\pi (T+H)} L(1/2,   f \otimes g \otimes h).
\end{equation}
An application of the  Kuznetsov and Voronoi formul\ae \, gives a dual shifted convolution problem that can be treated by a $\delta$-symbol method. 
With a final application of the spectral large sieve,
the sum \eqref{av}
can be shown to be $\ll (TH)^{1+\eps}$ for $H = T^{1/3}$,
which establishes the Weyl bound.  For convenience, we describe a
toy version of this argument, restricting each parameter
to the generic range. This is somewhat misleading because
smaller ranges of $m$ in \eqref{alternative} are punished by an
additional oscillation which complicates matters considerably,
but it nevertheless gives a flavour for
what is happening. Restricting \eqref{alternative} to $\nu = 1$ (for simplicity) and applying the Kuznetsov formula, we obtain an expression roughly of the  shape 
\begin{multline*}
\frac{H^{3/2}}{T^{5/2}} \sum_{m_1, m_2 \asymp T^2} \sum_{c \asymp T/H} \lambda_f(m_1)\lambda_f(m_2) S(m_1-1, m_2 - 1, c)\\ \times  e\left( \frac{2\sqrt{m_1m_2}}{c} - \frac{T^2 c}{  \sqrt{m_1m_2}}\right),
\end{multline*}
cf.\ \eqref{off1},  provided that $H \geq T^{1/3}$. (For smaller
$H$,
more terms in the exponential would be necessary.) 
The complicated exponential is reminiscent of the uniform asymptotic expansion of the $J$-Bessel function at imaginary index. The leading term of the Kuznetsov kernel equals the Voronoi kernel, a feature that is now crucially exploited: applying the Voronoi summation formula to $m_2$, the dual variable will be close to $m_1$ and a large portion of the oscillation disappears. We obtain roughly
$$\frac{H^{5/2}}{T^{5/2}} \sum_{c \asymp T/H} \sum_{h \asymp T^2/H^2} S(h-1, -1, c) \sum_{m \asymp T^2} \lambda_f(m) \lambda_f(m+h) e\Big(\frac{2Th^{1/2}}{  m^{1/2}} \Big),$$
cf.\ \eqref{newshift}. The inner sum is now a shifted
convolution problem with a moderately oscillatory factor of size
$T(h/m)^{1/2} \asymp T/H$. It can be spectrally decomposed by a
delta-symbol method (cf.\ \S\ref{sec56}).
Another application of the Voronoi and Kuznetsov formul\ae \, (cf.\ \S\ref{sec57} and \S\ref{sec58})   leads to a spectral sum of length $T/H$ having $(T/H)^2$ terms by Weyl's law. Note that this is the length of the $h$-sum, so the large sieve (which is itself an application of the Kuznetsov formula) can show its full power on the $h$-sum, leading to the desired final bound. As an aside, we see that this analysis employs the Kuznetsov formula three times, in various directions. 

We finally remark that a direct approach to \eqref{av}, by an approximate functional equation followed by the Kuznetsov formula, appears to be hopeless: we would obtain sums over $\lambda_{f}(n)\lambda_{g}(n)$ with $n \asymp T^4$ against oscillatory functions of (combined arithmetic and analytic) conductor of size $T^2$ (regardless of the choice of $H$), so that a ${\rm GL}(2) \times {\rm GL}(2)$ Voronoi summation formula would not reduce the length of summation. In other words, after applying the Kuznetsov formula, we run out of moves immediately.  %\vb{this paragraph is new} \pn{good}

\subsection{Unfolding} \label{sec:unfolding}
Following \cite{Su1}, we now sketch a beautiful, but completely different approach to
the formula \eqref{alternative}, specific to the case of discrete series representations $\pi_1, \pi_2$. 
%we  sketch an argument based on holomorphic Poincar\'e series. 
Suppose that $\pi_1, \pi_2$ are generated by holomorphic forms
$f, g$ and $\pi_3$ is generated by a Maa{\ss} form $h$  of
spectral parameter $t_h \geq 0$.  By Watson's formula, we have
\begin{equation*}%\label{wat}
L(1/2,   f \otimes g \otimes h) \approx e^{\pi t_h} t_h^{2-2k} \Big|\int_{z\in {\rm SL}_2(\Bbb{Z}) \backslash \Bbb{H}}  \bar{f}(z) g(z)h(z) y^k d\mu(z)\Big|^2.
\end{equation*}
We write $g$ as a linear combination of Poincar\'e series, and without much loss of generality we assume that 
\begin{equation*}%\label{po}
 g(z) = P_n(z) = \sum_{\gamma = (\begin{smallmatrix} a & b\\ c & d\end{smallmatrix})\in \Gamma_{\infty} \backslash {\rm SL}_2(\Bbb{Z})} (cz + d)^{-k} e(n\gamma z)
 \end{equation*}
is the $n$-th Poincar\'e series. We insert the Fourier expansions and unfold. This gives a  $y$-integral
\begin{displaymath}
\begin{split}
  \int_0^{\infty}y^k&  \sqrt{y} {\cosh\Big(\frac{\pi t}{2}\Big)}K_{it}(2 \pi m y) e^{-2\pi(m+n)y} \frac{dy}{y^2} \\
&\approx \frac{{e^{-\pi t/2}} t^{2k-2} }{m^{k-1/2}} \exp\left( \pm 2i t\sqrt{\frac{ n}{m}}\right) \min\Big(\frac{m^{k/2 - 1/4}}{t^{k - 1/2}}, 1\Big)
\end{split}
\end{displaymath}
for large $t$ and fixed $n$. This integral was first analyzed by Good \cite[Section 4]{Go} in terms of hypergeometric $_2F_1$ functions; the analysis is long and difficult.  At least if the weight $k$ is fixed but relatively large (for small $k$, one needs to work a little harder), the typical range is $m \ll t^2$, and we obtain that $L(1/2,   f \otimes g \otimes h)$ is essentially the absolute-square of a linear combination of 
\begin{equation}\label{jut}
%L(1/2,   \pi_1 \otimes \pi_2 \otimes \pi_3) \approx\Big| 
\frac{1}{t_h^{1/2}}\sum_{m \ll t_h^2} \frac{\lambda_h(m) \lambda_f(m+n)}{m^{1/4}}  \exp\left( \pm 2it_h \sqrt{\frac{ n}{m}}\right) % \Big|^2
\end{equation}
for a fixed number of $n$'s.   It is very interesting to note that  this resembles closely \eqref{alternative}. The previous argument is due to \cite{Su1}. 

The unfolding step is obviously not applicable in the set-up of Bernstein and Reznikov. Following an idea of Zagier \cite{Za}, the formal identity
\begin{equation}\label{zag}
\begin{split}
\int_{z\in {\rm SL}_2(\Bbb{Z})\backslash \Bbb{H}} f(z) d\mu(z) &= \frac{\pi}{3} \int_{z\in{\rm SL}_2(\Bbb{Z})\backslash \Bbb{H}} f(z) \underset{s=1}{\res} E(z, s) d\mu(z) \\
&= \frac{\pi}{3} \oint_s \int_{z\in\Gamma_{\infty}\backslash \Bbb{H}} f(z)  y^s d\mu(z) \frac{ds}{2\pi i} 
\end{split}
\end{equation}
for an ${\rm SL}_2(\Bbb{Z})$-invariant function $f$ can be used to mimic unfolding in more general situations (see \cite[Appendix A]{Ho} for a related idea). Applied to the triple product of three classical Maa{\ss} forms it yields a $y$-integral  involving three $K$-Bessel functions
\begin{equation*}%\label{tripleint}
\int_{0}^{\infty} K_{it_1}(m_1 y) K_{it_2}(m_2 y) K_{it_3}(m_3 y) y^{s-1/2} dy
\end{equation*}
for $m_1+m_2+m_3 = 0$.  This can still be analyzed to some
extent, but the resulting highly oscillatory shifted convolution
problems become  untreatable with the required precision. This
is the reason why the attempt on the Maa{\ss} case in   \cite{Su2} remained
incomplete.
Nevertheless, the unfolding step \eqref{zag} is also present in
our argument  (cf.\ \S\ref{final}).
It acts as a hinge between the triple product identity and sums
over Fourier coefficients,
leading eventually to the description \eqref{alternative}
for the central $L$-value in terms of shifted convolution sums.

 \subsection{The analytic test vector problem}\label{sec13}
The art in using the triple product formula \eqref{period} to
bound $L$-functions consists of choosing appropriate test
vectors.  The traditional test vector problem asks for explicit
$v_1,v_2,v_3$ for which the local factor
$\mathcal{L}_\infty(v_1, v_2,v_3)$ is nonzero.  Spherical
vectors are often test vectors in this sense, but are usually
\emph{not} the best choice due to the exponential decay of the
local factor.  For analytic applications, it is useful to work
with \emph{analytic} test vectors: vectors for which the local
factor is not merely nonvanishing, but enjoys (informally
speaking) a reasonable quantitative lower bound.
Michel--Venkatesh \cite[\S3.6.1]{MV} gave a robust supply of
test vectors under local assumptions relevant for
Rankin--Selberg subconvexity.  We will revisit and extend their
approach to the triple product setting, removing all local
assumptions in a uniform way.

For our analysis of test vectors, we adopt the language of
\emph{analytic newvectors} \cite{JN}, which is well-suited for
keeping track of the essential invariance properties.
Analytic newvectors are approximate archimedean analogues of the
classical $p$-adic newvectors introduced by Casselman \cite{C}
% for $\GL_2(\Q_p)$
(see also
% and generalized by
% Jacquet--Piatetski-Shapiro--Shalika
\cite{JPSS}).
% to
% $\GL_n(\Q_p)$).
Let $K_0(p^N)$ denote the standard congruence
subgroup consisting of matrices in $\PGL_2(\Z_p)$ whose
lower-left entry is divisible by $p^N$.  Let $\xi$ be a generic
irreducible representation of $\PGL_2(\Q_p)$.  Denote by
$c(\xi)$ the conductor exponent of $\xi$, so that $p^{c(\xi)}$
is the usual arithmetic conductor of $\xi$.
The main result of local newvector theory
\cite{C}
is that there is a unique (up to scalar) nonzero vector
$v\in\xi$ such that $\xi(g)v=v$ for all $g\in K_0(p^{c(\xi)})$.
Such vectors $v$ are called \emph{newvectors}.

An archimedean analogue of the family
of congruence subgroups
$K_0(p^N)\subseteq \PGL_2(\Z_p)$
is
the family of subsets $K_0(X,\tau)$ that is defined by the image in $\PGL_2(\R)$ of the set
$$
\left\{\begin{pmatrix}
    a&b\\c&d
  \end{pmatrix}\in\mathrm{GL}_{2}(\R)
  : % \middle|
  \begin{aligned}
    &|a - 1| < \tau,
    \quad
    |b|<\tau, \\
    &|c|<\frac{\tau}{X},
    \quad
    |d-1|<\tau
  \end{aligned}
\right\}. $$%/\R^{\otimes}.$$
Here $X$ is a large
positive parameter,
thought of as tending off to infinity,
while $\tau \in (0,1)$ is taken small but fixed.
An archimedean analogue
of local newvector theory
is given by
\cite[Theorem 1]{JN}:
%\todo{Is the convention that we always use ``le''
 % rather than ``leq''?
 % If so, we should do a mass search-and-replace before
%  submitting.
%}
for each fixed $0 \leq \vartheta < 1/2$
and arbitrary $\delta>0$,
there is a $\tau>0$ so
that for every
generic irreducible unitary $\vartheta$-tempered (see \S\ref{sec:gener-bounds-whitt}) 
%\vb{This should be defined} \pn{added forward reference} 
representation $\pi$ of $\PGL_2(\R)$,
% (realized in its Whittaker
% model with respect to some fixed
% character),
there is a unit vector $v\in\pi$ such that 
$$\|\pi(g)v-v\|<\delta
\quad
\text{for all $g\in K_0(C(\pi),\tau)$.}$$
We refer to such
vectors as \emph{analytic newvectors}
(suppressing,
for terminological brevity,
the dependence of this notion upon the parameters $\delta$ and
$\tau$).
Such vectors $v$ may be constructed explicitly
as fixed bump functions
in the Kirillov model
% by specifying
% as those whose image $W_v$ in the Whittaker model of $\pi$
% is a fixed bump function
(see
\cite[Theorem 7]{JN}).
% in which case
% \[W_v(1)\asymp 1, \quad |W_v(g)-W_v(1)|<\delta \text{ for } g\in K_0(C(\pi),\tau).\]

%\pn{added}
Inspired by the construction
of \cite[\S3.6.1]{MV},
we approach
the analytic
test vector problem
for the local triple product
periods $L_\infty(v_1,v_2,v_3)$
by choosing $v_1$ and $v_3$ to be analytic newvectors.
The choice
of $v_2$ is simplest
to describe
when $\pi_2$
is a principal series
representation.
In that case, we
describe $v_2$
in the induced
model by
a function on the lower-triangular
subgroup
supported within $O(1/X)$
of the identity. 
 More generally, we make use of the fact that
$\pi_2$ may be embedded
in a (not necessarily unitary) principal series representation.

\subsection{Plan for the paper}
Having chosen test vectors $v_1, v_2, v_3$
as indicated above, we need to solve
three main problems:
\begin{itemize}
\item
  We need to compute (a lower bound for)
  the matrix coefficient integral
  $\mathcal{L}_{\infty}(v_1, v_2, v_3)$. This will be done in \S
  \ref{local}.
  The idea of the proof, as in \cite[\S3.7.2]{MV},
  is to write the matrix coefficient integral
  as the square of a Rankin--Selberg integral
  and then to estimate the latter
  by playing the support
  properties of $v_1$ in its Kirillov model
  and $v_2$ in its induced model
  against the invariance properties of $v_3$.
  One subtlety
  is that we have not assumed that any
  of our representations belongs to the principal series.
  For this reason, 
  the reduction to Rankin--Selberg integrals
  is achieved in general only after embedding $\pi_2$
  into a principal series representation
  and using the standard intertwining operator to normalize
  its inner product.

 \item We use (a refined version of)  the formula \eqref{zag} to
  compute the right hand side of \eqref{period}. This leads to
  an integral of three archimedean Whittaker functions that will
  be computed asymptotically in \S \ref{whittaker}.
 % \pn{tweaked:}
  The proof involves several applications
  of the local functional equation and stationary phase
  analysis,
  but no input concerning special functions beyond Stirling's formula.
    This yields the expression \eqref{alternative}. In other words, choosing test vectors as above has the exact same effect as using Poincar\'e series in the holomorphic case and unfolding, cf.\ \eqref{jut}.  This is a rather remarkable feature.
%This will be done in \S \ref{whittaker}.  
\item We need to bound the shifted convolution problem
  \eqref{alternative}. This can be done by analytic number theory,
  roughly as indicated in \S\ref{analytic}. This is the content of
  \S \ref{shifted}. 
  It is here that we implement the ``hard analysis''
  required by our short spectral summation.
\end{itemize}

Theorems \ref{thm1} and \ref{thm2} are then an easy consequence
and will be derived in \S \ref{final}.

\begin{Acknowledgements}
  The second author thanks ETH Z\"urich and Max Planck Institute for Mathematics, where parts of this article were worked out, for providing a perfect research atmosphere.  This paper was revised while the third author was at the Institute for Advanced Study, supported by the National Science Foundation under Grant No. DMS-1926686.  All authors would like to thank the referees for a careful reading of the manuscript.
\end{Acknowledgements}

\section{Preliminaries}

\subsection{Basic notation}\label{sec:basic-notation} Throughout we work with the group $G := {\rm PGL}_2(\mathbb{R})$, and its subgroup $N$ of unipotent upper triangular matrices which are equipped with the usual Haar measures. 
We use the notation
 \begin{displaymath}
 \begin{split}
 & n(x) := \begin{pmatrix} 1 & x \\ 0 & 1 \end{pmatrix}, \, \, \,
 n'(x) := \begin{pmatrix} 1 & 0 \\ x & 1 \end{pmatrix}, \, \, \,
 a(y) := \begin{pmatrix} y & 0 \\ 0 & 1 \end{pmatrix}, \\
 & 
 k(\theta) := \begin{pmatrix} \cos \theta & - \sin \theta \\
   \sin \theta & \cos \theta \end{pmatrix}, \, \, \,
 w := \begin{pmatrix} 0 & -1 \\ 1 & 0 \end{pmatrix}.
  \end{split}
\end{displaymath}
We view $k(\theta)$ as a function of
$\theta \in \mathbb{R} / \pi \mathbb{Z}$.  For convenience, we
may assume $\theta$ taken in the interval $[-\pi/2, \pi/2]$.

We write  $d^\times y = dy/|y|$ for the Haar measure on $\mathbb{R}^\times$ and $dx$ for the Lebesgue measure on $\R$. We fix a $G$-invariant measure $dg$ on $N\backslash G$ given in Iwahori coordinates by
$$N\backslash G \ni g=a(y)n'(x),\quad dg=\frac{d^\times y}{|y|} dx,$$
and in Iwasawa coordinates  by
$$N\backslash G \ni g=a(y)k(\theta),\quad dg = \frac{d^\times
  y}{|y|}d\theta.$$
We equip $\GL_2(\mathbb{R})$ with the Haar measure
compatible with the chosen Haar measures on $\mathbb{R}^\times$
and $G$
via the short exact
sequence
$1 \rightarrow \mathbb{R}^\times \rightarrow \GL_2(\mathbb{R})
\rightarrow G \rightarrow 1$.

Let $\mathfrak{X}(\mathbb{R}^\times)$ denote the character group
of $\mathbb{R}^\times$.
Each $\chi \in \mathfrak{X}(\mathbb{R}^\times)$
is uniquely of the form
$\chi = |.|^s \sgn^{a}$ for some $s \in \mathbb{C}$
and $a \in \{0, 1\}$.
We set $\Re(\chi) := \Re(s)$,
$\Im(\chi):=\Im(s)$,
and $C(\chi) := (1 + |\Im s|)/(2\pi)$, cf.\ \eqref{crho}. %\vb{updated/added}
The group $\mathfrak{X}(\mathbb{R}^\times)$ is a complex
manifold
with respect to the coordinate charts $\chi \mapsto s$.
For a function $f : \mathfrak{X}(\mathbb{R}^\times) \rightarrow
\mathbb{C}$
of sufficient decay and $\sigma \in \mathbb{R}$,
we define the contour integral
\begin{equation*}%\label{eqn:}
  \int_{\Re(\chi) = \sigma}
  f(\chi) \, d \chi
  :=
  \frac{1}{2} \sum_{a \in \{0,1\} }
  \int_{\Re(s) = \sigma}
  f( |.|^s \sgn^{a}) \, \frac{ds}{2\pi i}.   %\vb{\text{what is $d\chi$?  maybe $\frac{ds}{2\pi i}$?}}
\end{equation*}
For a smooth function $f : \mathbb{R}^\times \rightarrow
\mathbb{C}$
of sufficient decay, we then have the Mellin inversion formula
\begin{equation}\label{eqn:mellin-inversion}
  f(y) = \int_{\Re(\chi) = \sigma}
  \chi(y) \left( \int_{t \in \mathbb{R}^\times}
  f(t) \chi^{-1}(t) \, \frac{d t}{|t|} \right) \, d \chi.
\end{equation}

\subsection{Local $\gamma$-factors,
  Stirling's formula,
  and the analytic conductor}
\label{localTate}
Let $\rho$ be a
finite-dimensional
representation of
the
Weil group $W_\mathbb{R}$.
Let
$\psi(x) := e^{2 \pi i x}$
be the standard additive character
of $\mathbb{R}$.
The local $\gamma$-factor
of $\rho$ 
is defined as usual
by
\begin{equation}\label{localgamma}
\gamma(s,\rho) :=
\gamma(s,\rho,\psi):=\epsilon(s,\rho,\psi)\frac{L(1-s,\tilde{\rho})}{L(s,\rho)},
\end{equation}
where $\epsilon$ and $L$ denote the
$\epsilon$-factor and  $L$-factor, respectively, a
description of which can be found in \cite[\S3.1]{Wat} for the
cases relevant in this paper.
We regard $\psi$ as fixed
once and for all,
and for this reason
we drop it from the notation.

The analytic conductor $C(\rho)$ has been defined in
various slightly different ways
(see for instance \cite[\S 2]{IS},  %{MR1826269}
%\todo{insert bibliography entry for Iwaniec--Sarnak ``Perspectives''},
\cite[\S5]{IK}
\cite[\S3.1.8]{MV}).
For many applications,
it is unimportant precisely which definition
is used:
what matters is just that
$C(\rho)$
controls the local $\gamma$-factor
in the sense that for small enough $s$,
and
under favorable conditions,
one has
at least the rough approximation
\begin{equation}\label{eqn:vague-gamma-vs-condcturo}
  \gamma(s,\rho) \approx C(\rho)^{1/2-s}.
\end{equation}
For the purposes of this paper,
it will be convenient
to normalize the definition
of $C(\rho)$
somewhat more precisely,
so that a correspondingly
more precise form of
\eqref{eqn:vague-gamma-vs-condcturo}
holds. While we could work with ad hoc definitions, it is useful
to present this in a slightly more general context. {The purpose
  of the following computation is
  to give a uniform asymptotic formula for the local gamma
  factors in the cases relevant for our application. 
  This is achieved in \eqref{maass} and \eqref{hol} below and used in \S  \ref{sec:estimates-f_1sharp}.}

 \subsubsection{Stirling's formula}
% We recall Stirling's formula.
With the principal branch of the logarithm,
we have
$$\Gamma(z) = z^{-1/2} \Big(\frac{z}{e}\Big)^z
\Big(\mathcal{G}_N(z) + O_{N, \eps}(|z|^{-N})\Big),  \quad |\arg(z)| \leq \pi - \eps, \quad |z| \geq \eps$$
for some smooth function $\mathcal{G}_N$ satisfying 
$$|z|^j  \frac{d^j}{dz^j} \mathcal{G}_N(z) \ll_{j, N} 1$$
for all $N, j\in \Bbb{Z}_{\geq 0}$.
 
 \subsubsection{Characters of $\mathbb{R}^\times$}
\label{sec:local-gamma-characters-reals}
Set $\Gamma_\R(s):=\pi^{-s/2}\Gamma(s/2)$.  The basic archimedean local $\gamma$-factors
over $\mathbb{R}$
are given
(with respect
to the standard character $\psi$ of $\mathbb{R}$, as above)
by
\begin{equation*}%\label{eqn:real-gamma}
  \gamma(s,\sgn^a)
  =
  i^a
  \frac{\Gamma_\mathbb{R}(1 - s + a)}{
    \Gamma_\mathbb{R}(s + a)}
  \quad
  (s \in \mathbb{C}, a \in \{0,1\}),
\end{equation*}
corresponding
to the characters $|.|^s \sgn^a$
of $\mathbb{R}^\times$.
For $s = \sigma + i \tau$,
we define
$g_{\sigma,a}(\tau)$
by writing
\begin{equation*}%\label{eqn:}
  \gamma(s,\sgn^a)
  =
  \left( \frac{|\tau|}{2 \pi e} \right)^{1/2-s}
  g_{\sigma,a}(\tau).
\end{equation*}
The factor $g_{\sigma,a}(\tau)$
is ``mild''
in the sense that
whenever $\sigma$ is restricted to a fixed interval and 
$\min_{n
  \in \Bbb{N}}|s-n| \geq \eps$
for some fixed $\eps > 0$, we have 
 \begin{equation}\label{eqn:stir1}
  \partial_\tau^j g_{\sigma,a}(\tau)
  \ll (1 + |\tau|)^{-j}
\end{equation}
% \end{lemma}
for all fixed
$j \in \Bbb{Z}_{\geq 0}$;
this estimate follows from Stirling's formula for $|\tau| \geq
1$ and is otherwise trivial.

From this estimate, we derive a useful approximation
for the local variation of $\gamma(s,\sgn^a)$,
as follows.
For $w = u + i v$,
we may write
\begin{equation}\label{eqn:gamma-real-conductors}
  \gamma(s + w,\sgn^a)
  =
  \left( \frac{|v|}{2 \pi} \right)^{1/2-s-w}
  \exp(i \phi_v(\tau))
  g_{\sigma,u,a}(\tau,v),
\end{equation}
where
\begin{equation}\label{eqn:phiR}
  \phi_v(\tau) :=
  - (v + \tau) \log \left( \frac{|1 + \tau/v| }{e}
  \right) %\vb{\text{added absolute value signs}}
\end{equation}
and
% \pn{added a couple more absolute value signs:}
% \begin{equation}\label{eqn:}
$ g_{\sigma,u,a}(\tau,v)
=
(e^{-1} |1 + \tau/v|)^{1/2-\sigma-u}
g_{\sigma+u,a}(\tau+v). $
% \end{equation}
For $|v| \geq \max(1,2 |\tau|)$
and $\sigma, u \ll 1$,
we conclude from \eqref{eqn:stir1} that 
\begin{equation}\label{eqn:g-estimates}
  \partial_\tau^{j_1}
  \partial_v^{j_2}
  g_{\sigma,u,a}(\tau,v)
  \ll
  |v|^{-j_1-j_2}.
\end{equation}

\subsubsection{Characters of $\mathbb{C}^\times$}
We now record the analogous discussion over $\mathbb{C}$.
Set $\Gamma_{\Bbb{C}}(s) = 2(2\pi)^{-s} \Gamma(s)$.
The basic
local $\gamma$-factors
over $\mathbb{C}$
are given with
respect to the standard additive character
$\psi_\mathbb{C}(x) := e^{2 \pi i (x + \bar{x})}$
of $\mathbb{C}$
by
\begin{equation}\label{eqn:complex-gamma}
  \gamma_{\mathbb{C}}(s,\sgn_{\mathbb{C}}^a)
  =
  i^{a + 1}
  \frac{\Gamma_\mathbb{C}(1 - s + |a|/2)
  }{\Gamma_\mathbb{C}(s + |a|/2) }
  \quad
  (s \in \mathbb{C}, a \in \mathbb{Z}),
\end{equation}
corresponding
to
the character $|.|_{\mathbb{C}}^s \sgn_{\mathbb{C}}^a$
of $\mathbb{C}^\times$;
here
$|z|_{\mathbb{C}} := z \bar{z}$,
$\sgn_{\mathbb{C}}(z) := z / |z|$.
We extend the definition
\eqref{eqn:complex-gamma}
to arbitrary
$a \in \mathbb{R}$
by taking
$i^{a+1} := \exp(\tfrac{i \pi }{2} (a + 1))$. We suppose henceforth
that $a \geq 0$.
For $s = \sigma + i \tau$,
we define
$g_{\mathbb{C},\sigma}(\frac{a}{2} +i \tau)$
by writing
\begin{equation*}%\label{eqn:}
  \gamma_{\mathbb{C}}(s,\sgn_{\mathbb{C}}^a)
  =
  \left( \frac{|a/2 + i \tau |}{2 \pi e} \right)^{1-2 s}
  \left( \frac{
      a/2 + i \tau 
    }{
      |      a/2 + i \tau 
      |
    }
  \right)^{-a}
  g_{\mathbb{C},\sigma}
  \left(\frac{a}{2} +i \tau \right).
\end{equation*}
Again, for $\sigma$ restricted to a fixed interval and $(s, a)$ a fixed distance away from poles of $\gamma_\mathbb{C}(s,\sgn_{\mathbb{C}}^a)$, Stirling's formula implies 
   \begin{equation}\label{eqn:stir2}
    \partial_{a}^{j_1}
    \partial_{\tau}^{j_2}
    g_{\mathbb{C},\sigma}(z)
    \ll_{\mathcal{D},j_1,j_2} (1 + |\tau| + |a|)^{-j_1-j_2}
  \end{equation}
  for $j_1, j_2 \geq 0$. 
%\end{lemma}
  We write
 \begin{equation}\label{eqn:gamma-complex-conductors}
  \gamma_{\mathbb{C}}(s+w,\sgn^a_{\mathbb{C}})
  =
  \left( \frac{
     |
      a/2 + i v 
     |
    }{
      2 \pi 
    }
  \right)^{1 - 2 s - 2 w}
  \exp(i \phi_{\mathbb{C},v,a}(\tau))
  g_{\mathbb{C},\sigma,u}
  \left(\tau, \frac{a}{2} + i v \right),
\end{equation}
where
%\pn{fixed by adding the final term}
\begin{equation}\label{eqn:phiC}
  \phi_{\mathbb{C},v,a}(\tau) :=
  - 2 (v + \tau )
  \log \left(
    \frac{1}{e}
    \left\lvert 1 + \frac{i \tau }{a/2 + i v}
    \right\rvert
  \right)
  -
  a
  \arg \left( \frac{a}{2} + i (\tau + v) \right)
\end{equation}
and
%\begin{equation}\label{eqn:}
 $ g_{\mathbb{C},\sigma,u}\left(\tau, \frac{a}{2} + i v\right)
  =
 (e^{-1} 
   |1 +  i \tau /(\frac{a}{2} + i v)|
    )^{1 - 2 \sigma - 2 u}
  g_{\mathbb{C}, \sigma+u}
(\frac{a}{2}+ i (v + \tau )  ).$
%\end{equation}
For $| a/2+ i v  | \geq \max(1,2 |\tau|)$
and $\sigma, u \ll 1$,
we obtain from \eqref{eqn:stir2} that 
\begin{equation}\label{eqn:g-C-estimates}
  \partial_\tau^{j_1}
  \partial_v^{j_2}
  \partial_a^{j_3}
  g_{\mathbb{C}, \sigma,u}
  \left(\tau, \frac{a}{2} + i v \right)
  \ll
  (|v| + |a|)^{-j_1-j_2-j_3}.
\end{equation}

\subsubsection{The general definition}
Any $n$-dimensional
representation of the Weil group
$W_\mathbb{R}$
is isomorphic to a direct sum
\begin{equation}\label{eqn:rho-sum}
  \rho
  =
  (\oplus_{j=1}^{n_1} |.|^{w_j} \sgn^{a_1})
  \oplus
  (\oplus_{j=1}^{n_2} |.|_{\mathbb{C}}^{z_j}
  \sgn_{\mathbb{C}}^{b_j})
\end{equation}
where $n = n_1 + 2 n_2$,
$w_j, z_j \in \mathbb{C}, a_j \in \{0,1\}$ and
$b_j \in \mathbb{Z}_{\geq 1}$.
Here we identify the indicated
characters of $\mathbb{C}^\times \cong W_\mathbb{C}$ with the
corresponding two-dimensional induced representations of
$W_{\mathbb{R}}$.
The local $\gamma$-factor of $\rho$
%(with respect to $\psi$ as in \S\ref{sec:local-gamma-characters-reals})
is now given by
\begin{equation*}%\label{eqn:}
  \gamma(s,\rho)
  =
    \prod_{j=1}^{n_1}
    \gamma(s+w_j,\sgn^{a_j})
    \prod_{j=1}^{n_2}
    \gamma_{\mathbb{C}}(s+z_j,\sgn_{\mathbb{C}}^{b_j})
  .
\end{equation*}

Write $w_j = u_j + i v_j$ and $z_j = x_j + i y_j$.
We define
the analytic conductor
\begin{equation}\label{crho}
  C(\rho) :=
    \prod_{j=1}^{n_1}
    \frac{\max(1,|v_j|)}{2 \pi }
    \prod_{j=1}^{n_2}
    \frac{\max \left(1,
         b_j/2  + i y_j
          \right)^2
    }{(2 \pi)^2 }
\end{equation}
and, using \eqref{eqn:phiR} and \eqref{eqn:phiC}, the phase function
\begin{equation*}%\label{eqn:}
  \phi_\rho(\tau) :=
  \sum_{j=1}^{n_1}
  \phi_{v_j}(\tau)
  +
  \sum_{j=1}^{n_2}
  \phi_{\mathbb{C},y_j,b_j}(\tau)
\end{equation*}
and the factors
\begin{equation*}%\label{eqn:}
  g_{\sigma}(\tau,\rho)
  :=
    \prod_{j=1}^{n_1}
    g_{\sigma,u_j,a_j}(\tau,v_j)
    \prod_{j=1}^{n_2}
    g_{\mathbb{C},\sigma,x_j}
    \left(\tau,\frac{b_j}{2} + i y_j \right)
 ,
\end{equation*}
\begin{equation*}%\label{eqn:}
  e_\rho
  :=
    \prod_{j=1}^{n_1}
    \left(
      \frac{|v_j|}{2 \pi }
    \right)^{-i v_j}
    \prod_{j=1}^{n_2}
    \left(
      \frac{
        b_j/2          + i y_j
      }{2 \pi }
    \right)^{- 2 i y_j}. 
\end{equation*}

By the \emph{dual} (resp.\ \emph{conjugate})
of $\rho$,
we mean the representation
obtained by negating (resp.\ by conjugating) the parameters
$w_j, z_j$ in \eqref{eqn:rho-sum}. We summarize the previous discussion in the following lemma. 
\begin{lemma}\label{lem:gamma-s-rho-via-conductors}
  Suppose that
  $\rho$ is isomorphic
  to its conjugate dual. 
  Then
  \begin{equation}\label{eqn:gamma-s-rho-via-conductors}
    \gamma(s,\rho)
    =
    e_\rho
    C(\rho)^{1/2-s}
    \exp(i \phi_\rho(\tau))
    g_{\sigma}(\tau,\rho).
  \end{equation}
  If moreover
  $\rho$ is self-dual
  (equivalently, self-conjugate),
  then
  \begin{equation*}%\label{eqn:}
    e_\rho = 1.
  \end{equation*}
\end{lemma}
\begin{proof}
  The content of our hypothesis
  is that we have the equalities of
  multisets
  \begin{equation}\label{eqn:multisets-1}
    \{
    (w_1,a_1),
    \dotsc,
    (w_{n_1},a_{n_1})
    \}
    =
    \{
    (-\overline{w_1},a_1),
    \dotsc,
    (-\overline{w_{n_1}},a_{n_1})
    \},
  \end{equation}
  \begin{equation}\label{eqn:multisets-2}
    \{
    (z_1,b_1),
    \dotsc,
    (z_{n_2},b_{n_2})
    \}
    =
    \{
    (-\overline{z_1},b_1),
    \dotsc,
    (-\overline{z_{n_2}},b_{n_2})
    \}.
  \end{equation}
  It follows that
  \[
    \prod_{j=1}^{n_1}
    \left(
      \frac{|v_j|}{2 \pi }
    \right)^{-u_j}
    =
    \prod_{j=1}^{n_2}
    \left(
      \frac{|b_j/2 + i
        y_j |}{2 \pi }
    \right)^{-x_j}
    =
    1.
  \]
  We deduce \eqref{eqn:gamma-s-rho-via-conductors} by
  multiplying together the identities
  \eqref{eqn:gamma-real-conductors} and
  \eqref{eqn:gamma-complex-conductors}.  Assuming moreover that
  $\rho$ is self-dual, we obtain the additional equalities of
  multisets as in \eqref{eqn:multisets-1} and
  \eqref{eqn:multisets-2}, but without the conjugations,
  which in turn give that $e_\rho = 1$.
\end{proof}

The primary hypothesis of Lemma
\ref{lem:gamma-s-rho-via-conductors} is satisfied if, for
instance, $\rho$ corresponds to a \emph{unitary} representation
$\pi$ of $\GL_n(\mathbb{R})$, while the full hypotheses are
satisfied if $\pi$ is self-dual.

\subsubsection{The examples of interest}\label{223}
We consider in this paper
cuspidal automorphic representations
$\pi$
for $\SL_2(\mathbb{Z})$.
Each such $\pi$ defines
a generic irreducible unitary representation
of $\PGL_2(\mathbb{R})$,
hence a two-dimensional
representation $\rho_\pi$
of $W_\mathbb{R}$.
We set
\begin{equation}\label{eqn:gammarep}
  \gamma(s,\pi) := \gamma(s,\rho_\pi),
\end{equation}
and similarly define
$C(\pi), \phi_\pi(\tau)$
and $g_\sigma(\tau,\pi)$.
The possibilities for $\rho_\pi$
are as follows:
\begin{enumerate}
\item $\pi$
  is a principal series representation
  $\pi = \pi(r,a)$
  obtained
  by normalized induction
  of the character
  $|.|^{i r} \sgn^a$
  for some $r \in \mathbb{R} \cup (-1/2,1/2)i$
  and $a \in \{0, 1\}$,
  in which  case
  $\rho_\pi =
  |.|^{i r} \sgn^a \oplus |.|^{-i r} \sgn^a$,
  or
\item $\pi$
  is a discrete series representation
  $\pi = \pi(k)$
  of lowest weight $k \in 2 \mathbb{Z}_{\geq 1}$,
  in which case
  $\rho_\pi = \sgn_{\mathbb{C}}^{k-1}$.
\end{enumerate}
We note that any such $\pi$ is self-dual,
hence any such $\rho$ is
both self-dual and self-conjugate;
this property is evident in each example.
Thus  for $s = \sigma + i \tau$,
we have
\begin{equation}\label{eqn:gamma-s-pi-approximation}
  \gamma(s,\pi)
  =
  C(\pi)^{1/2 - s}
  \exp(i \phi_\pi(\tau))
  g_\sigma(\tau,\pi)
\end{equation}
where
\begin{itemize}
\item for $\pi = \pi(r,a)$,
  we have
\begin{equation}\label{maass}
\begin{split}
&  C(\pi)
  =
  \left( \frac{\max(1,|\Re(r)|)}{2 \pi } \right)^2,\\
&  \phi_\pi(\tau)
  =
  - (r + \tau )
  \log\Big(\Big|1 + \frac{\tau}{r}\Big|\Big)
  -
  (-r + \tau )
  \log\Big(\Big|1 -\frac{ \tau}{r}\Big|\Big)
  +
  2 \tau,
\end{split}
\end{equation}
%\[
 %  \vb{\text{added absolute value signs, changed $v$ to
% $r$}}
%  \pn{\text{changed $-2 \tau$ to $2 \tau$}}
%\]
%while
\item
  for $\pi = \pi(k)$,
\begin{equation}\label{hol}
\begin{split}
&    C(\pi) = 
    \left( \frac{\max(1,(k-1)/2)}{2 \pi } \right)^2,\\
&    \phi_\pi(\tau) =
    - 2 \tau
    \log
    \left(
      \frac{1}{e}
      \left\lvert
        1 + \frac{i \tau }{ (k-1)/2}
      \right\rvert
    \right)
    -
    (k-1)
    \arg \left( \frac{k-1}{2} + i \tau \right)
\end{split}
\end{equation}
\end{itemize}
 and
$g_\sigma(\tau,\pi)$
varies mildly
in the sense given
by the estimates
\eqref{eqn:g-estimates}
and
\eqref{eqn:g-C-estimates},
namely
\begin{itemize}
\item
  for $|r| \geq \max(1, 2 |\tau|)$
  and $\sigma \ll 1$,
  \begin{equation*}%\label{eqn:}
  \partial_\tau^{j_1}
  \partial_{r}^{j_2}
  g_\sigma(\tau,\pi(r,a))
  \ll
  |r|^{-j_1 - j_2},
\end{equation*}
%while
\item for
  $|k-1| \geq 4 |\tau|$
  and $\sigma \ll 1$,
  \begin{equation}\label{eqn:g-sigma-tau-pi-k-estimate}
    \partial_\tau^{j_1}
    \partial_k^{j_2}
    g_\sigma(\tau,\pi(k))
    \ll
    |k|^{-j_1 - j_2}.
  \end{equation}
\end{itemize}
We note that,
while $\pi(k)$
is not defined as a representation
for non-integral $k$,
each of the factors
$\gamma(s,\pi(k))$,
$C(\pi(k))$,
$\phi_{\pi(k)}$ and
hence also $g_\sigma(\tau,\pi(k))$
is defined
for any $k \in \mathbb{R}_{\geq 1}$;
see after \eqref{eqn:complex-gamma}.
For this reason,
it makes sense to differentiate
with respect to $k$
in \eqref{eqn:g-sigma-tau-pi-k-estimate}.

%\pn{added:}
 On one occasion, we will apply Lemma
\ref{lem:gamma-s-rho-via-conductors} to a Rankin--Selberg
convolution $\pi_1 \otimes \pi_2 \otimes \chi$ of a pair of
generic irreducible unitary representations of
$\PGL_2(\mathbb{R})$, twisted further by a character
$\chi$ of $\GL_1(\mathbb{R})$.
Writing $\chi = \chi_0 |.|^{\Re(\chi)}$
with $\chi_0$ unitary,
we have
\begin{equation}\label{eqn:lemma-1-RS}
  \gamma(1/2,\pi_1 \otimes \pi_2 \otimes \chi)
  = \gamma(1/2 + \Re(\chi),\pi_1 \otimes \pi_2 \otimes \chi_0)
  \ll
  C(\pi_1 \otimes \pi_2
  \otimes \chi)^{-\Re(\chi)},
\end{equation}
where in the second step
we
invoke Lemma \ref{lem:gamma-s-rho-via-conductors}
and the accompanying Stirling estimates,
using the unitarity
of $\pi_1, \pi_2$ and $\chi_0$
to verify its hypotheses.

\subsection{General bounds for Whittaker functions}\label{sec:gener-bounds-whitt}
%\pn{tweaked the following}\sj{good}

Let $\pi$ be a generic irreducible unitary representation of $G := \PGL_2(\mathbb{R})$.  We recall that ``generic'' means that there is a $G$-equivariant embedding, necessarily unique,
\begin{equation*}
  \pi\hookrightarrow \{W:G\to\C\text{ smooth }\mid W(n(x)g)=e(x)W(g)\},
\end{equation*}
\begin{equation*}
  v \mapsto W_v,
\end{equation*}
where $G$ acts on the space on the right hand side by right translation.  The image of $\pi$ under such an embedding is called the \emph{Whittaker model} of $\pi$ with respect to $\psi$.   An invariant inner product on $\pi$ is given by
\begin{equation}\label{defn-inner-prod}
\langle v_1,v_2\rangle_{\pi}:=\int_{\R^\times}W_{v_1}(a(y))\overline{W_{v_2}(a(y))}\, d^\times y.
\end{equation}
When we speak below of $\pi$ being realized in its Whittaker model,  we mean that we identify $\pi$ with its image under such an embedding, with inner product normalized as in \eqref{defn-inner-prod}.

Fix $\vartheta \in [0,1/2)$.  We say that $\pi$  is \emph{$\vartheta$-tempered} if it lies in the discrete series or if, writing $\pi$ as a Langlands quotient of an isobaric sum $\sigma_1 \otimes |\det|^{s_1} \boxplus \sigma_2 \otimes |\det^{s_2}$, we have that each $|\Re(s_i)| \leq \vartheta$.  Then $\pi$ is $0$-tempered in the above sense if and only if it is tempered in the usual sense, i.e., its matrix coefficients lie in $L^{2+\eps}(G)$ for each $\eps > 0$.

In what follows, we work exclusively with smooth vectors in such representations.  Thus ``let $v \in \pi$'' is shorthand for ``let $v$ be a smooth vector in $\pi$.''

We denote by $\mathcal{S}_d$ the Sobolev norm on $\pi$ defined in \cite[(2.6)]{MV}.  It takes finite values on smooth vectors.

\begin{lemma}\label{lem:unram-W-bounds}
  Let $\pi$ be a $\vartheta$-tempered generic irreducible unitary representation of $G$, realized in its Whittaker model.  For each $W \in \pi$ and all $y \in \mathbb{R}^\times$ and $z \in \mathbb{R}$ with $|z| \leq 1000$, we have
  \begin{equation}\label{eqn:unram-W-bound-near-w}
    (y \partial_y)^{j_2}
    \partial^{j_1}_z
    W(a(y) w n(z))
    \ll \mathcal{S}_{d}(W)
    \min(|y|^{1/2-\vartheta}, |y|^{-N}).
  \end{equation}
  and
  \begin{equation}\label{eqn:unram-W-bound-near-1}
    (y \partial_y)^{j_2}
    \partial^{j_1}_z
    W(a(y) n'(z))
    \ll \mathcal{S}_{d}(W)
    \min(|y|^{1/2-\vartheta}, |y|^{-N})
  \end{equation}
  for all $j_1, j_2,N \in \mathbb{Z}_{\geq 0}$, where
  $d \in \mathbb{Z}_{\geq 0}$
  and the implied constants depend at most upon $j_1,j_2,N$.
  % \vb{We need to upgrade this. The bound on $z$ must be dropped, and we need polynomial dependence in $z$, depending on $j_1, j_2, N$}
\end{lemma}
\begin{proof}
  \cite[\S2.4.1 and \S3.2.3]{MV}. 
\end{proof}

\subsection{Smooth weight functions}\label{weight}
Let $\mathcal{X}$ be a large parameter which will be clear from the
context. We adopt the convention
that $\eps$ denotes a fixed (i.e., independent of $\mathcal{X}$) positive quantity,
whose precise meaning may change from line to line.
%\pn{added}
As usual,
the notation $A \ll B$
means that $|A| \leq C |B|$ for some fixed $C$;
we introduce subscripts as in $A \ll_j B$
to signify that $C$ may depend upon $j$.
We use the notation
$A \asymp B$ to denote that $A$ and $B$ are nonzero
real numbers for which $A/B$ lies in some fixed compact subset
of $(0,\infty)$;
we then have $A \ll B \ll A$. 
We introduce the abbreviation
$$A \preccurlyeq B \Longleftrightarrow A \ll_{\eps} \mathcal{X}^{\eps} B.$$
% where the meaning of $\eps$ can change from line to line.
We call an expression \emph{negligible} if it is
$\ll_N \mathcal{X}^{-N}$ for any $N > 0$.
%We will frequently use the following device. 
We call a smooth  function $V : \Bbb{R}^n \rightarrow \Bbb{C}$ %with
%fixed compact support within $(0, \infty)$ 
%\pn{does ``bounded'' imply ``supported in some fixed set'' or can the support vary arbitrarily?}
  \emph{flat} if  
\begin{equation}\label{deriv}
x_1^{j_1} \cdots x_n^{j_n} V^{(j_1, \ldots, j_n)}(x_1, \ldots, x_n) \preccurlyeq_{\textbf{j}}1
\end{equation}
for all $\textbf{j} \in \Bbb{Z}_{\geq 0}^n$. Clearly if $V$ is
flat, then so is $ \exp(iV)$. If in addition $V$ has fixed
compact support in  $(0, \infty)^n$, we call it \emph{nice}. 
We generally let $V$ denote a  nice function
in one or more variables, \emph{not necessarily the same at every
  occurrence}. In practice, $V$ may depend on some additional parameters having certain prescribed sizes; it will always be clear from the context with respect to which variables ``flatness'' is applied (in which   case  all implied constants are uniform in these parameters).

For   a nice function  $V$, we may
separate variables in $V(x_1, \ldots, x_n)$ by first inserting a redundant function $V(x_1) \cdots V(x_n)$ that is 1 on the support of $V$ and then  applying  Mellin inversion
\begin{displaymath}
\begin{split}
&V(x_1, \ldots, x_n) = V(x_1, \ldots, x_n)V(x_1) \cdots V(x_n) \\
&= \int_{\Re(s_1)=0} \cdots \int_{\Re(s_n)=0} \widehat{V}(s_1, \ldots, s_n) \Big( V(x_1) \cdots V(x_n) x_1^{-s_1} \cdots x_n^{-s_n}\Big) \frac{ds_1 \cdots ds_n}{(2\pi i)^n}.
\end{split}
\end{displaymath}
Here we can truncate the vertical integrals at height $|\Im s|
\preccurlyeq 1$ at the cost of a negligible error.
We will often separate variables in this way without
explicit mention.

\subsection{Integration by parts and stationary phase}\label{sec25}
We quote the following lemmas from \cite[Section 8]{BKY} and its
extension in \cite[Section 3]{KPY}.
\begin{lemma}\label{lem1}
 Let $Y \geq 1$, $X, P, U, S > 0$, %Y> 0$, %$Q \geq 1$, %$0 < \delta < 1/10$ and 
%$U \geq 1$, 
and suppose that $w$ %=w_{U}$
 is a smooth function with support on $[\alpha, \beta]$, satisfying
\begin{equation*}
w^{(j)}(t) \ll_j X U^{-j}.
\end{equation*}
Suppose $h$ %=h_{Q,U}$ 
  is a smooth function on $[\alpha, \beta]$ such that
\begin{equation*}
 |h'(t)| \gg S
\end{equation*}
for some $S > 0$, and
\begin{equation*}%\label{diffh0}
h^{(j)}(t) \ll_j Y P^{-j}, \qquad \text{for } j=2, 3, \dots.
\end{equation*}
Then 
\begin{equation*}
%\label{eq:Ipartsbound}
 \int_{t\in \Bbb{R}} w(t) e^{i h(t)} dt \ll_A (\beta - \alpha) X [(PS/\sqrt{Y})^{-A} + (SU)^{-A}].
\end{equation*}
\end{lemma}

%\pn{very mild tweaks to the statement}
\begin{lemma}\label{lem2}%[Stationary phase with a smooth weight function]
  Let $\mathcal{X}$ be a large parameter.
  Let $V$ be a flat function in
  the sense of
\S\ref{weight} with  support in $\times_{j=1}^d [c_{1j},
c_{2j}]$ for some
fixed
intervals $[c_{1j}, c_{2j}] \subseteq \Bbb{R}$ not containing
0. Let $X_1, \ldots, X_d > 0$, $Y \geq \mathcal{X}^{\eps}$.  Write
$\mathscr{S} =   \times_{j=1}^d [c_{1j}X_j, c_{2j}X_j] \subseteq
\Bbb{R}^d$. Let $\phi : \Bbb{R}^d \rightarrow \Bbb{R}$ be a
smooth function satisfying
the derivative upper bounds\footnote{The main result in \cite[Section 3]{KPY} states this with $\ll$ instead of $\preccurlyeq$, but our conclusion on $W$ is insensitive to $Q^{\eps}$-powers.}
\begin{equation*}%\label{lem2a}
\phi^{(j_1, \ldots, j_d)}(t_1; t_2, \ldots, t_d) \preccurlyeq Y \prod_{i=1}^d X_i^{-j_i}
\end{equation*}
for $\textbf{j} \in \Bbb{N}_0^d$ and $(t_1, \ldots t_d) \in
\mathscr{S}$, as well as
the following second derivative lower bound in the first variable:
\begin{equation*}%\label{lem2b}
\phi^{(2, 0, \ldots, 0)}(t_1; t_2, \ldots, t_d) \gg Y X_1^{-2}.
\end{equation*}
Suppose that there exists $t^{\ast}  = t^{\ast}(t_2, \ldots, t_d)$ such that $\phi^{(1, 0, \ldots, 0)}(t^{\ast}, t_2, \ldots, t_d) = 0$. Then for any $N > 0$ we have 
\begin{multline*}
\int_{\Bbb{R}} V\Big(\frac{t_1}{X_1}, \ldots, \frac{t_d}{X_d}\Big) e^{i\phi(t_1, \ldots, t_d)} dt_1 \\ = \frac{X_1}{Y^{1/2}} e^{i\phi(t^{\ast}, t_2, \ldots, t_d)} W\Big(\frac{t_2}{X_2}, \ldots, \frac{t_d}{X_d}\Big)  + O_N(\mathcal{X}^{-N})
\end{multline*}
 for a flat function $W = W_N$  with support in $\times_{j=2}^d [c_{1j}, c_{2j}]$. 
 \end{lemma}

\section{The local triple product factor}\label{local}

 Let $\pi_i$ for $i=1,2,3$ be generic irreducible unitary
representations of $G$ such that
\begin{itemize}
\item $\pi_1$ and $\pi_2$ are $\vartheta$-tempered, while
\item  $\pi_3$ is tempered.
\end{itemize}
We regard $\pi_1$ and $\pi_2$ as fixed.
We write $Q = C(\pi_3)$ for the conductor of $\pi_3$ and think of $Q$ as a large parameter.  The aim of this section is to obtain a lower bound for  the local triple product integral $\mathcal{L}_{\infty}(v_1, v_2, v_3)$ in \eqref{period} for a certain choice of vectors $v_j \in \pi_j$. The choice will be made at the beginning of \S\ref{choice} and the result will be stated in Theorem \ref{lower-bound} at the end of this section. 

%\sj{added}
Let $\psi$ denote the additive character of $N$ given by $n(x)\mapsto e(x)$. We realize $\pi_1$  (resp.\ $\pi_3$) in its Whittaker model with respect to $\psi$ (resp.\ $\bar{\psi}$), with inner products normalized as in  \eqref{defn-inner-prod}.

In this section we abbreviate $\chi_s := \chi\otimes|.|^s$
for   $\chi\in\mathfrak{X}(\R^\times)$ and $s\in\C$.

\subsection{Embedding via intertwiners}\label{sec:embedding}
Let $\chi\in\mathfrak{X}(\R^\times)$.
Let $\Ic(\chi)$ denote
the unitarily normalized induction of $\chi$ from the
%\pn{clarified that we work with the upper-triangular Borel.
%  This is important.  Many references work with the
%  lower-triangular one.
%}
standard upper-triangular Borel
subgroup in $G$,
consisting of smooth $f : G \rightarrow \mathbb{C}$ satisfying
$f(n(x) a(y) g) = |y|^{1/2} \chi(y) f(g)$.
Let $M(\chi)$
denote
the standard intertwining operator from the principal series
$\Ic(\chi)$ to $\Ic(\chi^{-1})$, defined by
the integral
\begin{equation}\label{defn-intertwiner}
    f\mapsto \int_{x\in\R} f(wn(x).)dx
\end{equation}
for $\Re(\chi)>0$ and then meromorphically continued to all of
$\mathfrak{X}(\R^\times)$.
 
Let $\chi=|.|^{1/2+k}$ for $k\in\Z_{\geq 0}$ and consider a $K$-type basis $\{f_{2l}\}_{l\in \Z}$ on $\Ic(\chi)$. From the computation of \cite[Proposition 2.6.3]{Bum} we see that $M(\chi)f_l=0$ for $|l|\geq 1+k$.  Thus $M(\chi)$ has a unique infinite dimensional kernel which is isomorphic
to the discrete series $D_k$ of weight $k$. 
 We normalize $M(\chi)$ as
\begin{equation}\label{normalized-intertwiner}
    M^*(\chi):=\gamma(0,\chi^2)M(\chi),
\end{equation}
where $\gamma$ is the local Tate gamma factor as in \S
\ref{localTate}.
 Then $M^*(\chi)$ is non-zero for all $\Re(\chi)>0$ and is
meromorphic for all $\chi$. 
 In other words $D_k$ can be embedded into the principal series representation $\Ic(\chi)$ with $\chi=|.|^{k+1/2}$ via the normalized intertwining operator $M^*(\chi)$. A similar embedding can be done for the complementary series representation as well, see \cite[\S2.6]{Bum}.

Let $\chi$
  with $\Re(\chi)\ge 0$ not be a pole of $M^*(\chi)$. From now on we will only consider $\chi$ for which either $\Re(\chi)=0$ or $\Im(\chi)=0$. Note that if $\Ic(\chi)$ is unitary then $\chi$ satisfies this property.
%\pn{Perhaps ``Let $\chi$
%  with $\Re(\chi)\ge 0$ not be a pole of $M^*(\chi)$?
%Would that have the same meaning?}
We can define a $G$-invariant sesquilinear pairing on $\Ic(\chi)$ by \begin{equation*}
    (f_1,f_2)_0:=
    \begin{cases}
    \int_{x\in \R}f_1(n'(x))\overline{f_2(n'(x))}dx,&\text{ if }\Re(\chi)=0,\\
    \int_{x\in \R}f_1(n'(x))\overline{M^*(\chi)f_2(n'(x))}dx,&\text{ if } \Im(\chi)=0.
    \end{cases}
\end{equation*}
for $f_1,f_2\in \Ic(\chi)$.
 
 There is a principal series representation
$\pi_2^p = \Ic(\chi)$,
with $\chi$ of nonnegative
real part,
into which $\pi_2$ embeds.
Explicitly:
\begin{itemize}
\item If $\pi_2$
is a tempered principal series $\Ic(\chi_0)$, i.e.\  if $\chi_0$ is
unitary,
then we choose $\chi=\chi_0$.
\item If $\pi_2$ is the weight $k$ discrete series $D_k$, then
  we choose $\chi=|.|^{1/2+k}$. 
\item If $\pi_2$ is the complementary series attached to
  $0<\sigma<1/2$ then we choose
  $\chi=|.|^\sigma$. %\sj{I deleted $\delta\in\{0,1\}$ as this is redundant for complementary series for which $\delta$ is necessarily $0$.}
\end{itemize}
In each case, we have a $G$-invariant embedding $\pi_2\hookrightarrow \pi_2^p=\Ic(\chi)$ and $c(\chi)\in\C^\times$ such that
$$\langle v_1, v_2\rangle_{\pi_2} = c(\chi)(f_{v_1},f_{v_2})_0=:(f_{v_1},f_{v_2}),$$
where $f_{v_i}$ are the images of $v_i$ under the above
embedding and $\langle,\rangle_{\pi_2}$ is as defined in \eqref{defn-inner-prod}. We refer to  \cite[\S2.6]{Bum} for details.

\subsection{Local Rankin--Selberg zeta integral} 
 Let $W_1 \in \pi_1$ and $W_3 \in \pi_3$ and 
 let $f_2\in \Ic(\chi)$ for some
$\chi\in\mathfrak{X}(\R^\times)$.
We may parametrize $f_2$
in terms of a Schwartz function, as follows. Let $e_2:=(0,1)\in\R^2$ and $\Phi\in \mathcal{S}(\R^2)$ a Schwartz function. We define
$$f_2(g):=\int_{t\in \R^\times}\Phi(e_2tg)\chi_{1/2}(\det(tg))d^\times t.$$
The above integral converges absolutely for $\Re(\chi)>-1/2$ and
continues meromorphically
%\pn{changed from ``has an analytic continuation,''
%  since there are in general poles}
to all $\chi\in\mathfrak{X}(\R^\times)$.

The local $\GL(2)\times\GL(2)$ Rankin--Selberg zeta integral of $\pi_1$ and $\pi_3$ is defined by
\begin{align*}
    \Psi(W_1,f_2,W_3)&:=\int_{g\in N\backslash G} W_1(g)f_2(g)W_3(g)dg\\
    &=\int_{g\in N\backslash\GL_2(\R)}W_1(g)W_3(g)\Phi(e_2g)\chi_{1/2}(\det(g))dg,
\end{align*}
for $\Re(\chi)$ sufficiently large and in general by meromorphic continuation.
%\vb{I looks like there is a typo above and below: $W_2$ should be $W_3$ ?}
The $\GL(2)\times \GL(2)$ local functional
equation (see \cite[Theorem 3.2]{Cog})
asserts, using the notation \eqref{eqn:gammarep} and \eqref{localgamma}, 
%\pn{changed $W_2(g)$ to $W_3(g)$
%  in the first integral of the following identity}
\begin{multline}\label{G-G-lfe}
    \gamma(1/2,\pi_1\otimes\pi_3\otimes\chi)\int_{g\in N\backslash\GL_2(\R)}W_1(g)W_3(g)\Phi(e_2g)\chi_{1/2}(\det(g))dg\\=\int_{g\in N\backslash\GL_2(\R)}\tilde{W}_1(g)\tilde{W}_3(g)\hat{\Phi}(e_2g)\chi^{-1}_{1/2}(\det(g))dg,
\end{multline}
where $\tilde{W}_i\in\tilde{\pi}_i$ is the contragredient of
$W_i$ defined by $\tilde{W}_i(g)=W(wg^{-\top})$ %\vb{changed the symbol. OK?} \pn{good}
and $\hat{\Phi}$ is
the Fourier transform of $\Phi$ defined by
$$\hat{\Phi}(y):=\int_{x\in\R^2}\Phi(x)e(y^{\top}x)dx.$$
 For $\chi$
a fixed distance away from a pole
or zero of $\gamma(1/2,\pi_1\otimes\pi_3\otimes\chi)$, we have
\begin{equation}\label{bound-gamma-factor}
  \begin{split}
  \gamma(1/2,\pi_1\otimes\pi_3\otimes\chi)^{-1}
  &\asymp
  \gamma(1/2,\tilde{\pi}_1\otimes\tilde{\pi}_3\otimes\chi^{-1})
  \\
  &\ll _{\Re(\chi)}
  C(\pi_1\otimes\pi_3 \otimes \chi)^{\Re(\chi)} \ll_{\pi_1,\chi}C(\pi_3)^{2\Re(\chi)}
  %C(\chi)^{4 |\Re(\chi)|}
  ,\\
\end{split}
\end{equation}
when $\chi$ is fixed with $\Re(\chi)\ge 0$.
The first estimate above follows from the definition of the
gamma factor.
The second estimate follows
from
\eqref{eqn:lemma-1-RS}. 
 The third estimate follows from repeated application of \cite[Lemma A.2]{HB}.
 
We record a variant of the local functional equation:
\begin{lemma}\label{normalized-lfe}
We have
\begin{equation*}
    \Psi(W_1,f_2,W_3)\gamma(1/2,\pi_1\otimes\pi_3\otimes\chi)=\Psi(W_1,M^*(\chi)f_2,W_3),
\end{equation*}
where $M^*(\chi)$ is as in \eqref{normalized-intertwiner}.
\end{lemma}

 \begin{proof}
Let $\Re(\chi)$ be sufficiently large. By expanding the
definition \eqref{defn-intertwiner}
of the intertwining operator, we see that
$$M(\chi)f_2(g)=\chi_{1/2}(\det(g))\int_{x\in\R}\int_{t\in \R^\times}\Phi((t,x)g)\chi^2(t)d^\times t \, dx.$$
We use local Tate functional equation to evaluate the above as
$$\gamma(0,\chi^2)^{-1}\chi(\det(g))|\det(g)|^{-1/2}\int_{t\in \R^\times}\hat{\Phi}((t,0)g^{-\top})\chi^{-2}(t)|t|d^\times
t$$
(compare with \cite[p.\ 225]{GJ}).
Recalling \eqref{normalized-intertwiner},
we may thus write
$$M^*(\chi)f_2(g)=\int_{t\in \R^{\times}}\hat{\Phi}(e_2twg^{-\top})\chi^{-1}_{1/2}(\det(tg^{-\top}))d^\times t.$$
We use the definition of $\tilde{W}_i$ and change variables
$g\mapsto wg^{-\top}$ in the right hand side of
the local functional equation \eqref{G-G-lfe} to write
\begin{multline*}
    \gamma(1/2,\pi_1\otimes\pi_3\otimes\chi)\int_{g\in N\backslash\GL_2(\R)}W_1(g)W_3(g)\Phi(e_2g)\chi_{1/2}(\det(g))dg\\=\int_{g\in N\backslash\GL_2(\R)}W_1(g)W_3(g)\hat{\Phi}(e_2wg^{-\top})\chi_{-1/2}(\det(g))dg.
  \end{multline*}
%  \pn{We should indicate somewhere choices of Haar measure,
%    especially that those on $\GL_2(\mathbb{R})$ and
%    $\PGL_2(\mathbb{R})$
%    are related via the obvious exact sequence
%  }\sj{added a paragraph explaining the Haar measures in \S2.1}
Folding the above integrals over $\R^\times$, the identity follows for $\Re(\chi)$ large. We conclude the proof by meromorphic continuation of the zeta integrals and the intertwiner.
\end{proof}

\begin{lemma}\label{extended-ichino}    
  Let $\pi_1$ and $\pi_2$ be $\vartheta$-tempered with $\vartheta<1/4$ and $\pi_3$ be tempered. Let $\pi_i\ni v_i\mapsto W_i$ for $i=1,3$ be realized in their respective Whittaker models equipped with the inner products as defined in \eqref{defn-inner-prod}. Also let $v_2\mapsto f_2$ under $\pi_2\hookrightarrow\pi_2^p=\Ic(\chi)$ as described in \S\ref{sec:embedding}. Then
$$\int_{g\in G}\prod_{i=1}^3\langle \pi_i(g)v_i,v_i\rangle_{\pi_i} dg = c(\chi)\Psi(W_1,f_2,W_3)\overline{\Psi(W_1,\tilde{f}_2,W_3)},$$
where 
$$\tilde{f}_2=
\begin{cases}
f_2, &\text{ if $\pi_2^p$ is a tempered principal series,}\\
M^*(\chi)f_2, &\text{ otherwise}.
\end{cases}$$
\end{lemma}

From \cite[\S2.5.1]{MV} we have the bound of the matrix coefficients
$$\langle \pi_i(g)v_i,v_i\rangle \ll_{\pi_i}
\Xi(g)^{1-2\vartheta}
\text{ for } i=1,2,
\quad
\langle \pi_3(g)v_3,v_3\rangle \ll_{\pi_3}
\Xi(g).$$
Here $\Xi$ is the Harish-Chandra $\Xi$-function,
which satisfies $\int_{g\in G}\Xi(g)^{2+\epsilon}dg<\infty$.
Thus from the assumption that $\vartheta<1/4$ we
see that the local triple product integral is absolutely convergent.

\begin{proof}
The proof is essentially given in  \cite[2.14.3 Lemma]{Nel}, but
in an analogous metaplectic setting. 
 We modify the relevant part of the proof. 
Note that the left hand side of the equation in the lemma is 
$$\int_{g\in G}\langle \pi_1(g) W_1,W_1\rangle (\pi_2^p(g) f_2,f_2)\langle \pi_3(g)W_3,W_3\rangle dg.$$
We define 
$$\xi_1:=W_1 f_2,\quad \xi_2:= W_1 \tilde{f}_2,$$
and note that 
$$\xi_i(ng)=\psi(n)\xi_i(g),\quad n\in N, g\in G.$$
Using Iwahori coordinates  $g=a(y)n'(x) \in N\backslash G$ and the transformation of $f_2$ under the Borel subgroup we compute the absolutely convergent integral  
\begin{equation*}
\begin{split}
   & \int_{h\in N\backslash G}\xi_1(hg)\overline{\xi_2(h)}\,dh\\
    &=\int_{x\in \R}\pi_2^p(g)f_2(n'(x))\overline{\tilde{f}_2(n'(x))}\int_{y\in \R^\times}\pi_1(g)W_1(a(y)n'(x))\overline{W_1(a(y)n'(x))}d^\times y \, dx.
    \end{split}
\end{equation*}
The inner integral evaluates to $\langle \pi_1(g)W_1,W_1\rangle$ and consequently, we have
$$\int_{h\in N\backslash G}\xi_1(hg)\overline{\xi_2(h)}\,dh = \langle \pi_1(g)W_1,W_1\rangle (\pi_2^p(g)f_2,\tilde{f}_2)_0.$$
Hence, the left hand side of the equation in the lemma equals  
$$c(\chi)\int_{g\in G} \int_{h\in N\backslash G}\xi_1(hg)\overline{\xi_2(h)}\langle \pi_3(g)W_3,W_3\rangle \,dh\,dg.$$
The above double integral is only conditionally convergent. We proceed exactly as in the proof of the identity \cite[(2.29)]{Nel} to evaluate the above integral as
$$c(\chi)\int_{h\in N\backslash G}\xi_1(h)W_3(h) \,dh \int_{h\in N\backslash G}\overline{\xi_2(h)W_3(h)}\, d h.$$
The proof is now complete.
\end{proof}

\subsection{Choice of the vectors}\label{choice}
We
%\pn{deleted ``use the line model'' -- this phrase seems to add nothing}
%use the line model to
choose $f_2\in \pi_2^p$ as before
$$f_2(g):=\int_{t\in\R^\times}\Phi(e_2tg)\chi_{1/2}(\det(tg))d^\times t,$$
where $\Phi$ is a smooth non-negative bump function on $\R^2$ sufficiently concentrated around the point $e_2=(0,1)$ in terms of $\pi_1$ and $\pi_2$ only. 
Such a vector has a non-zero preimage $v_2'\in\pi_2$. 
 We choose $$v_2:=a(Q)v_2'$$
where, as we recall, $Q = C(\pi_3)$ is the conductor of $\pi_3$ as in \S\ref{localTate}.   We choose
$v_i\in\pi_i$ for $i=1,3$ such that $v_i$ are analytic
newvectors, in the sense of \S\ref{sec13}, i.e.\ $v_i$ in their
Kirillov models (with conjugate additive characters of $N$) are
given by a fixed bump functions
in $C^\infty_c(\R^\times)$ sufficiently concentrated around
$1$. We denote by $W_i$ the images of $v_i$ in their Whittaker models for $i=1,3$.

We note that
$$\int_{x\in\R}f_2(n'(x))dx\asymp 1,\quad \int_{y\in\R^\times}W_1(a(y))W_3(a(y))\chi_{-1/2}(y)d^\times y \asymp 1.$$ 
We normalize $v_1,v_2',v_3$ so that both of the above integrals are $1$.

\begin{lemma}\label{main-lemma-lower-bound}
Let $\pi_1$ be $\vartheta$-tempered  with $\vartheta<1/2$ and $\pi_3$ be tempered. Let $\chi$ with $\Re(\chi)\ge 0$ 
%be such that either $\Re(\chi)=0$ or $\Im(\chi)=0$
and $f_2\in \Ic(\chi)$ be as chosen above. Then for $C(\pi_3)= Q$ sufficiently large we have 
$$\Psi(W_1,f_2(.a(Q)),W_3)\gg_{\pi_1,\pi_2} Q^{-1/2+\Re(\chi)}.$$ 
%for  $Q$ sufficiently large. 
\end{lemma}

\begin{proof}
We write the zeta integral with the Iwahori coordinates and change variables to obtain
\begin{multline}\label{display1}
\chi^{-1}_{1/2}(Q)\Psi(W_1,f_2(.a(Q)),W_3)\\=\chi^{-1}_{-1/2}(Q)\int_{y\in\R^\times}\int_{x\in\R}W_1\Big(a(y)n'\Big(\frac{x}{Q}\Big)\Big)W_3\Big(a(y)n'\Big(\frac{x}{Q}\Big)\Big)\\
\times f_2\big(a(yQ)n'(x)\big)dx\frac{d^\times y}{|y|}.
\end{multline}
Note that
$$f_2\big(a(yQ)n'(x)\big)=\chi_{1/2}(yQ)f_2(n'(x)),$$
and the support condition of $\Phi$ confirms that $f_2(n'(x))$ is supported in a sufficiently small neighbourhood of $0$. We rewrite the right hand side of \eqref{display1} as
\begin{equation}\label{integral1}
\begin{split}
  &  \int_{y\in\R^\times}\int_{x\in\R}\left(W_3\Big(a(y)n'\Big(\frac{x}{Q}\Big)\Big)-W_3(a(y))\right)\\& \quad\quad\quad \times W_1\Big(a(y)n'\Big(\frac{x}{Q}\Big)\Big)f_2(n'(x))\chi_{-1/2}(y)dx\,d^\times y\\
 &   +\int_{y\in\R^\times}\int_{x\in\R}\left(W_1\Big(a(y)n'\Big(\frac{x}{Q}\Big)\Big)-W_1(a(y))\right)\\&\quad\quad\quad \times W_3(a(y))f_2(n'(x))\chi_{-1/2}(y)dx\,d^\times y\\
&    +\int_{y\in\R^\times}W_1(a(y))W_3(a(y))\chi_{-1/2}(y)d^\times y\int_{x\in\R}f_2(n'(x))dx. 
    \end{split}
\end{equation}
Note that the third integral equals one by the choice of normalizations of the vectors.

As $\pi_1$  is fixed, we may apply  \eqref{eqn:unram-W-bound-near-1} to conclude that 
$$W_1\big(a(y)n'(x/Q)\big)\ll_N \min\big(|y|^{1/2-\vartheta},|y|^{-N}\big)$$
for $x \ll 1$. Moreover, given some sufficiently small constant $c> 0$,   then for $x$   sufficiently small (in terms of $c$ only)  and $C(\pi_i)\leq Q$ we have 
\begin{equation}\label{diff}
W_i(a(y)n'(x/Q))-W_i(a(y))\ll \frac{C(\pi_i)}{Q}c|y|^{1/2-\theta-\eta},
\end{equation}
for any $\eta>0$ and $\theta=\vartheta,0$ if $i=1,3$, respectively.
This estimate is essentially contained in   \cite[\S 2.1]{JN}
and can be seen as follows:  using Mellin inversion for
$W_i(a(y)n'(x/Q))|y|^{-\sigma}$ for $\sigma=1/2-\theta-\eta$ and
the $\PGL(2)\times \GL(1)$
local functional equation, 
we write the above difference (as in \cite[\S 2.1]{JN}\footnote{In that paper authors took $y=1$.}) 
%\vb{can you use the notation from Section 2.2 for this?}\sj{changed}
\begin{multline*}
\int_{\Re(\chi')=0}\chi'_{\sigma}(y)\frac{\chi'_{\sigma}(C(\pi_i))}{\gamma(1/2-\sigma,\pi_i\otimes{\chi'}^{-1})}\\
\times \int_{t\in\R^\times}\left(e\left(-\frac{txC(\pi_i)}{Q}\right)-1\right)W_i\big(a(C(\pi_i)t)w\big)\chi'_{\sigma}(t)\, d^\times t\, d\chi',
\end{multline*}
%\vb{This is a bit subtle, as we need sufficient $s$-decay in the $t$-integral, but I guess the details are in \cite{JN}} \sj{Modified the argument which is now more aligned with \cite{JN} so that we don't have to bring a  contour shifting argument} The rest   follows similarly as is \cite[\S 2.1]{JN}.
and \eqref{diff} follows as in \cite[\S2.1]{JN}.

Now we define 
$$I_2(\chi):=\int_{x\in\R} \left|f_2(n'(x))\right|dx\leq \int_{(t,x)\in\R^2}\Phi(tx,t)|t|^{2\Re(\chi)}dt\,dx\ll 1.$$
The first integral in \eqref{integral1} is therefore
$$\ll c I_2(\chi)\int_{y\in\R^\times}|y|^{1/2-\eta}\min (|y|^{1/2-\vartheta-
\eta},|y|^{-N})|y|^{\Re(\chi)-1/2}d^\times y\ll c,$$
as $\Re(\chi)\ge 0$ and $\vartheta<1/2$.
The second integral in \eqref{integral1} is similarly
$$\ll Q^{-1}I_2(\chi)\int_{y\in \R^\times}W_3(a(y))|y|^{1/2-\vartheta}|y|^{\Re(\chi)-1/2}d^\times y\ll Q^{-1}.$$
Thus we estimate
$$\Psi(W_1,f_2(.a(Q)),W_3)\gg Q^{-1/2+\Re(\chi)}(1+O(c)+O(1/Q))$$
which concludes the proof upon choosing $c$ sufficiently small in terms of the implied constants. 
\end{proof}

\begin{theorem}\label{lower-bound}
Let $\pi_i$ for $i=1,2,3$ be generic irreducible unitary representations of $G$ such that $\pi_1$ and $\pi_2$ are $\vartheta$-tempered with $\vartheta<1/4$
%\vb{``a bound'' is too imprecise - if you need $\vartheta < 1/2$, then the old version stating this explicitly was fine} \sj{changed}
and $\pi_3$ is tempered with sufficiently large  conductor $Q$.  The (smooth) vectors $v_i\in\pi_i$ specified at the beginning of \S
\ref{choice}
have the following properties:
\begin{enumerate}[(i)]
\item $\|v_i\| \asymp 1$
  for $i=1,2,3$.
\item We have $v_1 = v_1^0$
  and $v_2 = a(Q) v_2^0$
  where $v_1^0, v_2^0$ are fixed (independent of
  $Q$).
\item
  For any fixed nontrivial unitary character $\psi$ of $N$,
  the vector
  $v_3$ is given
  in the $\psi$-Kirillov model
  by a fixed bump function.
\item We have
$$\int_{g\in G}\prod_{i=1}^3\langle \pi_i(g)v_i,v_i\rangle dg \gg_{\pi_1,\pi_2} Q^{-1}.$$
%if $C(\pi_3)\ll X$.
\end{enumerate}
\end{theorem}

\begin{proof}
  Assertions (i), (ii),  and (iii) are clear from the construction.
  (The description of $v_3$ in the Kirillov model is independent
  of the choice of $\psi$: different choices give
  rise to models that are isomorphic to one another
  via left translation by a suitable diagonal matrix.)

  To verify (iv),
  we embed $\pi_2\hookrightarrow \pi_2^p=\Ic(\chi)$ for some $\chi$ with
  $\Re(\chi)\ge 0$ such that either $\Re(\chi)=0$ or $\Im(\chi)=0$. 
  For $\Re(\chi)=0$, the integral in question evaluates to
$$c(\chi)|\Psi(W_1,f_2(.a(Q)),W_3)|^2$$
by Lemma \ref{extended-ichino},
and Lemma \ref{main-lemma-lower-bound} implies the required bound. 
For $\Im(\chi)=0$ and $\Re(\chi)>0$,
we apply Lemma \ref{extended-ichino} and Lemma
\ref{normalized-lfe}
to see that the integral in question is
$$c(\chi)\overline{\gamma(1/2,\pi_1\otimes\pi_3\otimes\chi)}|\Psi(W_1,f_2(.a(Q)),W_3)|^2.$$
An appeal to   \eqref{bound-gamma-factor} and Lemma
\ref{main-lemma-lower-bound}
then completes the proof. %implies that the above is $\gg Q^{-1}(Q/C(\pi_3))^{2\Re(\chi)}$. This completes the proof in view of  $X\gg C(\pi_3)$ and $\Re(\chi)>0$.
\end{proof}

\section{A triple Whittaker integral}\label{whittaker}

\subsection{Setting and statement of results} In this section we
evaluate asymptotically an integral containing three Whittaker
functions which is the crucial ingredient for an understanding
of the right hand side of \eqref{period}.
{We will not need any knowledge on special functions, but we do use extensively Stirling's formula and stationary phase analysis as described in Sections \ref{localTate} and \ref{sec25}.}

We retain the
basic notation of \S\ref{sec:basic-notation}. We continue to adopt the following setting (as in the previous
section):
 \begin{itemize}
\item  $\pi_1$ and $\pi_2$ are fixed $\vartheta$-tempered
  generic irreducible unitary representations of $G$, 
  %\pn{added}
  with $0 \leq \vartheta < 1/2$ fixed. 
\item  $\pi_3$ is a varying tempered generic irreducible
  unitary
  representation of $G$, whose analytic conductor
  (normalized
  as in \S\ref{localTate})
  we
  denote by $Q := C(\pi_3)$. %\sj{is unitarity assumption dropped
 %   purposefully from this and the previous points?}
%  \pn{I added (perhaps re-added) ``unitary'' just now}
\item
  Recall that we realize each $\pi_j$ in its Whittaker model as a space of
  functions $W$ satisfying $W(n(x) g) = e(x) W(g)$, where $e(x)
  := e^{2 \pi i x}$.
  Also, recall from \eqref{defn-inner-prod} that we normalize this realization
  so that the inner product
  on $\pi_j$ is given in the Kirillov model
  by integration over the diagonal subgroup:
  $\|W\|^2 = \int_{y \in \mathbb{R}^\times} |W(a(y))|^2 \, d^\times y$. %\sj{probably now this is redundant given the discussion in the beginning of \S3}  
\item We let $W_j \in \pi_j$
  be the image of the vector
  $v_j$
  as in Theorem \ref{lower-bound}.
  Thus $W_1 = W_1^0$ and $W_2 = a(Q) W_2^0$
  with $W_1^0, W_2^0$ fixed
  (independent of $Q$).
\end{itemize}
The basic bounds
from Lemma \ref{lem:unram-W-bounds} can be used for $W_1 = W_1^0$ and $W_2^0$. We will derive useful bounds for $W_3$ in \S \ref{sec-newvector} below. 

We recall the notation and conventions of smooth weight
functions in \S \ref{weight}. Our basic large parameter here is
$Q$, so $A \preccurlyeq B$ means $A \ll_{\eps}
Q^{\eps}B$. As usual, the value of $\eps$  may
change from line to line.

For $y_1, y_2 \in \mathbb{R}^\times$ with $y_1 + y_2
\neq 0$,
we define $y_3 \in \mathbb{R}^\times$
by requiring that
\begin{equation}\label{eqn:y1-y2-y3-sum-zero}
  y_1 + y_2 + y_3 = 0.
\end{equation}
We set
\begin{equation}\label{defF}
  F(y_1, y_2) :=
  \int_{\theta \in \mathbb{R} / \pi \mathbb{Z}}
%  \Big(
    \prod_{j=1,2,3}
    W_j(a(y_j) k(\theta))
%  \Big)
\,   d \theta.
\end{equation}

The main result of this section is the following estimate for
$F$ and its derivatives.
\begin{theorem}\label{thm:triple-whittaker}
  We have
  \[
    F(y_1,y_2) =     \sum_{\pm}
   e\Big(\pm 2 \sqrt{Q} \Psi \Big(\frac{y_1}{y_2}\Big)\Big)
    \mathcal{N}_{\pm}(y_1,y_2) + \mathcal{E}(y_1,y_2),
  \]
 where $\Psi$ is a smooth function satisfying the estimates
 \begin{equation}\label{psi}
 \Psi(y) = |y|^{1/2} + O(|y|^{3/2}), \quad \Psi^{(j)}(y) \ll |y|^{1/2 - j}\quad (j \in \Bbb{N})
 \end{equation}
 and for fixed $j_1, j_2, N \geq 0$ we have 
    \[
    (y_1 \partial_{y_1})^{j_1} (y_2 \partial_{y_2})^{j_2}
    \mathcal{N}_{\pm}(y_1, y_2)
    \preccurlyeq 
     \Big( \frac{|y_2|}{Q}\Big)^{3/4} (1 + |y_1|)^{-N} \Big(1 + \frac{|y_2|}{Q}\Big)^{-N}
  %    \frac{      (Q/|y_2|)^{-3/4}
  %    }{\max(|y_1|^{-1},|y_1|)^N \max(Q^{1/3}/|y_2|,|y_2|/Q)^N}
  \]
 % \pn{strengthened $3/4$ exponent to $1$?  just checking}
  and $\mathcal{E} = \mathcal{E}_1 + \mathcal{E}_2 +
  \mathcal{E}_3$,
  where for some absolute constant $c > 0$ we
  have 
 % \pn{changed $\mathcal{E}$ to $\mathcal{E}_1$
 %   in first line}
 \begin{displaymath}
 \begin{split}
   & \mathcal{E}_1(y_1,y_2) =  Q^{-1/2}(1+|y_1| + |y_2|)^{-N},\\
    & \mathcal{E}_2(y_1, y_2) = 
%(1+|y_2|)^{-N} 
  Q^{c}(1+|y_1|+ |Qy_2|)^{-N},\\
  & \mathcal{E}_3(y_1, y_2) =  (1 + |y_1| + |y_2/Q|)^{-N} Q^{-N}. 
%\frac{1}{Q^{1/3 }} (1+|y_1|)^{-N}
 %  \Big(1 + \frac{|y_2|}{Q^{1/3 }}\Big)^{-N}.
\end{split}
\end{displaymath}
\end{theorem}
%As mentioned above, a change of $Q$ by a multiplicative factor $\asymp 1$ can be absorbed into the function $\mathcal{N}_{\pm}$ and does not affect the statement of the theorem. 
Here and in the following, all implied constants may depend on $N$ and $j$, with or without subscript. The proof gives an exact formula for $\Psi$ which however is irrelevant for our application. It depends mildly on whether $\pi_3$ belongs to the principal series or to the discrete series. {The key point of Theorem \ref{thm:triple-whittaker} is that it produces the desired and expected oscillatory factor, cf.\ \eqref{alternative}.}

\subsection{Preliminary decomposition}
It will be convenient first to switch to Iwahori coordinates.
We may find an even function
$\phi_0 \in C_c^\infty(\mathbb{R} / \pi \mathbb{Z})$, 
supported on   $(-\pi/3,\pi/3) + \pi \Bbb{Z}$,
so that $\phi_0(\theta) + \phi_0(\pi/2 + \theta) = 1$ for all
$\theta$.
Setting $z := \tan \theta$
and using the matrix identities
\[
  k(\theta) = n(-z) a(z^2 + 1) n'(z),
\]
\[
  k(\pi/2 - \theta) =
  n(z) a(z^2 + 1) w
  n(z)
\]
and the relation $d \theta = (1 + z^2)^{-1} \, d z$,
we see that
$F(y_1,y_2)
=
F_1(y_1,y_2) + F_{w}(y_1,y_2)$,
where,
with the abbreviation $y_j' := y_j (z^2 + 1)$, we have  
%\vb{can we agree on a consistent notation for integrals? I am happy either way, but it should be the same in all sections.} 
\[
  F_1(y_1,y_2)
  :=
  \int_{z \in \mathbb{R}}
  \Big(
    \prod_{j=1,2,3}
    W_j(a(y_j') n'(z))
  \Big)
  \phi_0(\arctan(z)) \, \frac{d z}{z^2 + 1},
\]
\[
  F_{w}(y_1,y_2)
  :=
  \int_{z \in \mathbb{R}}
  \Big(
    \prod_{j=1,2,3}
    W_j(a(y_j') w n(z))
  \Big)
  \phi_{0}(\arctan(z)) \, \frac{d z}{z^2 + 1}.
\]
Note that $z \mapsto \phi_0(\arctan(z))$ defines a smooth
compactly-supported function on $\mathbb{R}$.

We now further decompose $F_1$.
%We fix $\eps_1 > 0$ sufficiently small in terms of $\eps$
%(in fact, $\eps_1 = \eps/2$ would suffice).
We write $1$ as a sum
$\phi^\sharp + \phi^{\flat} $ of smooth
functions on $\mathbb{R}^\times$ with
\begin{itemize}
\item  $\phi^{\flat}(z)$
  supported on $z \ll Q^{-1}$,
\item  $\phi^{\sharp}(z)$ supported on
  $Q^{-1} \ll z $, 
\end{itemize}
and with each function $\phi$ in this
decomposition satisfying $(z \partial_z)^j \phi(z) \ll 1$.  We
accordingly decompose
$F_1 =  F_1^{\flat} + F_1^{\sharp}$ by
weighting the $z$-integral. In summary,
we have decomposed
\begin{equation}\label{eqn:F-equals-F-w-F-1-etc}
  F = F_w +  F_1^{\flat} + F_1^{\sharp}.
\end{equation}
We will verify that each of the three terms
on the
right hand side of
\eqref{eqn:F-equals-F-w-F-1-etc}
satisfies the conclusions of
Theorem \ref{thm:triple-whittaker}.
The first two terms are fairly straightforward
to analyze.
We treat them in
\S\ref{sec:trip-whit-prel-estim}. {Indeed, we will  see that the contribution of $F_1^{\flat}$ can be absorbed into the error term $\mathcal{E}$, while $F_w$ is non-negligible only if $y_1 \preccurlyeq 1$ and $y_2$ is roughly of size $Q$, in which case the oscillatory factor $e(\pm Q^{1/2} \Psi(y_1/y_2))$ is flat in the sense of \S \ref{weight}.} 
The somewhat more intricate analysis of $F_1^\sharp$ is
carried out in \S\ref{sec:estimates-f_1sharp}.

\subsection{Interlude: bounds for newvectors}\label{sec-newvector}

For the relevant asymptotic analysis we will need certain uniform bounds of the test vectors. We record them here.

\begin{lemma}
  Let $W \in \pi$,
  $y \in \mathbb{R}^\times$, $z \in \mathbb{R}$
  and $\sigma > -1/2+\vartheta$.
  Then $W(a(y) w n(z))$ admits the absolutely
  convergent integral representation
  \begin{equation}\label{eqn:mellin-expansion-W-aywnz}
    W(a(y)wn(z))=
    \int_{\Re(\chi) = \sigma } \chi(y)
    \gamma(1/2, \pi \otimes \chi)
    \left( 
      \int_{t\in\mathbb{R}^\times}
      e(t z)
      W(a(t))
      \chi(t)
      \,
      d^\times t
    \right)
    \,
    d \chi.
  \end{equation}
\end{lemma}
\begin{proof}
  By Mellin inversion \eqref{eqn:mellin-inversion} -- using
  the estimate \eqref{eqn:unram-W-bound-near-w} to
  verify its hypotheses -- we have
  $$W(a(y)wn(z))=
  \int_{\Re(\chi) =\sigma} \chi(y) \int_{t\in\mathbb{R}^\times}
  W(a(t)wn(z)) \chi^{-1}(t) \, d^\times t
  \, d\chi.$$
  By the $\mathrm{PGL}(2)\times\mathrm{GL}(1)$ local 
  functional equation (see \cite[Theorem 3.1]{Cog}), 
  %\vb{add reference, adjust measure} \sj{done}
  the inner integral evaluates to
  $$
  \gamma(1/2, \pi \otimes \chi)
  \int_{t\in\mathbb{R}^\times}
  W(a(t)n(z))
  \chi(t)
  d^\times t.$$
  We conclude by calculating that
  $W(a(t) n(z))
  = W(n(t z) a(t)) = e(t z) W(a(t))$.
\end{proof}

 \begin{lemma}\label{lem:newvector-W-bounds}
  \label{lem:bound-W-a-w-n-small-z} Let $\pi$ be tempered of conductor $Q$ and $W\in \pi$ an analytic newvector. 
  For  fixed $N >0$ and  $j \in \Bbb{Z}_{\geq 0}$ 
  %\vb{I use $\Bbb{Z}_{\geq 0}$ throughout. I am happy either way, but let's be consistent throughout the paper.} 
  we have 
  \begin{equation}\label{eqn:newvector-W-bound-1}
    (x
    \partial_{x})^j
    W\big(a(Q x) w n(z)\big)
    \preccurlyeq  \min\big(|x|^{-N}, |x|^N +
      |x|^{1/2}Q^{-N}\big)
      \end{equation}
   uniformly in $z \preccurlyeq 1$. 
\end{lemma}
\begin{proof}
  We write the proof
  in the case
  $j=0$.  The
  general case is treated similarly, using that
  $(y \partial_y)^j \chi(y) = s^j \chi(y)$ for $\chi = |.|^s \sgn^a$. %  We start with the first estimate
  %\eqref{eqn:newvector-W-bound-1}.
  
  By
  \eqref{eqn:mellin-expansion-W-aywnz},
  we have
  the Mellin expansion  
  \[
    W(a(Qx) w n(z))
    = \int _{\chi }
    \chi (Qx) \gamma (1/2, \pi \otimes \chi)
    V_z(\chi) \, d \chi,
  \]
  where
  \[
    V_z(\chi) := \int_{t\in\R^\times}
    e(t z) W(t) \chi(t) d^\times t.
  \]
  Since $z \preccurlyeq 1$ and $W$ is supported in a fixed compact set,
  we see that $V_z$ is entire and $V_z(\chi) \preccurlyeq C(\chi)^{-N}$  in vertical
  strips.  By Lemma \ref{lem:gamma-s-rho-via-conductors} and \cite[Lemma A.2]{HB}
    we have %We recall the
  %Stirling estimate
  %\vb{Subhajit writes $\gamma(s, \pi)$ instead of $\gamma(\pi, s)$, this should be synchronized}
  $$
  \gamma(1/2, \pi \otimes \chi)  \ll
  C(\pi \otimes \chi)^{-\Re(\chi)}
  \ll
  C(\pi)^{-\Re(\chi)}
  C(\chi)^{2 |\Re(\chi)|},
  % (Q + C(\chi)^2)^{-\Re   (\chi)},
  $$ 
  which holds uniformly in any fixed vertical
  strip and for any $\chi$ separated by $\gg 1$ from any pole of
  % $L_{\infty}(\pi \otimes \chi, 1/2)$ or
  $\gamma(1/2,\pi \otimes
  \chi)$ 
    (In more detail,
  we apply
  Lemma
  \ref{lem:gamma-s-rho-via-conductors}
  by writing $\chi = \chi_0|.|^{\Re(\chi)}$,
  with $\chi_0$ unitary,
  and using that
  $\gamma(1/2, \pi \otimes \chi)
  = \gamma(1/2 + \Re(\chi), \pi \otimes \chi_0)$
  and $C(\pi \otimes \chi_0) \asymp C(\pi \otimes \chi)$.) 
   We obtain an adequate estimate in the case $|x| \geq 1$
  by shifting the contour to $\Re(\chi) =
  -N$,
  passing no poles.

  It remains to consider the case $|x| \leq 1$.
  We shift the contour to $\Re(\chi) = N$.
  Recall that $Q$ is assumed sufficiently large in terms of ``fixed'' quantities;
  in particular, $Q$ is much larger than any fixed power of $N$. 
  It follows that if $\pi$ belongs to the discrete series,
  then this contour shift passes no poles.
  The required estimate follows (in the stronger
  form obtained by omitting the term
  $Q^{-N} |x|^{1/2}$).
  Suppose now that $\pi$ belongs to the principal series.
  By hypothesis, $\pi$ is tempered.
  Thus each pole that we cross
  is of the form $\chi = \sgn^a |.|^{\sigma + i t}$ with
  $\sigma \geq 1/2-\vartheta$ and  $t \asymp Q^{1/2}$.
  The required estimate follows from the rapid decay of $V_z$.
 \end{proof}

\subsection{The easy cases: estimates for $F_w,  F_1^{\flat}$}\label{sec:trip-whit-prel-estim}

\begin{prop}
  \label{prop12}
  The function $F_w$ satisfies
  the conclusions of
  Theorem \ref{thm:triple-whittaker}.
\end{prop}
\begin{proof}
  By the matrix identity $$a(y_2') w n(z) a(Q) = a ( y_2'/Q) w
  n(z/Q) \in G$$
  and the relations $W_1 = W_1^0, W_2 = a(Q) W_2^0$,
  we write $ F_w(y_1,y_2)$ as
   \begin{multline*}
      \int_{z \in \mathbb{R}}
   W_1^0\big(a(y_1' )w n(z)\big)
    W_2^0\Big(a\Big(\frac{y'_2}{Q}\Big) wn\Big(\frac{z}{Q}\Big)\Big)\\ \times 
    W_3\big(a(y_3') w n(z)\big)
    \phi_{0}(\arctan(z)) \, \frac{d z}{z^2 + 1}.
  \end{multline*}
   We use %\eqref{eqn:unram-W-bound-near-1} and
  \eqref{eqn:unram-W-bound-near-w} to bound the first two factors of the integrand as    
  \begin{multline}\label{eqn:bounds-for-V-prop-F-w}
  %\begin{split}
   (y_1 \partial_{y_1})^{j_1}
    (y_2 \partial_{y_2})^{j_2}
    W_1^0\big(a(y_1' )w n(z)\big)
    W_2^0\Big(a\Big(\frac{y'_2}{Q}\Big)w n\Big(\frac{z}{Q}\Big)\Big) \\     \ll
    (1 +  |y_1|)^{-N}  \Big(1 +  \frac{|y_2|}{Q}\Big)^{-N} 
 %   \frac{|x_j|^{1/2-\vartheta}}{(1 + |x_j|)^N}
 %\end{split}
  \end{multline}
uniformly for $z \ll 1$. We bound the third factor by \eqref{eqn:newvector-W-bound-1} %Lemma \ref{lem:bound-W-a-w-n-small-z} 
with  $x = y_3'/Q$ getting
$$ (y_3 \partial_{y_3})^{j}  W_3\big(a(y_3') w n(z)\big) \ll \min\Big(\Big|\frac{y_3}{Q}\Big|^{-N}, \Big|\frac{y_3}{Q}\Big|^N + Q^{-N}\Big).$$
Recalling \eqref{eqn:y1-y2-y3-sum-zero}, it is easy to see that 
  %  (y_2 \partial_{y_2})^{j_2}F_w(y_1,y_2)
$$
(y_1 \partial_{y_1})^{j_1}
    (y_2 \partial_{y_2})^{j_2}F_w(y_1,y_2) \ll   (1 + |y_1|)^{-N}  \min\Big(\Big|\frac{y_2}{Q}\Big|^{-N},  \frac{1+|y_2|^N}{Q^N}   \Big). $$
  If $\Psi$ satisfies \eqref{psi}, then $e( \pm2 \sqrt{Q\Psi(y_1/y_2)})$ is flat for $y_1 \preccurlyeq 1$ and $Q^{1-\eps} \ll |y_2| \ll Q^{1+\eps}$, so up to a contribution that can go into $\mathcal{E}_3$  we have   %    and therefore (again by \eqref{eqn:y1-y2-y3-sum-zero}) also
%   \begin{displaymath}
  % \begin{split}
  
\begin{multline*}
(y_1 \partial_{y_1})^{j_1}
    (y_2 \partial_{y_2})^{j_2}e( \pm2 \sqrt{Q\Psi(y_1/y_2)}) F_w(y_1,y_2) \\
    \preccurlyeq   (1 + |y_1|)^{-N}  \min\Big(\Big|\frac{y_2}{Q}\Big|^{-N},  \frac{1+|y_2|^N}{Q^N}   \Big)
    \end{multline*}
   which is admissible for Theorem \ref{thm:triple-whittaker}. 
\end{proof}

\begin{prop}
 The function  $F_{1}^{\flat}$ satisfies
  the conclusions of
  Theorem \ref{thm:triple-whittaker}.
\end{prop}
\begin{proof}
   
  By definition,
  \begin{equation*}%\label{eqn:}
    F_1^{\flat }(y_1,y_2)
    =
    \int_{z \in \mathbb{R}}
    \Big(
      \prod_{j=1,2,3}
      W_j\big(a(y_j') n'(z)\big)
    \Big)
    \phi^{\flat}(z) \phi_0(\arctan(z)) \, \frac{d z}{z^2 + 1}.
  \end{equation*}
  For $|z| \ll 1/Q$,
  we have
  \begin{equation*}\label{eqn:bounding-F1-W1}
  \begin{split}
  %  (y_1 \partial_{y_1})^j
 &   W_1(a(y_1') n'(z)) \ll   (1 +  |y_1|)^{-N}, \\
   %  (y_2 \partial_{y_2})^j
&    W_2\big(a(y_2') n'(z)\big)  = 
    W^0_2\big(a(Qy_2') n'(Qz)\big) \ll  (1 +  |y_2| Q)^{-N},\\
  %  \frac{(|y_2| Q)^{1/2-\vartheta}}{(1+|y_2|Q)^N}, 
% &    W_3(a(y_3') n'(z)) \ll \min\big(|y_3|^{-N}, |y_3|^N +Q^{- N}\big),
  %  \frac{1}{\max(|y_3|^{-1},|y_3|)^N}
  %  + \frac{Q^{-N} |y_3|^{1/2}}{(1 + |y_3|)^N}
  \end{split}
  \end{equation*}
  where  we used   \eqref{eqn:unram-W-bound-near-1}.
  %\pn{added:}
  The required estimate
  $F_1^{\flat } \ll \mathcal{E}_2 \leq \mathcal{E}$
  follows now from the weak \emph{a priori}  bound
  $\|W_3\|_{\infty}  \ll Q^{O(1)}$.
  
  To deduce the latter
  it suffices by the Iwasawa
  decomposition to estimate
  $W_3(a(y) k(\theta))$.
  We appeal first
  to
  % \cite[\S3.2.3]{MV} \sj{maybe refer to
  Lemma \ref{lem:unram-W-bounds},
  % instead?},
  which gives % (in the notation of \emph{loc.\ cit.})
  for any $W \in \pi_3$
  the estimate
  $W(a(y)) \ll \mathcal{S}_d(W)$
  for some fixed  $d$.
  Taking $W = k(\theta) W_3$,
  we obtain
  $W(a(y)) = W_3(a(y) k(\theta)) \ll \mathcal{S}_d(W) = \mathcal{S}_d(W_3)$.
  On the other hand,
  by
  \cite[\S3.2.5]{MV} and the fact
  that $y \mapsto W_3(a(y))$ is a fixed bump function,
  we have
  $\mathcal{S}_d(W_3) \ll  Q^{d'}$
  for some fixed $d'$.
  The required \emph{a priori} bound follows.
  \end{proof}

\subsection{The critical case: estimates for $F_1^{\sharp}$} \label{sec:estimates-f_1sharp}
\begin{prop}
 The function  $F_1^\sharp$ satisfies the conclusions of Theorem \ref{thm:triple-whittaker}.
\end{prop}
\begin{proof}
Let $Q^{-1} \ll z$. We use the identity
  \begin{equation*}%\label{eqn:matrix-identity-for-n-prime-of-z}
    a(y) n'(z) a(Q) = n(y/z) a(y/Qz^2) w n(1/Qz)\in G %n(1/z) a(1/z^2) w n(1/z)
  \end{equation*}
  to write
  \[
    W_2(a(y_2') n'(z))
    = e(y_2'/z) W_2^0\big(a(y_2'/Q z^2) w n(1/Q z)\big).
  \]
Similarly, we have 
\begin{align*}
    W_3(a(y_3') n'(z))
    &=
      W_3(n(y_3'/z) a(y_3'/z^2) w n(1/z)) \\
    &=  e(y_3'/z) W_3(a(y_3'/z^2) w n(1/z)). 
  \end{align*}
  Using the consequence $e(y_2'/z) e(y_3'/z) = e(-y_1'/z)$
  of the hypothesis \eqref{eqn:y1-y2-y3-sum-zero},
  we may
  write
  $F_1^\sharp(y_1,y_2)$
  in the form
 % \pn{changed $y_1$ to $y_1(z^2+1)$
 %   and $W_3(\dotsb n'(z))$
 % to $W_3(\dotsb w n(1/z))$}
  %\[
  \begin{multline*}
    \int_{z \in \mathbb{R}}
    W \left(y_1,\frac{y_2}{Q z^2}, z \right)
    e \left(-\frac{y_1 (z^2 + 1) }{ z} \right)\\
    \times W_3\Big(a\Big(\frac{y_3'}{z^2}\Big) w n\Big(\frac{1}{z}\Big) \Big)
    \phi^\sharp(z)
    \phi_0(\arctan(z)) \, \frac{d z}{z^2 + 1}
    \end{multline*}
  %\]
  where
  \begin{equation*}%\label{eqn:}
    W(x_1,x_2, z)
    :=
    W_1^0\big(a(x_1(1 + z^2)) n'(z)\big)
    W_2^0\big(a(x_2(1+z^2)) w n(1/Q z)\big).
  \end{equation*}
 
  From now on we will use extensively the conventions on smooth weight functions stated in \S \ref{weight}. In particular, $V$ denotes generally a flat function, \emph{not necessarily the same at every occurrence}.  
  
    We apply the substitution $z \mapsto 1/z$ and a
  smooth dyadic partition of unity localizing $\pm z \asymp Z$
  where $Z$ runs over $\preccurlyeq 1$ values (e.g.\ powers of
  2) satisfying $1 \ll Z \ll Q.$  For notational simplicity we
  restrict to $z > 0$, the case $z < 0$ being essentially
  identical.  Since $W$ is flat in $z$,  it suffices to estimate 
     \begin{equation}\label{eqn:suffice}
  \begin{split}
    %V_1(y_1)
& I_0(y_1, y_2, y_3) =  W \left(y_1,  \frac{y_2 Z^2}{Q} \right)\\
&\times 
    \int_{z \in \mathbb{R}}
    e \left(-y_1 (z + z^{-1}) \right)    W_3\big(a(y_3 {(1 +
    z^{2})}) 
    w n(z)\big)
    V \left( \frac{z}{Z} \right)
    \, \frac{d z}{|z|^2}
    \end{split}
  \end{equation}
  where $W$ satisfies the estimates
  \begin{equation}\label{boundw}
    (x_1 \partial_{x_1})^{j_1} (x_2 \partial_{x_2})^{j_2} W(x_1, x_2) \ll \prod_{j=1}^2(1 + |x_j|)^{-N}.  %\frac{|x|^{1/2-\vartheta}}{(1 +
    %  |x|)^N}.
  \end{equation}
  At this point we can easily deal with the contribution $Z \gg Q^{1/2 - \eps}$. A trivial estimate returns the bound
  $$\preccurlyeq Q^{-1/2}(1 + |y_1|)^{-N}  (1 + |y_2|)^{-N} = \mathcal{E}_1(y_1, y_2)$$
  which is acceptable. %From now on we assume $Z \ll Q^{1/2 - \eps}$. 
  
  We can also easily deal with the contribution $Z \preccurlyeq 1$. In this case, \eqref{eqn:newvector-W-bound-1} yields
  $$(y_3 \partial_{y_3})^j W_3\big(a(y_3{ (1 +
    z^{2})}) 
    w n(z)\big) \ll \min\left(\Big( \frac{|y_3|}{Q}\Big)^N + Q^{-N}, \Big( \frac{|y_3|}{Q}\Big)^{-N}\right).
$$
 Recalling \eqref{eqn:y1-y2-y3-sum-zero}, we argue as in the proof of Proposition  \ref{prop12} that under the present assumption $Z \preccurlyeq 1$, up to an error of size $\mathcal{E}_3$, we have
 $$(y_1 \partial_{y_1})^{j_1} (y_2 \partial_{y_2})^{j_2} I_0(y_1, y_2, y_3) \preccurlyeq (1 + |y_1|)^{-N} \min\left(\Big|\frac{y_2}{Q}\Big|^{-N},  \frac{1+|y_2|^N}{Q^N}   \right)$$
 but then also 
\begin{multline*}
(y_1 \partial_{y_1})^{j_1} (y_2 \partial_{y_2})^{j_2}e( \pm2 \sqrt{Q\Psi(y_1/y_2)}) I_0(y_1, y_2, y_3)\\ \preccurlyeq (1 + |y_1|)^{-N} \min\left(\Big|\frac{y_2}{Q}\Big|^{-N},  \frac{1+|y_2|^N}{Q^N}   \right)
\end{multline*}
  as required for the bound for $\mathcal{N}_{\pm}$. So from now on we assume
  $$Q^{\eps} \leq Z \leq Q^{1/2 - \eps}.$$

  We appeal to the (rapidly-convergent) integral formula
  \eqref{eqn:mellin-expansion-W-aywnz}    to rewrite the  integral over $z$ in \eqref{eqn:suffice} 
  as 
  \begin{multline}\label{eqn:defn-I-trip-whit}
    I :=
    \int_{\Re(\chi)=0}
    \gamma(1/2, \pi_3 \otimes \chi)\\
    \times\int_{z \in \mathbb{R}}
    \chi(y_3  ( z^2 + 1))
    e \left( - y_1 (z  +z^{-1}) \right)
  V\left( \frac{z}{Z} \right)
    I_1(z,\chi)
    \, \frac{d z}{|z|^2}
    \, d \chi 
  \end{multline}
  where
  \begin{equation*}%\label{eqn:}
    I_1(z,\chi) :=
      \int_{t \in \mathbb{R}}
      e(t z) \chi(t) W_3(a(t))
      \, \frac{d t}{|t|}.
    \end{equation*}
   % \sj{$W_3(t)=W_3(a(t))$?}
We keep in mind that the bound \eqref{boundw} allows us to assume $y_1 \preccurlyeq  1$, otherwise we can bound $I$ trivially by $\mathcal{E}_3$. 
  
    Let us write $\chi = |.|^{it_\chi} \sgn^{a_\chi}$
    with $t_\chi \in \mathbb{R}$
    and $a_\chi \in \{0,1\}$.
    We split the $\chi$-integral
    in \eqref{eqn:defn-I-trip-whit}
    according to the value of $a_\chi$
    and regard that value as fixed from now on.   We have expressed $I$ as a triple integral in $t, z, t_{\chi}$, and we will apply stationary phase in each of the variables, one at a time. Here $z\asymp Z$, $t \asymp 1$, and we will see in a moment that the $t$-integral is negligible unless $t_{\chi} \asymp Z$. Stationary phase saves a factor $Z^{1/2}$ in each of these integrals, so we expect that $I$ is of size $Z^{-3/2}$, and we can explicitly compute the oscillatory behavior.  
    
    \medskip

 \textbf{Step 1: the $t$-integral.}   We  apply stationary phase analysis to find, for each fixed
  $N \in \mathbb{Z}_{\geq 0}$, a nice function $V$ so that
  %  \pn{changed $V \left( \frac{-t_\chi }{z} \right)$
      %to
    %  $V \left( \frac{-t_\chi }{z}, \frac{z}{Z} \right)$}
    \begin{equation}\label{eqn:V-z-chi-estimate}
      I_1(z,\chi) =
      z^{-1/2}
      V \left( \frac{-t_\chi }{z}, \frac{z}{Z} \right)
      \chi \left( \frac{t_\chi }{2 \pi e z} \right)
      + O(Q^{-N}).
    \end{equation}
    To see this, we apply Lemma \ref{lem2} if $-t_{\chi}/z
    \asymp  1$ by taking 
      \begin{displaymath}
    \begin{split}
    & \phi(t)= \phi(t; z, t_{\chi}) =  2 \pi z t + t_\chi \log |t|, \quad t^{\ast} = -\frac{t_\chi}{2 \pi z},\\
     & (X_1, X_2, X_3) = (1, Z, Z), \quad Y = Z
     \end{split}
     \end{displaymath}
    and otherwise the integral is negligible by Lemma \ref{lem1} with $U=P=1$, $S=Y=Z$.  Here and henceforth the contribution of all negligible error terms is covered by $\mathcal{E}_3$ in the statement of Theorem \ref{thm:triple-whittaker}. 
        We deduce that $I$ is given up
    to acceptable error (that can go into $\mathcal{E}_3$) by
    \begin{equation*}%\label{eqn:}
    \begin{split}
  &    \int_{\Re(\chi)=0}
      \gamma(1/2, \pi_3 \otimes \chi)
      \chi \left( \frac{t_\chi y_3}{2 \pi e} \right)\\
      &\times 
      \int_{z \in \mathbb{R}}
      \chi(z)
      e  ( - y_1 z)  
     V \left( \frac{z}{Z} ,  \frac{-t_\chi }{Z} \right) \chi(1 + z^{-2})e(-y_1z^{-1})
      \, \frac{d z}{|z|^{5/2}}
      \, d \chi.
      \end{split}
    \end{equation*}
    Note that for $|t_{\chi}| \asymp z \geq Q^{\eps}$ and $y_1 \leq Q^{\eps}$ the function $\chi(1 + z^{-2})e(-y_1z^{-1})$ is flat in both $z$ and $t_{\chi}$, so we can incorporate this factor into $V$ (and continue to call this new function $V$). 
  %  By separation of variables,
    Thus we may reduce further to estimating
    \begin{equation}\label{eqn:the-integral-I-after-evaluating-z-integral}
      \int_{\Re(\chi)=0}
      V \left( \frac{-t_\chi }{Z} \right)
      \gamma(1/2, \pi_3 \otimes \chi)
      \chi \left( \frac{t_\chi y_3}{2 \pi e} \right)      I_2(\chi,y_1)
      \, d \chi
    \end{equation}
    with
    \begin{equation*}%\label{eqn:}
      I_2(\chi,y_1)
      :=
      \int_{z \in \mathbb{R}}
      V \left( \frac{z}{Z} \right)
      \chi(z )
      e \left( - y_1 z \right)
       \frac{d z}{|z|^{5/2}}. 
    \end{equation*}
    
    \medskip
        
    \textbf{Step 2: the $z$-integral.} By stationary phase analysis,
    we may find for each fixed $N$
    a  nice function $V $ (again potentially different from previous versions of $V$)
    so that
     \begin{equation}\label{eqn:approx-I-2-via-V-6}
      %V \left( \frac{-t_\chi }{Z} \right)
      I_2(\chi,y_1)
      =
      Z^{-2} V\Big(  \frac{t_\chi }{Z y_1} \Big)
     % V_4' \left( \frac{t_\chi }{Z} \right)
    %  V_6 \left( \frac{t_\chi }{Z y_1} \right)
      \chi \left( \frac{t_\chi }{2 \pi e y_1} \right)
      + O(Q^{-N}),
    \end{equation}
 %  $ with the same estimate valid for any desired  fixed number of $y_1$-derivatives.
 %   \pn{added:}
 whenever $-t_{\chi} \asymp Z$. 
 We apply Lemma \ref{lem2},  taking
 \begin{displaymath}
\begin{split}
& \phi(t) = \phi(t; y_1, t_{\chi}) =  - 2 \pi y_1 t + t_\chi \log|t|, \quad t^{\ast} =   \frac{t_\chi }{ 2 \pi y_1}, \\
& (X_1, X_2,X_3) = (Z,  1, Z), \quad Y = Z
\end{split}
\end{displaymath}
   (if $|y_1| \asymp 1$, and otherwise the integral is negligible by Lemma \ref{lem1} with $U=Y=P = Z$, $S = 1$). 
    Thus  we reduce to studying
      \begin{equation}\label{finalint}
      Z^{-2}
      V(-y_1)
      \int_{\Re(\chi)=0}
      V \left( \frac{-t_\chi }{Z} \right)
      \gamma(1/2, \pi_3 \otimes \chi)
      \chi \left( \frac{t_\chi^2 y_3 }{(2 \pi e)^2 y_1} \right)
      \, d \chi,
   \end{equation}
   for certain nice functions $V$. 
   % with $V_7, V_8$ nice.
  
  \medskip
  
  \textbf{Step 3: the $t_{\chi}$-integral.}  We now evaluate $\gamma(1/2, \pi_3 \otimes \chi)$ asymptotically 
    using Stirling's formula, and appeal in particular to \eqref{eqn:gamma-s-pi-approximation} 
and the subsequent explicit formul\ae.     
  %  . If $\pi_3$ is principal series, we
  %  use \eqref{maass};
  %  if it is discrete series, we use \eqref{hol}.
   Since $|t_{\chi}| \asymp Z \ll Q^{1/2 - \eps}$, both are
   applicable. We start with the principal series case.
   In view
   of \eqref{maass}, the relevant phase function
   may be written
\begin{displaymath}
\begin{split}
&\phi(t_{\chi})  = \phi(t_{\chi}; r, y_1, y_3) \\
&= t_{\chi}\log\Big|   \frac{t_\chi^2 y_3 }{(2 \pi e)^2 y_1}\Big| +  (r -t_{\chi}) \log (r-t_{\chi}) - (r+t_{\chi}) \log(r + t_{\chi}) + 2t_{\chi}\log (2\pi e) 
\end{split}
\end{displaymath}
   For convenience, we record some relevant derivatives:
   \begin{displaymath}
     \begin{split}
       &\frac{\partial}{\partial t_{\chi}} \phi (...) = \log \Big|\frac{t_{\chi}^2 y_3}{y_1(r^2 - t_{\chi}^2)}\Big|, \quad  \frac{\partial}{\partial r}  \phi(...) = \log \frac{r-t_\chi }{r+t_\chi }, \\
       &  \frac{\partial}{\partial y_1}   \phi(...) =-\frac{t_\chi }{y_1}, \quad  \frac{\partial}{\partial y_3} \phi(...) = \frac{t_\chi }{y_3},\\
       & \frac{\partial^2}{\partial t^2_{\chi}} \phi(...) =  \frac{2r^2}{t_{\chi}(r^2 - t_{\chi}^2)}, \quad \frac{\partial^2}{\partial r^2} \phi(...) = \frac{2 t_\chi }{r^2 - t_\chi ^2}, \quad  \frac{\partial^2}{\partial r \partial t_\chi } \phi(...) = -\frac{2r}{r^2 - t_\chi ^2}.
     \end{split}
   \end{displaymath}
   Using \eqref{eqn:y1-y2-y3-sum-zero},
   we deduce that
   $$|t^{\ast}_{\chi}| =  r |y_1/y_2|^{1/2}. $$
   The integral \eqref{finalint}
   is negligible by an application of Lemma \ref{lem1} with $U=Y=P=Z$, $S=1$ unless $\sgn(y_1) = - \sgn(y_2)$ and $|y_2| \asymp Q/Z^2$. In this case, we apply  Lemma \ref{lem2}  with
   $$ (X_1, X_2, X_3, X_4) = (Z, Q^{1/2}, 1, Q/Z^2), \quad Y = Z. $$
   We compute
   %\pn{changed $\psi$ to $\Psi$
    % to disambiguiate from the additive character}
 $$\phi\Big( r \Big|\frac{y_1}{y_2}\Big|^{1/2}; r, y_1, y_3\Big) =2 r \Psi\Big(\frac{y_1}{y_2}\Big), \quad \Psi(y) = |y|^{1/2} \arctanh(|y|) -  \arctanh(|y|^{1/2})  $$
 where
 \begin{equation}\label{psi1}
   \Psi(y)  = -|y|^{1/2} + O\big(|y|^{3/2}\big), \quad \Psi^{j}(y) \ll |y|^{1/2 - j}.
 \end{equation}
 
 The analysis in the discrete series case is similar. We write
 $\kappa = (k-1)/2$ and  apply Lemma \ref{lem2} and \eqref{hol}
 with
% \pn{changed some $\tau$ to $t_\chi$, $k$ to $\kappa$}
\begin{displaymath}
     \begin{split}
       & \phi(t_{\chi})  = \phi(t_{\chi}; \kappa, y_1, y_3)= t_{\chi}\log\Big|   \frac{t_\chi^2 y_3 }{e^2 y_1 \kappa^2}\Big|  - 2\kappa\arctan\frac{t_{\chi}}{\kappa} - t_{\chi} \log \frac{\kappa^2 + t_{\chi}^2}{e^2\kappa^2} ,\\
       & \frac{\partial}{\partial t_{\chi}}\phi(...) = \log \Big|\frac{t_{\chi}^2 y_3}{y_1(\kappa^2 + t_\chi^2)}\Big|, \quad  \frac{\partial}{\partial \kappa}\phi(...) =- 2\arctan\frac{t_\chi}{\kappa}, \quad \frac{\partial}{\partial y_1}\phi(...) =-  \frac{t_\chi}{y_1}, \\
       & \frac{\partial}{\partial y_3}\phi(...) =  \frac{t_\chi}{y_3}, \quad \frac{\partial^2}{\partial t_{\chi}^2} \phi(t_{\chi}) =  \frac{2\kappa^2  }{t_{\chi}(\kappa^2 + t_{\chi}^2)}. 
     \end{split}
   \end{displaymath}
  % \pn{changed some $h$ to $\phi$ and $r$ to $\kappa$}
 We have $\phi^{(1, 0, 0, 0)}(t^{\ast}_{\chi}; \kappa, y_1, y_3) = 0$ if and only if  $\sgn(y_1) = \sgn(y_3)$ and  $$|t^{\ast}_{\chi}| =  \kappa \sqrt{\frac{y_1}{y_3 - y_1}} $$%\Big(\frac{1}{2} \Big(\sqrt{\frac{4y_1}{y_3} + 1} - 1\Big)\Big)^{1/2} .$$
 Again this is in the support of the integrand if and only if
 $|y_2| \asymp Q/Z^2$, otherwise the integral is negligible. We
 compute
  \begin{displaymath}
   \begin{split}
     & \phi\left(\kappa\sqrt{\frac{y_1}{y_3 - y_1}}; \kappa, y_1, y_3\right) =-2\kappa\arctan\sqrt{\frac{y_1}{y_3 - y_1}} =  2\kappa\tilde{\Psi}\Big(\frac{y_1}{y_2}\Big), \\
     &\tilde{\Psi}(y) = - \arctan \left(\Big(\frac{1}{|x|} + 2\Big)^{-1/2}\right)% \arctan\sqrt{\frac{\sqrt{5-4/(1+|y|)}-1}{2}} - \sqrt{2\sqrt{5-  4/(1+|y|)} - 2} 
   \end{split}
 \end{displaymath}
 where  $\tilde{\Psi}$ satisfies the same formul\ae \, as in \eqref{psi1}.  

 By the definition of the conductor in \eqref{maass} and \eqref{hol} we conclude that \eqref{finalint} equals (up to a negligible error)
 $$Z^{-3/2} V\Big(-y_1, \frac{y_2Z^2}{Q}\Big)e\left(-2 \sqrt{Q} \Psi\Big(\frac{y_1}{y_2}\Big)\right)$$
 in the principal series case and 
 $$Z^{-3/2} V\Big(-y_1, \frac{y_2Z^2}{Q}\Big) e\left(-2 \sqrt{Q}\tilde{\Psi}  \Big(\frac{y_1}{y_2}\Big)\right)$$
 in the discrete series case. 

 We made in the beginning the assumption $z > 0$. The case $z < 0$ leads to an analogous expression, with the minus-sign in the exponential removed. This completes the proof. 
\end{proof}

\section{A shifted convolution problem}\label{shifted}

\subsection{Some preparation}

\subsubsection{Bessel functions} 
We need the following uniform asymptotic formul\ae \, for Bessel
functions. For $t \in \Bbb{R}$, $|t| > 1$ and
$x > 0$, we have  \cite[7.13.2(17)]{EMOT}   \begin{equation}\label{bessel1}
\begin{split}
 \frac{J_{2it}(2x)}{\cosh(\pi t)} &= %\left\{\begin{array}{l} \cos\\ \sin
%\end{array}\right\} \left( \tilde{\omega}(x, t)\right)  \tilde{f}_M(x, t)  + O(|t|^{-M}), 
\sum_{\pm}  e^{\pm 2 i   \omega(x, t)} \frac{f_N^{\pm}(x, t)}{x^{1/2} + |t|^{1/2}}  % + e^{-i  \tilde{\omega}(x, t)} \bar{f}_M^{\pm}(x, t)
 + O_N((x+|t|)^{-N}),  \\
& \omega(x, t) =  |t| \cdot \arcsinh \frac{|t|}{x} - \sqrt{t^2 + x^2},
\end{split}
\end{equation}
where for any fixed $N> 0$ the function $f_N^{\pm}$ is flat.  
 The
error term estimate
stated
in \cite{EMOT} is $O(x^{-N})$, but for  $x \leq t^{1/3}$, say, the
estimate $O(|t|^{-N})$ follows from the power series expansion
\cite[8.402]{GR} of $J_{2it}(x)$.
When we apply this formula in practice, we first extract the
negligible error  $O_N((x + |t|)^{-N})$ in the series expansion
given by \cite[7.13.2(17)]{EMOT} and \cite[8.402]{GR},
without pausing to estimate any derivatives of that error.
We then differentiate the remaining series expansion
to verify the
flatness condition.

For future reference we note the identities
\begin{equation}\label{identities}
\begin{split}
&\frac{\partial}{\partial t} \omega(x, t) =
\arcsinh\frac{t}{x}, \quad \frac{\partial^2}{\partial t^2}
\omega(x, t) = \frac{1}{x^2 + t^2},\\
&\frac{\partial}{\partial x} \omega(x, t) = - \frac{\sqrt{x^2 +
    t^2}}{x}, \quad  \omega(x,t) = -x + \frac{t^2}{2 x} + O \left( \frac{t^4}{x^3} \right).
    \end{split}
    \end{equation}

For $|t| \geq 2x > 0$ we have
by \cite[7.13.2(19)]{EMOT} (again coupled with the power series expansion \cite[8.445 \& 8.485]{GR} for   $x \leq t^{1/3}$) or \cite[(20) with $z \leq 1/2$]{Ba}
\begin{equation}\label{bessel2}
\begin{split}
 K_{2it}(2x) \cosh(\pi t)& = %\left\{\begin{array}{l} \cos\\ \sin
%\end{array}\right\} \left( \tilde{\omega}(x, t)\right)  \tilde{f}_M(x, t)  + O(|t|^{-M}), 
|t|^{-1/2} \sum_{\pm}  e^{\pm 2 i   \omega^{\ast}(x, t)}  g_N^{\pm}(x, t)  % + e^{-i  \tilde{\omega}(x, t)} \bar{f}_M^{\pm}(x, t)
 + O_N((x+ |t|)^{-N}),  
\\
&\omega^{\ast}(x, t) =  |t| \cdot \arccosh \frac{|t|}{x} - \sqrt{t^2 - x^2},
\end{split}
\end{equation}
where for any fixed $N> 0$ the function $g_N^{\pm}$ is flat. 

Similarly, for $x \geq 2k$ we have \cite[(4.24)]{Ol}
\begin{equation}\label{bessel3}
\begin{split}
J_{2k-1}(2x)& = \sum_{\pm} e^{\pm 2ik \tilde{\omega}(x, k)} \frac{{h_N^{\pm}(x, k)}}{x^{1/2}} +   O_N((x+k)^{-N}), \\
& \tilde{\omega}(x, k) = - k \arctan  \sqrt{\frac{x^2}{k^2} - 1} +    \sqrt{ x^2 - k^2}, 
\end{split}
\end{equation}
where $h_N^{\pm}$ is flat. For fixed index, these formul\ae \, simplify greatly, and we have
\begin{equation}\label{fixed}
  \begin{split}
    \frac{J_{2it}(2x)}{\cosh(\pi t)} &= \frac{1}{x^{1/2} + x^{2|\Im t|}} \sum_{\pm}  e^{\pm 2 i   x}  f^{\pm}(x) ,\\
    J_{2k-1}(2x) &= \frac{1}{x^{1/2}+1} \sum_{\pm}  e^{\pm 2 i   x}  h^{\pm}(x) ,
    \\ K_{2it}(2x) \cosh(\pi t) & = \frac{e^{-2x}}{x^{1/2} + x^{2|\Im t|}} g(x) 
  \end{split}
\end{equation}
%\pn{Tweaked somewhat pedantically}
for  $t \in \Bbb{C}$, $k \in \Bbb{N}$, $x >
0$
with $|\Im t| \leq 1/4$ (for simplicity)
and
$t, k$ in a fixed compact set,
where $f^{\pm}, g, h^{\pm}$ can be chosen to be flat (depending
on $t$ or $k$).
See \cite[Lemma 15]{BM} for details on how to glue together the asymptotic formul\ae \, for $x > 1$ and $x < 1$. \\

\emph{Remark:}   As the referee remarked, \cite{EMOT} contains no proofs. The expansions \eqref{bessel1}, \eqref{bessel2}, \eqref{bessel3} are all relatively simple to obtain, since we are in the so-called oscillatory range away from possible degenerate points ($t = x$ for the $K$-function and $k = x$ for the $J$-function with real order). All three uniform asymptotic expansions can be obtained from the integral representations \cite[8.421.1/2, 8.405]{GR}, \cite[8.432.4]{GR} and  \cite[8.411]{GR}
\begin{equation*} 
\begin{split} 
&K_{it}(x) =  \frac{1}{2\cosh(\pi t/2)} \int_{-\infty}^{\infty} \cos(x \sinh v)\exp( it v) dv, \\
& J_{it}(x) =  \frac{1}{\pi} \int_{-\infty}^{\infty}  \big(\cosh(\pi t/2) \sin (x \cosh v) - i \sinh(\pi t/2) \cos(x\cosh v) \big)\exp( it v)dv,\\
& J_{k}(x) =   \frac{1}{2\pi i} \int_{-\pi}^{\pi} \exp(-ki\theta + ix\sin\theta)d\theta
\end{split}
\end{equation*}
by an application of Lemma \ref{lem2}. For the improper integrals, note that the tail can be estimated by partial integration using Lemma \ref{lem1}, cf.\ e.g.\ \cite[Section 4.4]{BB}.

\subsubsection{Jutila's circle method}
We quote Jutila's circle method \cite{Ju1}. 
\begin{lemma}\label{lem3}  Let $Q \geq 1$ and $V$ be a smooth, nonnegative, nonzero  function with support in $[1, 2]$. For $r \in \mathbb{Q}$ write $I_{r}(\alpha)$ for the characteristic function of the interval $[r- 1/Q, r+1/Q]$ and define
\begin{displaymath}
  \Lambda := \sum_q V\Big(\frac{q}{Q}\Big) \phi(q),  \quad I(\alpha) = \frac{Q}{2   \Lambda} \sum_q V\Big(\frac{q}{Q}\Big) %\underset{d\, (\text{{\rm mod }} q)}{\left.\sum \right.^{\ast}} 
  \sum_{\substack{d\, (\text{{\rm mod }} q)\\ (d, q) = 1}} I_{d/q}(\alpha).
\end{displaymath}
Then $I(\alpha)$ %\sj{$\tilde{I}(\alpha)=I(\alpha)$?} 
is a good approximation to the characteristic function on $[0, 1]$ in the sense that
\begin{displaymath}
  \int_0^1 (1 - I(\alpha))^2 d\alpha \ll_{\eps} Q^{\eps - 1}
\end{displaymath}
for any $\eps > 0$.
\end{lemma}

\subsection{Notation and set-up}
%\pn{Added}
Let $T$ be a large positive real number. We recall once again the notation and conventions from \S \ref{weight}, in particular with respect to weight functions $V$.  Moreover, $A \preccurlyeq B$ denotes $A \ll T^{\eps} B$. We consider two more parameters $M$ and $H$ satisfying  
\begin{equation}\label{MH}
M \preccurlyeq T^2, \quad H =  T^{1/3 + \eps},
\end{equation}
and  $\nu \in \Bbb{Z} \setminus \{0\}$ with $\nu \preccurlyeq 1$. The choice of $H$ will eventually turn out to be the optimal choice, and it simplifies the argument if we make it right away at this point. With slightly more extra work, we could run the same argument for $T^{1/3} \ll H \ll T^{1-\varepsilon}$.  %\vb{this sentence is new} \pn{good (mild tweaks)}

Let $g$ a fixed holomorphic or Maa{\ss} Hecke eigenform for ${\rm SL}_2(\Bbb{Z})$ with Hecke eigenvalues $\lambda_g(n)$ of weight $k_g$ or spectral parameter $t_g$. %Let $u_j$ run through
 Let  $\sigma_t(m) = \sum_{ab=n} a^{it} b^{-it}$.  We fix a choice of sign $\pm$.
With these notations we consider 
\begin{equation}\label{L1}
\begin{split}
\mathcal{L} := & \frac{1}{TM^{1/2}}\sum_{2\pi T \leq  t_j \leq 2\pi (T+H)} \frac{1}{L(\Ad^2 u_j, 1)}\\ &\quad\quad\quad\times \Big| \sum_{m } V\Big(\frac{m}{M}\Big) \lambda_j(m) \lambda_g(m+\nu) \exp\Big(\pm 2i\frac{t_j \sqrt{|\nu|}}{\sqrt{m}}\Big)\Big|^2 \\
&+ \frac{1}{TM^{1/2}}\int_{2 \pi T}^{2 \pi (T+H)}\frac{1}{|\zeta(1 + 2 i t)|^2}\\&\quad\quad\quad \times \Big|\sum_{m} V\Big( \frac{m}{M}\Big) \sigma_t(m) \lambda_g(m+\nu) \exp\Big(\pm 2i\frac{t  \sqrt{|\nu|}}{\sqrt{m}}\Big) \Big|^2 \frac{dt}{2\pi} 
\end{split}
\end{equation}
where $u_j$ runs through 
Hecke Maa{\ss} cusp forms for ${\rm SL}_2(\Bbb{Z})$ with Hecke eigenvalues $\lambda_j(n)$ and spectral parameter $t_j \in  [2\pi T, 2\pi (T+H)]$. {The right hand side of \eqref{L1} is essentially a combination of \eqref{alternative} and \eqref{av}; this explains its relevance.}

Analogously we also consider the holomorphic analogue
\begin{multline}\label{L2}
%\begin{split}
\tilde{\mathcal{L}} :=  \frac{1}{TM^{1/2}}\sum_{\substack{4\pi T \leq k \leq 4\pi(T+H)\\ k  \text{ even}}} \sum_{u_j \in B_k} \frac{1}{L(\Ad^2 u_j, 1)}\\
\times\Big| \sum_{m } V\Big(\frac{m}{M}\Big) \lambda_j(m) \lambda_g(m+\nu) \exp\Big(\pm 2i\frac{k \sqrt{|\nu|}}{\sqrt{m}}\Big)\Big|^2 
%\end{split}
\end{multline}
where $u_j$ runs over a Hecke eigenbasis $B_k$ of cusp forms  of weight $k$. For simplicity let us assume $\nu > 0$, the other case being essentially identical.  The aim of this section is to prove the following theorem. 
\begin{theorem}\label{prop3}
Let $T$ be a large parameter and let $\mathcal{L}, \tilde{\mathcal{L}}$ be defined as in \eqref{L1}, \eqref{L2} with $M, H$ as in \eqref{MH}.   Then 
%In this section we prove the upper bound
%\begin{equation*}%\label{asymp}
$\mathcal{L}, \tilde{\mathcal{L}} \preccurlyeq TH.$
%\end{equation*}
\end{theorem}

{The proof of Theorem follows the steps outlined in \S \ref{analytic}. In particular, we will eventually transform the spectral sum  \eqref{L1} into a ``reciprocal'' spectral sum of length $T/H$ in \S \ref{sec58} to which we apply the large sieve.}

For $M \leq T^{2/3 + \eps}$ we can estimate trivially  (using a standard Rankin--Selberg bound)
$$\mathcal{L} \preccurlyeq \frac{1}{TM^{1/2}} THM^2 =HM^{3/2} \preccurlyeq TH$$
in agreement with Theorem \ref{prop3}. 
From now on we assume $M \geq T^{2/3+\eps}$. Then
$$\exp\left(\pm 2i\frac{(t-2 \pi T) \sqrt{\nu}}{\sqrt{m}}\right)$$
is flat for $t  \in [2\pi T, 2\pi( T+H)]$ by our choice of $H$ in \eqref{MH}, so we can replace $t_j$ and $t$ with $2 \pi T$ in the exponential (using the by now familiar device to separate variables in nice functions after having multiplied by a suitable function with compact support in $2\pi T+O(H)$). We restrict to the positive sign in the exponential, the negative sign being essentially identical. 

We can majorize the characteristic function on $[2\pi T, 2\pi(T+H)]$ by $h(t)=  \exp( - (\frac{t-2 \pi T}{H})^2)$ and then symmetrize  with respect to $t \mapsto -t$ in order to make the expression amenable for the Kuznetsov formula. 
%$$h(t) = \exp\Big( - \Big(\frac{t-T}{H}\Big)^2\Big) + \exp\Big( - \Big(\frac{t+T}{H}\Big)^2\Big) $$
%which is admissible for the Kuznetsov formula, although 
The exact shape of the function plays no role. 

\subsection{Application of the spectral summation formula}
%\vb{I added a few details in this subsection}  \sj{looks good to me} \pn{fixed a couple obvious typos}
We open the square in \eqref{L1} and \eqref{L2} and apply the Kuznetsov formula \cite[Theorem 16.3, along with a conversion from Fourier coefficients to Hecke eigenvalues]{IK}. Since $\int h(t) t \, \text{tanh}(\pi t) \, dt \ll TH$ by our choice of $h$,  the diagonal term is bounded by 
$$\preccurlyeq \frac{1}{TM^{1/2}} TH \cdot M = M^{1/2} H \preccurlyeq TH$$
by \eqref{MH} in agreement with Theorem \ref{prop3}. 
For the off-diagonal term we must understand the Bessel transform
\begin{equation}\label{kuz}
\int_{t\in\Bbb{R}}  \exp\left( -\Big(\frac{t-2 \pi T}{H}\Big)^2\right)   \frac{J_{2it}(x)}{\cosh(\pi t)} t\, dt
\end{equation}
As  usual, we use holomorphicity to shift the contour a bit; in this way, we can truncate the $c$-sum in \eqref{off} by some large power of $T$ (cf.\ \cite[p.\ 75]{JM}). Having this done, we may smoothly truncate
the integral to the interval $[2\pi T - H T^\eps, 2\pi T + H T^\eps]$.
For $x \leq T^{1-\eps} H$, we apply Lemma \ref{lem1} and
the uniform asymptotic formula \eqref{bessel1} (along with \eqref{identities}) with
%\pn{changed $R$ from $T/x$
%to account for the logarithmic behavior of $\arcsin(t/x)$ when $x \leq T$}
$$S = \min( 1, T/x), \quad Q = Y = T+ x, \quad U = H$$
to see that the integral is negligible. Having recorded the condition
\begin{equation}
\label{cond}
  x \geq T^{1-\eps} H,
\end{equation}  
   we don't exploit any further cancellation in the
   integral. Using \eqref{bessel1} along with a Taylor expansion
   we have, for $t \in [2 \pi T - H T^{\eps}, 2 \pi T + H
   T^{\eps}]$ and $x \geq T^{1-\eps} H$,
   %we obtain
   the
   approximation
  \begin{equation}\label{tay}
\frac{J_{2it}(2x)}{\cosh(\pi t)} = %\left\{\begin{array}{l} \cos\\ \sin
%\end{array}\right\} \left( \tilde{\omega}(x, t)\right)  \tilde{f}_M(x, t)  + O(|t|^{-M}), 
x^{-1/2} \sum_{\pm}  e^{\pm 2 i  (-x + \frac{1}{2}(2 \pi T)^2/x)}  F^{\pm}(x, t) + O(|t|^{-N}) 
\end{equation}
for a flat function $F^{\pm}$. We substitute this into \eqref{kuz} and integrate trivially over $t$. Thus it suffices to estimate the following off-diagonal term
\begin{equation}\label{off}
\begin{split}
\frac{TH}{TM^{1/2}}   \sum_{ m_1, m_2  } &V\Big(\frac{m_1}{M}, \frac{m_2}{M}\Big)\sum_{c  }V\Big( \frac{c}{C}\Big) \frac{S(m_1, m_2, c)}{C}   \lambda_g(m_1+\nu)  \lambda_g(m_2+\nu)\\
&\times \frac{C^{1/2}}{M^{1/2}}
e\left(\pm\Big( \frac{2\sqrt{m_1m_2}}{c} -\frac{T^2c}{\sqrt{m_1m_2}}\Big)
\pm \frac{2 T\nu^{1/2}(\sqrt{m_2} - \sqrt{m_1})}{\sqrt{m_1m_2}}\right)
\end{split}
\end{equation}
by $\preccurlyeq TH$ (with all sign combinations), as required in Theorem \ref{prop3},
where $C$ runs through $\preccurlyeq 1$ numbers (e.g.\ powers of
2) satisfying 
%\pn{Added lower bound}
\begin{equation}\label{C}
  1 \leq C \preccurlyeq \frac{M}{TH}
\end{equation}
%\pn{Added}
and each weight function $V$ is nice.  

\medskip

We pause for a moment and consider the average
$\tilde{\mathcal{L}}$  over holomorphic forms, in which we
replace the Kuznetsov formula with the Petersson formula
\cite[Theorem 14.5]{IK}.  Here the analogue of \eqref{kuz} is
%\pn{Changed $k - T \mapsto k - 2 T$
%  and $J_{k-1}(x) \mapsto J_{k-1}(2 \pi x)$
 % to be consistent with next display}
$$\sum_{k \in 2\Bbb{Z}} i^k V\left(\frac{k-4 \pi  T}{H}\right) k J_{k-1}(2 \pi x).$$
Using the Fourier representation \cite[8.411.1]{GR} of the Bessel function coupled with Poisson summation, this equals
\begin{displaymath}
\begin{split}
&\sum_{k \in 2\Bbb{Z}} i^k V\left(\frac{k-4 \pi T}{H}\right) k    \int_{-1/2}^{1/2} e\big((1-k)\theta + x \sin2\pi \theta \big) d\theta\\
= & \frac{1}{2}  \sum_{\substack{h \in \Bbb{Z}\\ h \text{ odd}}}   \int_{-1/2}^{1/2}  \int_{-\infty}^{\infty} V\left(\frac{y- 4 \pi T}{H}\right) y  e\Big( \frac{yh}{4} +  (1-y)\theta + x \sin 2\pi\theta\Big) dy\, d\theta. 
\end{split}
\end{displaymath}
The $y$-integral is negligible unless $h = \pm 1$ and $\theta =
\pm 1/4 + O(T^{\eps} H^{-1})$, but then
%\pn{Added a few words based on my understanding of the argument}
the
remaining
$\theta$-integral,
smoothly truncated to the latter range,
is negligible unless $x \geq T^{1-\eps} H$.
This last condition is the analogue of \eqref{cond}.
The analogue of \eqref{tay}, derived from \eqref{bessel3}, is
\begin{equation*}
i^k J_{2k}(x) = %\left\{\begin{array}{l} \cos\\ \sin
%\end{array}\right\} \left( \tilde{\omega}(x, t)\right)  \tilde{f}_M(x, t)  + O(|t|^{-M}), 
x^{-1/2} \sum_{\pm}  e^{\pm 2 i  (x + \frac{1}{2}(2 \pi T)^2/x)}  \tilde{F}^{\pm}(x, k) + O(k^{-N}) 
\end{equation*}
for the present range of variables which leads to the same expression as  \eqref{off} except for a sign in the exponential which won't play a role later. 

\medskip

We now continue with the analysis of \eqref{off}. In
preparation for an application of Voronoi summation, we shift
the 
variables $ m_1, m_2$ by $\nu$. By a Taylor expansion, this makes no difference in the exponential, the resulting correction term being flat. It therefore suffices to bound
\begin{equation}\label{off1}
\begin{split}
\frac{H}{M }   \sum_{ m_1, m_2  } &V\Big(\frac{m_1}{M}, \frac{m_2}{M}\Big)\sum_{c  }V\Big( \frac{c}{C}\Big) \frac{S(m_1-\nu, m_2-\nu, c)}{\sqrt{C}}  \lambda_g(m_1)  \lambda_g(m_2 )\\
&\times e\left(\pm\Big( \frac{2\sqrt{m_1m_2}}{c} - \frac{T^2c}{\sqrt{m_1m_2}}\Big) \pm   \frac{2 T\nu^{1/2}(\sqrt{m_2} - \sqrt{m_1})}{\sqrt{m_1m_2}}\right)   .
\end{split}
\end{equation}

\subsection{Preparatory interlude}\label{interlude}
 We pause to recall the
Voronoi formula (see e.g. \cite[Lemma 25 \& Lemma 6]{BM}):
for $c \in \mathbb{N}$ and $(b,c) = 1$, we have 
\[
  \sum _{n}
  V \left( n \right)
  \lambda_g(n)
  e \left( \frac{b n }{c} \right)
  =
  \frac{1}{c}
  \sum_{\pm}
  \sum_n
  V^{\wedge}_{\pm} \left( \frac{n}{c^2} \right)
  \lambda_g(n)
  e \left( \mp \frac{\bar{b} n }{c} \right),
\]
\[
    V^{\wedge }_{\pm}(y)
  = \int_0^\infty V(x) \mathcal{J}^{\pm}(4 \pi \sqrt{x
    y}) \, d x,
\]
where
$\mathcal{J}^{\pm} = \mathcal{J}^{\pm}_g$ is given
for $g$ a Maa{\ss} form of eigenvalue $1/4+t_g^2$
by
\[
  \mathcal{J}^{+}(x)
  =
  \pi i \frac{J_{2 i t_g}(x) - J_{- 2 i t_g}(x)}{\sinh(\pi
    t_g)},
  \quad 
  \mathcal{J}^{-}(x)
  = 4 \cosh(\pi t_g) K_{2 i t_g}(x)
\]
and
for $g$ a holomorphic form of weight $k_g$
by
  \[
    \mathcal{J}^+(x)
  =
  2 \pi i^{k_g} J_{k_g-1}(x),
  \quad
  \mathcal{J}^-(x) = 0.
\]

In the following sections we will twice have to compute integrals of the form
\begin{equation}\label{osc-int}
\int_{x\in \Bbb{R}} V\Big(\frac{x}{M}\Big) e(\alpha x^{1/2} + \beta x^{-1/2}) dx
\end{equation}
for  certain $\alpha, \beta \in \Bbb{R}$ satisfying
\begin{equation}\label{big}
  |\alpha| M^{1/2} + |\beta|M^{-1/2} \gg M^{\eps}.
\end{equation}
In this case it follows from Lemma \ref{lem1} with $$U = M^{1-\eps/2} , \quad P = M, \quad S = |\alpha| M^{-1/2} + |\beta| M^{-3/2}, \quad Y = MS$$
that \eqref{osc-int} is negligible unless $\beta/\alpha \asymp
M$ (which,
by the conventions of \S\ref{weight},
implies in particular that
$\sgn(\alpha) = \sgn(\beta)$)
in which case by Lemma \ref{lem2} 
(after restricting to dyadic ranges
$\alpha \asymp A$ and $\beta \asymp B$ 
 and also possibly
restricting the support of
$V$
to a neighbourhood of $t^{\ast}$) with
$$\phi(t)  = \phi(t; \alpha, \beta) = \alpha t^{1/2} + \beta t^{-1/2}, \quad t^{\ast} = \frac{\beta}{\alpha},  \quad \frac{\partial }{\partial t} \phi(t^{\ast}; \alpha, \beta) = \frac{\beta}{2t_0^{5/2}} \asymp \frac{\beta}{M^{5/2}}, $$
 it equals
\begin{equation}\label{alphabeta}
  \frac{M^{5/4}}{|\beta|^{1/2}}
  V\Big(\frac{\alpha }{A} \Big)
  V\Big(\frac{\beta }{B} \Big)
  V\Big(\frac{\beta/\alpha}{M}\Big)e\big(2 \sgn(\alpha)\sqrt{\alpha \beta}\big),
\end{equation}
as usual with different functions $V$, and up to a negligible
error.

\subsection{Application of Voronoi summation}\label{sec:appl-voron-summ}

We open the Kloosterman sum in \eqref{off1} and apply the
Voronoi formula
to
the following
$m_2$-sum:
%\pn{changed the sign of the final term in the phase
%  from $+$ to $-$}
$$\sum_{m_2} V\Big(\frac{m_2}{M}\Big)\lambda_g(m_2) e\Big(
\frac{\bar{d}m_2}{c}\Big)e\left( \pm\Big(
\frac{2\sqrt{m_1m_2}}{c} - \frac{T^2c}{\sqrt{m_1m_2}}\Big)
\mp \frac{2 T\nu^{1/2}}{\sqrt{ m_2}}\right) $$
where $(d, c) = 1$.
We analyze the integral transforms using
\eqref{bessel3} and \eqref{fixed}.
% If $g$ is a Maa{\ss} form we get two terms containing a Bessel
% $J$-function and a Bessel $K$-function. We analyze these using
% \eqref{fixed}.
In the Maa{\ss} case,
we see that
the $\mathcal{J}^-$-term is negligible
thanks to
the rapid decay of the Bessel
$K$-function
and the consequence $M/C^2 \succcurlyeq H^2 \geq T^{2/3+\eps}$
of our hypotheses  \eqref{MH} and \eqref{C}.
In either case, the $\mathcal{J}^+$ term contributes
% For the term involving the Bessel $J$-function in the holomorphic and Maa{\ss} case we see that    we can  re-write the above sum as
\begin{equation}\label{vor}
\begin{split}
\sum_{m_2}\lambda_g(m_2) e\Big( &\frac{-d m_2}{c}\Big) \frac{1}{c} \int_{x\in\Bbb{R}}  V\Big( \frac{x}{M}\Big) \left(\Big(\frac{\sqrt{ m_2x}}{c}\Big)^{1/2} + \Big(\frac{\sqrt{ m_2x}}{c}\Big)^{2|\Im t_g|}\right)^{-1} \\
&\times e\left( \sigma_1\Big(\frac{2\sqrt{m_1x}}{c} - \frac{T^2c}{\sqrt{m_1x}}\Big)+\sigma_2 \frac{2 T\nu^{1/2}}{\sqrt{ x}}\right)e\left( \sigma_3\frac{2\sqrt{xm_2}}{c}\right)  dx
\end{split}
\end{equation}
with $\sigma_1, \sigma_2, \sigma_3 \in \{\pm1 \}$, and where as
usual the meaning of $V$ may have changed. The $x$-integral is
of the shape \eqref{osc-int} with
$$\alpha = \frac{2 ( \sigma_1 \sqrt{m_1} + \sigma_3 \sqrt{m_2})}{c},
  \quad \beta =  -\sigma_1 \frac{T^2c}{\sqrt{m_1}} + \sigma_2  2T\nu^{1/2}. $$
For the analysis of this integral it is important to note that 
$$\frac{\sqrt{m_1x}}{c} \asymp \frac{M}{C}$$
is by at least a factor $H^2 T^{-\eps}$ larger than
$$\frac{T^2c}{\sqrt{m_1x}} \asymp \frac{T^2C}{M}$$
by \eqref{MH} and \eqref{C}, and the latter is larger than
$T\nu^{1/2}x^{-1/2} \preccurlyeq TM^{-1/2}$ unless $C \preccurlyeq 1$
and
%\pn{Changed from $T \preccurlyeq M^2$}
$T^2 \preccurlyeq M$.  
%\pn{added}
Let us define 
\begin{equation}\label{R}
  R := T^2C^2/M.
\end{equation}
If $R \geq T^{\eps}$, then \eqref{big} is satisfied, and we conclude from the discussion in the previous subsection that the $x$-integral is negligible unless   $\sigma_1 = -\sigma_3$  and\footnote{This implies in particular $m_2 - m_1 > 0$. In the holomorphic average $\tilde{\mathcal{L}}$, the term $-T^2c/\sqrt{m_1x}$ would not have a minus sign, so that here $m_2 - m_1 < 0$. This sign is responsible for the choice of the integral transform in the final application of the Kuznetsov formula in \S\ref{sec58}. We mention this only for the sake of clarity. In the following all sign combinations are treated uniformly.}
$$m_2 - m_1 \asymp R.$$
If $R \preccurlyeq 1$, the same argument still shows that 
 $m_2 - m_1 \preccurlyeq R \preccurlyeq 1$ (which allows in
particular $m_1 = m_2$).
In particular,
$\sqrt{m_2x}/c \asymp M/C \gg T^{1-\eps}H\gg 1$ is large and therefore we can drop the term containing $|\Im t_g|$  in \eqref{vor}.

Before we proceed, we estimate the total contribution of
$R \preccurlyeq 1$ in \eqref{vor} when substituted into
\eqref{off1} trivially by
%\pn{changed LHS from
%  $\preccurlyeq \frac{H}{M}M \frac{1}{M^{1/2}}$.
%}
$$\preccurlyeq \frac{H}{M} \cdot M
\cdot M^{1/2}= HM^{1/2} \preccurlyeq HT,$$
which is acceptable for Theorem \ref{prop3}.
So from now on we assume that 
\begin{equation}\label{large}
   R \geq T^{\eps}.
\end{equation}
In view of \eqref{MH} we can then also assume
\begin{equation}\label{large1}
  T^2C/M \geq T^{\eps}.
\end{equation}
 For such $R$ and $m_2 - m_1  \asymp R$   and $\sigma_1 = -\sigma_3$ we see from \eqref{alphabeta}  
 that    \eqref{vor} can be recast as 
 \begin{multline*}
\frac{M/T}{c} \sum_{ m_2} V\Big( \frac{m_2 - m_1}{R}\Big)\lambda_g(m_2) e\Big( -\frac{d m_2}{c}\Big)\\ \times e\left( \pm   \sqrt{8(\sqrt{m_2} - \sqrt{m_1}) \Big|\frac{T^2}{\sqrt{m_1}} \pm \frac{2 T\nu^{1/2}}{  c}\Big|} \right),
\end{multline*}
up to a negligible error and for suitable sign combinations. 
Plugging back into \eqref{off1} and calling $r = m_2 - m_1$, $m
= m_1$, we obtain a total contribution
\begin{displaymath}
\begin{split}
\frac{H}{TC^{1/2} }  \sum_{ r} & \sum_{ m} \sum_{c  }V\Big( \frac{r}{R}, \frac{m}{M}, \frac{c}{C}\Big)  \frac{ S(-r-\nu, -\nu, c)}{c} \lambda_g(m)  \lambda_g(m+r)\\
&\times e\left( \pm   \sqrt{8(\sqrt{m+r} - \sqrt{m}) \Big| \frac{T^2}{\sqrt{m}} \pm \frac{2 T\nu^{1/2}}{ c}\Big|}\pm \frac{2 T\nu^{1/2}}{ \sqrt{m}}\right).
%e\Big( \frac{2\sqrt{m_1m_2}}{c} - \frac{K^2c}{\sqrt{m_1m_2}} + \frac{K(\sqrt{m_2} - \sqrt{m_1})}{\pi\sqrt{m_1m_2}}\Big)  
\end{split}
\end{displaymath}
 For notational simplicity the previous display deals only with the case $r > 0$. The case $r < 0$, coming from the holomorphic average $\tilde{\mathcal{L}}$ can be treated in the same way. We simplify the exponential a bit using suitable Taylor
expansions.
First,
%\pn{added}
using the expansion
$\sqrt{ \sqrt{1 + 2 x} - 1}
= x^{1/2} - x^{3/2}/4 + \dotsb$,
we replace
$\sqrt{|\sqrt{m+r} - \sqrt{m}|}$ with $(r/2)^{1/2} m^{-1/4}$,
the error being flat since,
%\pn{very mildly tweaked}
by \eqref{MH} and \eqref{C},
 $$
\frac{1}{M^{1/4}}
\cdot 
\frac{R^{3/2}}{M^{3/2}} \cdot \frac{T}{M^{1/4}} = \frac{T^4 C^3}{M^3} \preccurlyeq \frac{T}{H^3} \preccurlyeq 1.$$
%the error in replacing $\sqrt{\sqrt{m+r} - \sqrt{m}}$  with $(r/2)^{1/2} m^{-1/4}$ is flat. 
Similarly, we can replace
$$\Big| \frac{T^2}{\sqrt{m}} \pm \frac{2 T\nu^{1/2}}{  c}\Big|^{1/2} \quad 
\text{with} \quad \Big|\frac{T}{m^{1/4}} \pm \frac{   m^{1/4}  \nu^{1/2}}{c}\Big|$$
up to a flat function.
% And finally we can replace the denominator $m^{1/2}$ with $(m+\nu)^{1/2}$, up to a flat function, using that  $T/(MH) \preccurlyeq 1$.   
Thus it suffices to bound (with various sign combinations)\begin{displaymath}
\begin{split}
\frac{H}{TC^{1/2} }  \sum_{ r} & \sum_{ m} \sum_{c  }V\Big( \frac{r}{R}, \frac{m}{M}, \frac{c}{C}\Big)  \frac{ S(-r-\nu, -\nu, c)}{c} \\
&\times \lambda_g(m)  \lambda_g(m+r)e\left(\pm \frac{2 T r^{1/2}}{ m^{1/2}}  \pm  \frac{2 (r\nu)^{1/2}}{c} \pm  \frac{2 T\nu^{1/2}}{ m^{1/2}}\right)%e\Big( \frac{2\sqrt{m_1m_2}}{c} - \frac{K^2c}{\sqrt{m_1m_2}} + \frac{K(\sqrt{m_2} - \sqrt{m_1})}{\pi\sqrt{m_1m_2}}\Big)  
\end{split}
\end{displaymath}
by $\preccurlyeq TH$. %For notational simplicity let us assume $r > 0$, the other case being essentially identical. 
 We write the previous display as
\begin{equation}\label{newshift}
\begin{split}
\frac{H}{TC^{1/2} }  \sum_{ r} F(r)& \sum_{ m}  V\Big( \frac{ r}{R}, \frac{m}{M}\Big)    \lambda_g(m)  \lambda_g(m+r) e\left(\pm \frac{2 T (r^{1/2} \pm   \nu^{1/2})}{m^{1/2}} \right)%e\Big( \frac{2\sqrt{m_1m_2}}{c} - \frac{K^2c}{\sqrt{m_1m_2}} + \frac{K(\sqrt{m_2} - \sqrt{m_1})}{\pi\sqrt{m_1m_2}}\Big)  
\end{split}
\end{equation}
where
$$F(r)=  \sum_{c  }V\Big(\frac{c}{C}\Big)\frac{ S(-r-\nu, -\nu, c)}{c}  e\Big(\pm  \frac{2 (r\nu)^{1/2}}{c}\Big).$$

\subsection{An average of Kloosterman sums}

 We pause for a moment and prove that
\begin{equation}\label{2norm}
\sum_{  r \asymp R} |F(r)|^2 \preccurlyeq R.
\end{equation}
It is tempting to use the Kuznetsov formula, but we can argue in an elementary way.  
 We insert a smooth weight and open the square getting 
\begin{multline*}
%\begin{split}
\sum_{r}V\Big(\frac{  r}{R}\Big) |F(r)|^2 =%&
\sum_{c_1, c_2 \asymp C}\frac{1}{c_1c_2} V\Big( \frac{c_1}{C}, \frac{c_2}{C}\Big)\\
\times\underset{d_1 \, (\text{mod } c_1)}{\left.\sum
\right.^{\ast} } \underset{d_2 \, (\text{mod } c_2)}{\left.\sum
\right.^{\ast} }   e\left(-\nu\Big( \frac{ d_1 +  \bar{d}_1}{c_1} +  \frac{ d_2 +  \bar{d}_2}{c_2}\Big) \right)\\
%&\quad\quad\quad
\times\sum_{r}V\Big(\frac{ r}{R}\Big)
e\left( -\frac{(d_1c_2 + d_2c_1)r}{c_1c_2}\right)
e\left(\pm 2 (r\nu)^{1/2} \frac{c_2 - c_1}{c_1c_2}\right).
%\end{split}
\end{multline*}
We split the $r$-sum into residue classes modulo $c_1c_2$ and
apply Poisson summation. The combined conductor of the exponentials is $\preccurlyeq  C^2 R^{1/2}/C$, and since 
$R\gg T^{-\eps} C^2 R^{1/2}/C$, it is easy to see that the dual sum picks
up at most $\preccurlyeq 1$ terms. 
 Hence we obtain the upper bound
%the central Poisson term survives, and up to a negligible error we obtain the upper bound
$$\sum_{  r \asymp R} |F(r)|^2  \ll R \sum_{h \preccurlyeq 1} \sum_{c_1, c_2 \asymp C}\frac{1}{c_1c_2} \underset{d_1c_2 + d_2c_1 \equiv h \, (\text{mod }c_1c_2)} {\underset{d_1 \, (\text{mod } c_1)}{\left.\sum
\right.^{\ast} } \underset{d_2 \, (\text{mod } c_2)}{\left.\sum
\right.^{\ast} } } 1.$$
If $h = 0$, then
the inner double sum vanishes unless $c_1 = c_2$. If $h \not= 0$, then the congruence fixes $d_j$ modulo $c_j/(c_j, h)$. In either case we confirm \eqref{2norm}.\\

\emph{Remark:} It is clear from the proof that the smooth weight function $V(c/C)$ in the definition of $F(r)$ plays no role here and could be replaced with arbitrary bounded weights $\alpha_c \ll 1$ for  $c \asymp C$. The only assumption on $R$ and $C$ used in the proof is $C \preccurlyeq R^{1/2}$.

\subsection{Application of the circle method}\label{sec56}

We now return to \eqref{newshift} and  treat the $m$-sum as a shifted convolution problem:
$$\sum_{n, m} V\Big( \frac{n}{M}, \frac{m}{M}\Big) \lambda_g(n) \lambda_g(m)  e\left(\pm \frac{2 T (r^{1/2} \pm   \nu^{1/2})}{ n^{1/2}} \right) \int_0^1 e\big((m-n - r )\alpha\big) d\alpha.$$
We choose a gigantic parameter $Q = T^{1000}$ and replace the characteristic function on $[0, 1]$  with $I(\alpha)$, defined in Lemma \ref{lem3}. By Cauchy-Schwarz and trivial bounds, this introduces an error at most $M^2 Q^{\eps-1/2}$ which can be neglected. In this way, the $\alpha$-integral becomes
$$ \frac{Q}{2\Lambda} \sum_q V\Big(\frac{q}{Q}\Big)  \sum_{\substack{d\, (\text{{\rm mod }} q)\\ (d, q) = 1}} \int_{-1/Q}^{1/Q} e\left( \Big( \frac{d}{q} + \alpha\Big) (m-n - r ) \right) d\alpha.$$
The portion
$e(\alpha (m-n - r))$
 is obviously flat, so we end up with bounding
\begin{multline}\label{beforeVor}
%\begin{split}
\frac{1}{\Lambda}  \sum_q V\Big(\frac{q}{Q}\Big)\sum_{\substack{d\, (\text{{\rm mod }} q)\\ (d, q) = 1}} \sum_{n, m} V\Big( \frac{n}{M}, \frac{m}{M}\Big)\\ \times\lambda_g(n) \lambda_g(m)  e\left(\pm \frac{2 T (r^{1/2} \pm  \nu^{1/2})}{ n^{1/2}} \right)  e\Big(   \frac{d}{q}  (m-n - r ) \Big)
%\end{split}
\end{multline}
where $\Lambda = \sum_q V(q/Q) \phi(q) = Q^{2 + o(1)}$. Recall that this represents the $m$-sum in \eqref{newshift}. 

\subsection{Voronoi again}\label{sec57}
Having separated the variables $m, n$ by the circle method, we
apply the Voronoi formula (\S\ref{sec:appl-voron-summ})
to both sums in \eqref{beforeVor}. This is simple for the $m$-sum 
$$\sum_m V\Big(\frac{m}{M}\Big) \lambda_g(m) e\Big( \frac{dm}{q}\Big)$$
because it contains no archimedean oscillation. If $g$ is a Maa{\ss} form, we get two terms, one with a Bessel $J$-function and one with a Bessel $K$-function, and as before we use \eqref{fixed} for the analysis. The dual variable  can be truncated at $\preccurlyeq Q^2/M$ at the cost of a negligible error, by the oscillatory behaviour of the Bessel $J$-function \eqref{bessel1} with $t \ll 1$ and the rapid decay of the Bessel $K$-function.  %Using again that ${\rm SL}_2(\Bbb{Z})$ has no exceptional eigenvalues, $J$ and $K$ blow up at most logarithmically at $0$, and so 
Using a smooth partition of unity, we obtain $\preccurlyeq 1$ partial sums of the shape
\begin{equation}\label{m-sum}
\frac{M}{Q} \sum_m V\Big(\frac{m}{M'}\Big)  \left(\frac{M'M}{ Q^2}\right)^{-|\Im t_g|} \lambda_g(m) e\Big( \pm \frac{\bar{d}m}{q}\Big), \quad M' \preccurlyeq Q^2/M.
\end{equation}
The same analysis (slightly simpler) applies if $g$ is holomorphic. 

For the $n$-sum
$$\sum_n V\Big(\frac{n}{M}\Big) \lambda_g(n) e\Big(\pm \frac{2 T (r^{1/2} \pm \nu^{1/2})}{ n^{1/2}} \Big) e\Big( -\frac{dn}{q}\Big)$$
we argue as follows.
We note that $Tr^{1/2}x^{-1/2} \asymp TR^{1/2}M^{-1/2} \geq T^{\eps}$ by \eqref{MH} and \eqref{large}, so there is sizeable oscillation. In particular, we see from the last formula in \eqref{fixed} and Lemma \ref{lem1} with $U= P  = M$, $Y = TR^{1/2}M^{-1/2}$, $S = Y/M$ (so that $PS/\sqrt{Y}$ and $SU$ are both $\gg T^{\eps}$) that  the Bessel $K$-term is negligible. %
 For the Bessel $J$-term %, using the asymptotic formula \eqref{bessel1} with $t \ll 1$, 
again by \eqref{fixed} we need to understand the transform
\begin{multline*}
\int_{x\in\R} V\Big(\frac{x}{M}\Big) e\left(\pm \frac{2 T (r^{1/2} \pm \nu^{1/2})}{ x^{1/2}} \right)e\left( \frac{\pm2 \sqrt{nx}}{q}\right) \\ \times \left(\Big(\frac{ \sqrt{nx}}{q}\Big)^{1/2} + \Big(\frac{ \sqrt{nx}}{q}\Big)^{2|\Im t_g|} \right)^{-1} dx,
\end{multline*}
%Here we note that $Tr^{1/2}x^{-1/2} \asymp TR^{1/2}M^{-1/2} \geq T^{\eps}$ by \eqref{MH} and \eqref{large}, so that the first exponential has sizeable oscillation. 
which is of the shape \eqref{osc-int} with
$$\alpha = \pm\frac{2\sqrt{n}}{q}, \quad \beta = \pm 2T(r^{1/2}
\pm \nu^{1/2}).$$
%\pn{added}
The condition \eqref{big}
is satisfied in view of \eqref{MH} and \eqref{large}.
From the discussion in \S\ref{interlude} we conclude that the $x$-integral   is negligible unless $$n \asymp \frac{Q^2 T^2 R}{M^2}, $$
 in which case it equals, up to a negligible error,
\begin{displaymath}
\begin{split}
V\left( \frac{n}{Q^2T^2R/M^2}\right) &e\left( \pm \frac{4 T^{1/2} (r^{1/2} \pm\nu^{1/2})^{1/2} n^{1/4}}{q^{1/2}}\right)\\
& \times \left( \Big(\frac{T R^{1/2}}{M^{1/2}}\Big)^{1/2}+ \Big(\frac{T R^{1/2}}{M^{1/2}}\Big)^{2|\Im t_g|}\right)^{-1} \left(\frac{TR^{1/2}}{M^{5/2}}\right)^{-1/2}.
\end{split}
\end{displaymath}
Here we can afford to drop the term involving $|\Im t_g|$ and see that the
 $n$-sum is of the form
\begin{multline*}
%\begin{split}
\frac{1}{q} \Big(\frac{TR^{1/2}}{M^{3/2}}\Big)^{-1} \sum_{n } V\left( \frac{n}{Q^2T^2R/M^2}\right) \lambda_g(n)e\Big( \frac{\bar{d}n}{q}\Big)\\ \times  e\Big( \pm \frac{4 T^{1/2} (r^{1/2} \pm \nu^{1/2})^{1/2} n^{1/4}}{q^{1/2}}\Big).
% \end{split}
 \end{multline*}
%$$\frac{M^{3/2}}{R^{1/2} QT} \sum_{n } W\Big( \frac{n}{Q^2T^2R/M^2}\Big) \lambda_g(n)e\Big( \frac{\bar{d}n}{q}\Big) e\Big( \pm \frac{4T^{1/2} (|r|^{1/2} \pm \nu^{1/2})^{1/2} n^{1/4}}{q^{1/2}}\Big).$$
Substituting this and \eqref{m-sum} back into \eqref{beforeVor}, we can replace  \eqref{beforeVor} with terms of the form 
\begin{equation}\label{recast}
\begin{split}
\frac{M}{Q} &\left(\frac{M'M}{ Q^2}\right)^{-|\Im t_g|}\frac{M^{3/2}}{TR^{1/2}} \frac{1}{Q^2} \sum_{m} \sum_{n}V\Big(\frac{m}{M'}\Big)V\left( \frac{n}{Q^2T^2R/M^2}\right)  \lambda_g(n) \lambda_g(m) \\
&\times\sum_{q  } V\Big(\frac{q}{Q}\Big) \frac{S(\pm m + n, -r, q)}{q}
e\Big( \pm
\frac{ 4 T^{1/2} (r^{1/2} \pm \nu)^{1/2} n^{1/4}}{q^{1/2}}\Big).
\end{split}
\end{equation}
 We note that $n$ is always substantially bigger than $m$, since $R \geq T^{\eps}$, so the arguments of the Kloosterman sum never vanish.

\subsection{Kuznetsov again}\label{sec58} Eventually \eqref{recast} has   to be inserted into \eqref{newshift}, but before we do this we focus on the $q$-sum 
which  calls for an application of the Kuznetsov formula \cite[Theorem 16.5, 16.6]{IK}. Before we carry this out, we simplify the exponential a bit. By a Taylor expansion we can replace $n^{1/4}$ with $(n \pm m)^{1/4}$, the error being flat since
$$\frac{T^{1/2} R^{1/2} m}{n^{3/4}Q^{1/2}} \preccurlyeq \frac{T^{1/2} R^{1/2} Q^2 M^{3/2}}{M Q^{3/2} T^{3/2} R^{3/4} Q^{1/2}} = \frac{M^{1/2}}{TR^{1/4}} \preccurlyeq \frac{1}{R^{1/4}} \leq 1.$$
We can also replace  $(r^{1/2}\pm  \nu^{1/2})^{1/2}$ with $ r^{1/4} \pm \frac{1}{2} \nu^{1/2} r^{-1/4} $, the total error being flat since
$$\frac{T ^{1/2} n^{1/4} }{Q^{1/2} R^{3/4}} \preccurlyeq \frac{T^{1/2} Q^{1/2} T^{1/2} R^{1/4}}{M^{1/2} Q^{1/2} R^{3/4}} =  \frac{1}{C} \leq 1.$$
Therefore the $q$-sum in \eqref{recast} becomes
\begin{multline*}
\sum_{q  } V\Big(\frac{q}{Q}\Big) \frac{S(\pm m + n, -r, q)}{q}  e\Big( \pm \frac{4 T^{1/2} r ^{1/4} (n\pm m)^{1/4}}{q^{1/2}}\Big) e\Big(\pm  \frac{2   ( \nu T)^{1/2} (n\pm m)^{1/4}}{r^{1/4}  q^{1/2}} \Big).
\end{multline*}
The first exponential fits very well into the shape of the Kuznetsov formula, the second does not. Unfortunately it is not (always) flat, but we can afford to open it by Mellin inversion. We first add a redundant weight function $
V(r/R)$ and then write
\begin{displaymath}
\begin{split}
&V\Big(\frac{  r}{R}\Big) e\left(\pm  \frac{2( \nu T)^{1/2} (n\pm m)^{1/4}}{r^{1/4} q^{1/2}} \right) \\
&= \int _{\Re(s)=0} \left(\int_{x\in\Bbb{R}_{>0}} V\Big(\frac{  x}{R}\Big) e\left(\pm  \frac{2 ( \nu T)^{1/2} r^{1/4}(n\pm m)^{1/4}}{x^{1/2} q^{1/2}} \right) x^{s} \frac{dx}{x}\right) r^{-s} \frac{ds}{2\pi i}.
\end{split}
\end{displaymath}
The outer $s$-integral can be truncated at
$$\Im s \preccurlyeq \frac{T^{1/2} (Q^2T^2R/M^2)^{1/4}}{R^{1/4}Q^{1/2}} = \frac{T}{M^{1/2}}.$$
We sacrifice all cancellation in the $x, s$-integrals and pull
them outside of all sums, including the $r, m, n$-sums in
\eqref{recast} and \eqref{newshift} which are currently not
displayed.
%\pn{added}
(We remark that in the ``generic'' range $M \approx T^2$,
we sacrifice nothing here.)
The remaining $q$-sum is of the form
%bound the $q$-sum  by
\begin{equation}\label{qsum}
  %\preccurlyeq
   \frac{T}{M^{1/2}} \sum_q  \frac{S(\pm m + n, -r, q)}{q}  \Phi_x\left(\frac{\sqrt{(n \pm m)r}}{q}\right)
     \end{equation}
where $ x \asymp R$ and 
$$\Phi_x(z) = V\Big(\frac{z}{Z}\Big) e\left( \pm  4 (T z)^{1/2}\Big(1 \pm \frac{ \nu^{1/2}}{2x^{1/2}}\Big) \right)$$
with
\begin{equation}\label{z}
Z =  \frac{R^{1/2}(Q^2T^2R/M^2)^{1/2}}{Q} = \frac{RT}{M} = \frac{T^3C^2}{M^2}
\end{equation}
by \eqref{R}. 
It is important that $\Phi_x$ does not depend on any of the variables $n, m, r, q$. 
We can now apply the Kuznetsov formula to the $q$-sum.
For a quantitative analysis, we need to understand the three integral transforms
\begin{displaymath}
\begin{split}
\check{\Phi}_+(t) = &\int_0^{\infty} \Phi_x(z) \frac{J_{2it}(z)- J_{-2it}(z)}{\sinh(\pi t)} \frac{dz}{z},\\
\check{\Phi}_-(t) =& \int_0^{\infty} \Phi_x(z) K_{2it}(z) \cosh(\pi t) \frac{dz}{z},\\
\check{\Phi}_{\text{hol}}(k) =& \int_0^{\infty} \Phi_x(z) J_{k-1}(z) \frac{dz}{z}. 
\end{split}
\end{displaymath}
To this end, it is important to note that $\Phi_x$ has sizeable
oscillation since
\begin{equation*}%\label{sizeable}
T^{1/2}z^{1/2} \asymp T^{1/2} Z^{1/2} = T^2C/M \gg T^{\varepsilon}%\min(T^{\eps}, zHT^{-\eps}),
\end{equation*}
%\begin{cases} T^{\eps},\\
%z H  T^{-\eps}\end{cases}$$
by   \eqref{large1}. % and  the
%second from \eqref{C} and \eqref{z}.
%\pn{added,
%  since it seems relevant in the analysis below}
We note also, by \eqref{C} and 
\eqref{z}, that 
\begin{equation}\label{sizeable1}
Z/T = (T C/M)^2 \preccurlyeq 1/H^2,
\end{equation}
so that $Z$ is much smaller than $T$.

We start with an analysis of $\check{\Phi}_+(t)$, recalling the formula \eqref{bessel1}.  First we observe that it is negligible unless $t \asymp T^{1/2} Z^{1/2}$, otherwise we may apply Lemma \ref{lem1} with
 $$U = Z T^{-\eps}, \quad P  = Z, \quad  Y = \max(|t|, T^{1/2} Z^{1/2}), \quad S = Y/Z. $$
If $t \asymp T^{1/2} Z^{1/2}$, we can apply Lemma \ref{lem2} (or
in fact the one-dimensional version \cite[Proposition
8.2]{BKY}). We will not compute the stationary point $x_0$ and
the shape of the resulting phase (although it can be done
algebraically and leads to a quadratic equation with a unique
solution if potentially the support of
%\pn{changed $W$ to $V$}
$V$ is slightly
restricted), but only bound the size of the integral to be
 $$\ll \Big( \frac{T^{1/2}}{Z^{3/2}}\Big)^{-1/2} \frac{1}{(TZ)^{1/4}} \frac{1}{Z} = \frac{1}{(TZ)^{1/2}}.$$ 
Here the first factor comes from the stationary phase
analysis, the second factor from \eqref{bessel1}
(noting that $Z^{1/2} + |t|^{1/2} \asymp |t|^{1/2} \asymp (T Z)^{1/4}$)
and the last factor from the measure $dz/z$.
  
For $\check{\Phi}_-(t)$ we can argue in the same way, using
\eqref{bessel2}, if $|t| \gg Z$ with a sufficiently large
implied constant.
For $|t| \ll Z$ with a sufficiently
small implied
constant,
% \pn{should be ``small implied constant'',
%  right?}
the Bessel $K$-function is negligible (cf.\ e.g.\ \cite[(A.3)]{BLM}), and for $|t| \asymp Z$ we simply regard the Bessel kernel as part of the weight function and  use Lemma \ref{lem1} with 
$$U = \min(Z T^{-\eps}, 1),  \quad P  = Z, \quad  Y =   T^{1/2}
Z^{1/2}, \quad S = Y/Z$$ (cf.\ e.g.\ \cite[(A.1) \& (A.3)]{BLM}
for the relevant bounds for Bessel functions)
to show that this contribution is
negligible, too.

Similarly we see that $\check{\Phi}_{\text{hol}}(k) $ is
negligible in all cases.

As an aside, we note that this analysis
is
independent of potential exceptional eigenvalues,
whose contribution would always be negligible (because of \eqref{sizeable1}). 

Summarizing the previous discussion, we can re-write \eqref{qsum},   up to a negligible error, as
\begin{equation}\label{dualspectral}
    \frac{T}{M^{1/2}}  \frac{1}{TZ} \sum_{t_j \asymp (TZ)^{1/2}}  \frac{\rho_j(n \pm m) \rho_j( -r)}{\cosh(\pi t_j)}\Psi(t_j) + \text{continuous spectrum}
    \end{equation}
where $\Psi$ is some function of which we only need to know $\Psi\preccurlyeq 1$, and the sum runs over the Fourier coefficients (in usual normalization) of Hecke Maa{\ss} cusp forms for ${\rm SL}_2(\Bbb{Z})$ with spectral parameter $t_j$.  We do not need to be more precise since the spectral sum (including the continuous contribution) will disappear in a moment when we apply the Cauchy--Schwarz inequality and the large sieve.

\subsection{Cauchy--Schwarz and the large sieve}

 We recall that the previous display represents the $q$-sum in
\eqref{recast} which itself is the $m$-sum in \eqref{newshift}.
Applying the Cauchy--Schwarz inequality, we deduce that the total
contribution to $\mathcal{L}$ and $\tilde{\mathcal{L}}$
is
$\preccurlyeq \Delta
\sqrt{\Sigma_1 \Sigma_2}$,
% \begin{equation}\label{almostdone}
%   \Delta
%   \sqrt{\Sigma_1 \Sigma_2},
% \end{equation}
where
\begin{align*}
  \Delta &:=
  \frac{H}{TC^{1/2} }  \frac{M}{Q}
  \left(\frac{M'M}{ Q^2}\right)^{-|\Im t_g|}
  \frac{M^{3/2}}{TR^{1/2}}  \frac{1}{Q^2}\frac{T}{M^{1/2}}
           \frac{1}{TZ},
           \\
  \Sigma_1 &:=
  \sum_{t_j \asymp T^2C/M}  \frac{1}{\cosh(\pi t_j)}
  \Big|\sum_{r  } V\Big( \frac{r}{R}\Big)F(r)   \rho_j(r)
             \Big|^2 + (\dotsb),
             \\
  \Sigma_2 &:=
  \sum_{t_j \asymp T^2C/M}  \frac{1}{\cosh(\pi t_j)}
  \Big|\sum_{s  } \rho_j(s)G(s) \Big|^2 
  + (\dotsb),
\end{align*}
where $(\dotsb)$ denotes the continuous spectrum
contribution
and
$$G(s) :=\sum_{n \pm m = s}V\Big(\frac{m}{M'}\Big)V\left( \frac{M^2n}{Q^2T^2R}\right)  \lambda_g(n) \lambda_g(m). $$
We prepare for the large sieve by estimating the 2-norm
\begin{multline*}
  %\begin{split}
    \sum_{s} |G(s)|^2 = \sum_{n_1 \pm m_1 = n_2 \pm m_2} V\Big(\frac{ m_1}{M'}, \frac{m_2}{M'} \Big)\\ \times V\left( \frac{M^2n_1}{Q^2T^2R}, \frac{M^2n_2}{Q^2T^2R}\right)    \lambda_g(n_1) \lambda_g(m_1)\lambda_g(n_2) \lambda_g(m_2).
  %\end{split}
\end{multline*}
We detect the condition $n_1 \pm m_1 = n_2 \pm m_2$ by a Fourier integral $\int_0^1 e((n_1 \pm m_1 -( n_2 \pm m_2))\alpha) d\alpha$ and use Wilton's bound to conclude that
$$\sum_{s} |G(s)|^2  \preccurlyeq M'\frac{Q^2T^2R}{M^2}.$$
Now the scene has been prepared for the endgame with the
spectral large sieve of Deshouillers--Iwaniec \cite{DI}.
Using this and recalling that $Q$ is very large,
we deduce that $\Delta \sqrt{\Sigma_1 \Sigma_2}$ is
\begin{multline*}
  %\begin{split}
    \preccurlyeq 
    % \Delta 
    \frac{H}{TC^{1/2} }  \frac{M}{Q} \left(\frac{M'M}{
      Q^2}\right)^{-|\Im t_g|} \frac{M^{3/2}}{TR^{1/2}}
    \frac{1}{Q^2} \frac{T}{M^{1/2}} \frac{1}{TZ}\\
    \times\left( \Big( \frac{T^4 C^2}{M^2} + R\Big)R\right)^{1/2} \left( \frac{Q^2T^2R}{M^2} \cdot  \frac{Q^2T^2R}{M^2} \cdot M' \right)^{1/2}. 
  %\end{split}
\end{multline*}
%\pn{very very minor additions here}
Since $|\Im t_g| < 1/2$,
this expression is increasing in $M'$, so we can replace $M'$
with its largest value $Q^2/M$ (up to $T^{\eps}$), cf.\
\eqref{m-sum}, so that we can drop the term $ (M'M/Q^2)^{-|\Im
  t_g|}$.  Simplifying and using
\eqref{z},
\eqref{R}, \eqref{C}
and
\eqref{MH},
we obtain the final upper bound 
\begin{displaymath}
  \begin{split}
    % &  \frac{T}{KC^{1/2} }  \frac{M}{Q} \frac{M^{3/2}}{R^{1/2} Q T} \frac{1}{Q}  \frac{T}{M^{1/2}} \frac{1}{TZ}\Big( \Big( \frac{T^4 C^2}{M^2} + R\Big)R\Big)^{1/2} \Big( \frac{Q^2T^2R}{M^2} \cdot  \frac{Q^2T^2R}{M^2} \cdot \frac{Q^2 }{M} \Big)^{1/2}\\
    & \preccurlyeq  \frac{HM^{3}}{Q^3 R^{1/2}C^{3/2}T^3} \cdot \frac{T^2C}{M} R^{1/2} \cdot \frac{Q^3 T^2 R}{M^{5/2}} \\&= \frac{HT R}{(MC)^{1/2}} = \frac{H T^3 C^{3/2}}{M^{3/2}} \preccurlyeq \frac{T^{3/2}}{H^{1/2}} \preccurlyeq TH. 
  \end{split}
\end{displaymath}
This finishes the proof of  Theorem \ref{prop3}.

\section{Proof of the main results}\label{final}

We deduce Theorems \ref{thm1} and \ref{thm2}
by applying the combination of
Theorems \ref{lower-bound}, \ref{thm:triple-whittaker} and
\ref{prop3} to the
triple product formula \eqref{period}.
Recall the set-up and the choice of test vectors from the
beginning of \S \ref{whittaker}.
Ichino's formula \cite{Ic}
says that the identity
\eqref{period}
holds with $\mathcal{L}_{\infty}$
a constant multiple of the matrix coefficent integral
as
in
Theorem \ref{lower-bound}.
Therefore
% Combining Ichino's formula \eqref{period} with Theorem \ref{lower-bound}, we obtain
$$\frac{L(1/2, \pi_1\otimes \pi_2 \otimes \pi_3)}{L(1, \Ad^2 \pi_1)L(1, \Ad^2 \pi_2)
L(1, \Ad^2 \pi_3)} \ll Q \Big|\int_{g\in\Gamma\backslash G} v_1v_2v_3(g) dg\Big|^2$$
with $\Gamma = {\rm PGL}_2(\Bbb{Z})$. 
 We now appeal to the following regularized version of
\eqref{zag}, cf.\ \cite[(2.10)]{Ju}
or
\cite[Thm 5.6]{MR3356036}.
%\pn{added a self-reference
%  which, I think, gives a more compact
%  summary of what's needed}
Recall from \S\ref{sec:basic-notation}
our notation for Iwasawa coordinates
on $G$.
For a $\Gamma$-invariant function $\Phi$ on $G$
of rapid decay near the cusps,
%(e.g., $\Phi(n(x) a(y) k(\theta)) \ll y^{-\eps}$
%as $y \rightarrow \infty$),
we have 
\begin{multline*}
\int_{g\in \Gamma\backslash G} \Phi(g) dg = 2\int_{\Re(s)=a}
\pi^{-s} \Gamma(s)\zeta(2s) (2s-1)\\ \times \Big(
\int_{\Gamma_N\backslash G^+} \Phi(n(x) a(y) k(\theta)) y^{s} dx \frac{dy}{y^2}   d\theta  \Big)\frac{ds}{2\pi i}
\end{multline*}
where
$\Gamma_N := \Gamma \cap N$
denotes the upper-triangular unipotent subgroup of $\Gamma$,
$G^+$ denote the positive-determinant subgroup
of $G$, 
and the parameter
$a > 1$
is at our disposal. Note that the $s$-integral is rapidly
convergent
due to the decay of $\Gamma(s)$.
We apply this formula with $\Phi = v_1v_2v_3$,  insert the Fourier expansion and integrate over $x$. This gives
\begin{displaymath}
\begin{split}
&\int_{\Gamma_N\backslash G} v_1v_2v_3(n(x) a(y) k(\theta)) y^{s} dx \frac{dy}{y^2}   d\theta  \\
& = \sum_{\substack{n_1, n_2, n_3 \not= 0\\ n_1 + n_2 + n_3 = 0}} \frac{\lambda_{\pi_1}(n_1)\lambda_{\pi_2}(n_2)\lambda_{\pi_3}(n_3)}{\sqrt{n_1n_2n_3}} \int_{0}^{\infty} F( n_1 y,  n_2y ) y^{s} \frac{dy}{y^2}
\end{split}
\end{displaymath}
{with $F$ as in \eqref{defF}.} 
  Shifting the $s$-contour to the far right, we can restrict the
  $y$-integral to $y \gg Q^{-\eps}$, the remaining error being
  $O(Q^{-N})$. On the other hand, since $n_1 \in
  \Bbb{Z}\setminus \{0\}$, the upper bound for $F$ in Theorem
  \ref{thm:triple-whittaker} implies that we can also restrict
  to $y \ll Q^{\eps}$ and $n_1 \preccurlyeq 1$ at the cost of a
  negligible error. We insert the asymptotic formula from
  Theorem \ref{thm:triple-whittaker}. The error terms
  $\mathcal{E}_2, \mathcal{E}_3$ 
  contribute negligibly,
  while $\mathcal{E}_1$ contributes $\preccurlyeq Q^{-1/2}$.
  It remains to consider the contribution of $\mathcal{N}$.
  We smoothly decompose
  the sum over $n_2$ into dyadic ranges
  $n_2 \asymp M \preccurlyeq Q$. 
  %\pn{added, although it's only mildly convenient}
  We focus on the contribution from $M > 0$;
  the case $M < 0$ may be treated similarly.
  We estimate the
  contribution from $M \preccurlyeq Q^{1/3}$ trivially (using Cauchy--Schwarz and standard Rankin--Selberg bounds) by
  $$\sum_{n_1 \preccurlyeq 1} |\lambda_{\pi_1}(n_1)| \sum_{n_2 \asymp M} \frac{|\lambda_{\pi_2}(n_2) \lambda_{\pi_3}(n_1+n_2)|}{|n_2|} \Big( \frac{|n_2|}{Q}\Big)^{3/4} \preccurlyeq \frac{ M^{3/4}}{ Q^{3/4}} \preccurlyeq Q^{-1/2}.$$
  For $M\geq Q^{1/3 + \eps}$ we insert the full asymptotic
  formula for $\mathcal{N}$,
  giving
  \begin{displaymath}
 \begin{split}
& \frac{L(1/2, \pi_1\otimes \pi_2 \otimes \pi_3)}{L(1, \Ad^2 \pi_1)L(1, \Ad^2 \pi_2)
L(1, \Ad^2 \pi_3)} 
\preccurlyeq 1 + \\
&Q \sup_{
    Q^{1/3+\eps} \leq M  \preccurlyeq Q
}
\sum_{\nu \preccurlyeq 1} \Big|
% \sum_{m\not= 0, -\nu}
\sum_{m}
V\Big(\frac{m}{M}\Big) \frac{\lambda_{\pi_2}(m) \lambda_{\pi_3}(m+\nu)}{\sqrt{|m(m + \nu)|} } \Big(\frac{|m|}{Q}\Big)^{3/4} e\Big(\pm 2 \sqrt{Q} \Psi\Big(\frac{\nu}{m}\Big)\Big)\Big|^2
\end{split}
\end{displaymath}
%\pn{tweaked}
for some nice function $V$
(cf. \S\ref{weight}).
We implicitly restrict the sum over $m$
to the support of $V(m/M)$;
in particular,
$m \neq 0, -\nu$.
In the given range of $M$,
we can replace $\Psi(y)$ with $|y|^{1/2}$,
the error being flat.
By the usual procedure
(\S\ref{weight})
of separating
variables
and changing the weight  function, we arrive at the upper bound
$$1 +   \sup_{Q^{1/3+\eps} \leq M  \leq Q}  \frac{1}{(MQ)^{1/2}}   \sum_{\nu \preccurlyeq 1} \Big| \sum_{m}V\Big(\frac{m}{M}\Big)  \lambda_{\pi_2}(m) \lambda_{\pi_3}(m+\nu)  e\big(\pm 2 \sqrt{Q|\nu/m|}\big)\Big|^2.
$$
For $\pi_1, \pi_2$ fixed, we average this over $\pi_3$ in a spectral window $T \leq \sqrt{Q} \leq T+H$ with $H = T^{1/3+\eps}$. From Theorem \ref{prop3} %and the standard bounds \vb{can we clarify the various $L(1,\Ad^2 \pi)$?} 
%$$C(\pi)^{-\eps} \ll  L(1, \Ad^2 \pi) \ll C(\pi)^{\eps} $$
we obtain the first bounds in Theorems \ref{thm1} and \ref{thm2}. The second bound in Theorem \ref{thm1} follows directly by dropping all but one term (using positivity of central triple product values), while the second bound in Theorem \ref{thm2} follows from a standard argument based on the functional equation (see e.g.\ \cite[p.\ 63]{Go}).

\end{document}